\documentclass[12pt]{amsart}
\usepackage{amssymb,amsmath,amsthm, amscd, enumerate, mathrsfs}
\usepackage{graphicx, hhline}
\usepackage[all]{xy}
\usepackage{braket}
\usepackage[usenames]{color}
\usepackage{hyperref}
\usepackage{fancyhdr}
\usepackage[top=30truemm,bottom=30truemm,left=30truemm,right=30truemm]{geometry}
\usepackage[pagewise]{lineno}

\hypersetup{colorlinks=true}

\title[Complex analytic generalized pairs]{On generalized canonical bundle formula and boundedness of complements in complex analytic setting}
\author{Kenta Hashizume}
\date{2026/09/05}
\address{Department of 
Mathematics, Faculty of Science, Niigata University, Niigata 950-2181, Japan}
\email{hkenta@math.sc.niigata-u.ac.jp}

\newtheorem{thm}{Theorem}[section]

\newtheorem{lem}[thm]{Lemma}
\newtheorem{cor}[thm]{Corollary}
\newtheorem{prop}[thm]{Proposition}

\theoremstyle{definition}
\newtheorem{defn}[thm]{Definition}
\newtheorem{rem}[thm]{Remark}

\newtheorem*{ack}{Acknowledgments} 
 
\newtheorem*{b-divisor}{b-divisors}

\newtheorem*{quasi-log}{Quasi-log complex analytic space} 
 
\newtheorem*{g-pair}{Generalized pairs} 
\newtheorem*{adj-g-pair}{Divisorial adjunction for generalized pairs} 
\newtheorem*{mmp-g-pair}{MMP for generalized pairs}

\newtheorem{step1}{Step}
\newtheorem{step2}{Step}
\newtheorem{step3}{Step}

\newtheorem*{claim*}{Claim}

\begin{document}

\begin{abstract}
We establish the generalized canonical bundle formula for generalized lc-trivial fibrations with irrational coefficients over non-compact bases in the complex analytic setting, and we show that the discriminant b-divisor and moduli b-divisor are compatible with restriction to arbitrary open subsets. We also discuss the boundedness of complements in this setting. 
\end{abstract}

\maketitle

\tableofcontents

\section{Introduction}\label{sec--intro}

In this paper, we establish the generalized canonical bundle formula for generalized lc-trivial fibrations over non-compact bases and prove the boundedness of complements in the complex analytic setting. 

We work with the notion of generalized pairs, introduced by Birkar--Zhang \cite{bz}. 
Generalized pairs have played a crucial role in Birkar's work on the boundedness of complements \cite{birkar-compl} and the boundedness of Fano varieties of fixed dimension with mild singularities \cite{birkar-bab}. 
These developments have led to significant progress in birational geometry, and the language of generalized pairs is now standard in this area. 
In the complex analytic setting, there are two extensions of this notion depending on the framework adopted for the nef part. 
One is the notion of generalized pairs whose nef part is given by b-nef $\mathbb{R}$-b-divisors. 
Another is the notion of generalized pairs whose nef part is given by nef closed $(1,1)$-currents. 
The former is very natural in the context of projective morphisms between complex analytic spaces. 
The latter, on the other hand, is useful in the context of compact but non-projective K\"{a}hler varieties because of the lack of ample divisors.   

We adopt the former definition in this paper. 
The canonical bundle formula plays a central role in birational geometry, especially in inductive arguments on the dimension of varieties.
The canonical bundle formula for klt-trivial fibrations was studied by Kawamata \cite{kawamata-subadjunction-II} and Ambro \cite{ambro1}, \cite{ambro2}, and
Fujino--Gongyo \cite{fg-lctrivial} established the canonical bundle formula for lc-trivial fibrations. 
Their theorems were extended to the framework of generalized pairs by Filipazzi \cite{filipazzi-gen-can-bundle-formula}. 
For recent developments on the b-semiampleness conjecture, see Bakker--Filipazzi--Mauri--Tsimerman \cite{bfmt}. 
For the canonical bundle formula in the K\"{a}hler context, see Hacon--P\u{a}un \cite{haconpaun}. 

We extend Filipazzi's theorem to generalized lc-trivial fibrations in Fujino's setting of complex analytic varieties \cite{fujino-analytic-bchm}.

\begin{thm}
\label{thm--gen-can-bundle-formula-full-intro}
Let $\pi \colon (X,B+M) \to S$ be a generalized sub-pair with the nef part $\boldsymbol{\rm M}$, and let $f\colon (X,B+M) \to Z$ be a generalized lc-trivial fibration over $S$. 
Let $W \subset S$ be a Stein compact subset. 
Suppose that there exists a Zariski open dense subset $U \subset Z$ such that $B|_{f^{-1}(U)}$ is effective.
Then, after replacing $S$ with a suitable open neighborhood of $W$, we have the following properties.
\begin{enumerate}[(i)]
\item \label{thm--gen-can-bundle-formula-full-intro-(i)}
Let $\boldsymbol{\rm G}$ (resp.~$\boldsymbol{\rm N}$) be the discriminant $\mathbb{R}$-b-divisor (resp.~the moduli $\mathbb{R}$-b-divisor) associated to $f$. 
Then $\boldsymbol{\rm N}$ is a b-nef$/S$ $\mathbb{R}$-b-Cartier $\mathbb{R}$-b-divisor on $Z$. 
In particular, $(Z, \boldsymbol{\rm G}_{Z} +\boldsymbol{\rm N}_{Z}) \to S$ is a generalized sub-pair with the nef part $\boldsymbol{\rm N}$. 
Furthermore, if $(X,B+M) \to S$ is generalized lc (resp.~generalized klt), then $(Z, \boldsymbol{\rm G}_{Z} +\boldsymbol{\rm N}_{Z}) \to S$ is also generalized lc (resp.~generalized klt). 
\item \label{thm--gen-can-bundle-formula-full-intro-(ii)}
For any open subset $\tilde{S} \subset S$ that does not necessarily contain $W$, if we put $\tilde{X} \subset X$ and $\tilde{Z} \subset Z$ as the inverse images of $\tilde{S}$, then the generalized lc-trivial fibration $(\tilde{X}, B|_{\tilde{X}}+M|_{\tilde{X}}) \to \tilde{Z}$ over $\tilde{S}$ satisfies (\ref{thm--gen-can-bundle-formula-full-intro-(i)}). 
\end{enumerate}
In particular, the discriminant $\mathbb{R}$-b-divisor and the moduli $\mathbb{R}$-b-divisor associated to $f$ are compatible with the restriction over any open subset (Lemma \ref{lem--restriction-can-bundle-formula}). 
Moreover, if $\boldsymbol{\rm M}$ is a finite $\mathbb{R}_{>0}$-linear combination of b-nef$/S$ $\mathbb{Q}$-b-Cartier $\mathbb{Q}$-b-divisors, then so is $\boldsymbol{\rm N}$. 
\end{thm}

We explain the differences between Theorem \ref{thm--gen-can-bundle-formula-full-intro} and \cite[Theorem 2.3]{haconpaun}. 
The main difference lies in the compactness assumptions in generalized lc-trivial fibrations. 
Theorem \ref{thm--gen-can-bundle-formula-full-intro} applies to generalized lc-trivial fibrations over non-compact bases.
Hence, the nef part of the generalized sub-pair is not necessarily globally nef. 
In particular, if $S=Z$ in Theorem \ref{thm--gen-can-bundle-formula-full-intro}, then the nef part is only fiberwise semi-positive. 
Moreover, Theorem \ref{thm--gen-can-bundle-formula-full-intro} allows divisors with irrational coefficients. 

The compatibility of the discriminant $\mathbb{R}$-b-divisor and the moduli $\mathbb{R}$-b-divisor under restriction is non-trivial and useful in applications. 
This property does not appear in the literature even for klt-trivial fibrations. 
The starting point for the proof of Theorem \ref{thm--gen-can-bundle-formula-full-intro} is \cite[Theorem 21.4]{fujino-analytic-bchm}. 
After improving \cite[Theorem 21.4]{fujino-analytic-bchm}, we prove Theorem \ref{thm--gen-can-bundle-formula-full-intro} for generalized klt-trivial fibrations by using Filipazzi's idea \cite{filipazzi-gen-can-bundle-formula}.
We then extend the framework to generalized lc-trivial fibrations. In the proof, weak semistable reduction and the minimal model program (MMP, for short) in the complex analytic setting play important roles. 
We refer to \cite{eh-semistablereduction} and \cite{eh-analytic-mmp} for these key ingredients, respectively. We emphasize that we do not rely on any result from \cite{bfmt} or variations of mixed Hodge structures (cf.~\cite[Theorem 3.6]{fg-lctrivial}) in the proof of Theorem \ref{thm--gen-can-bundle-formula-full-intro}. 
We also remark that we do not use the foliation theory as in \cite{chlx}.  
To the best of our knowledge, this strategy of proof does not appear even in the algebraic setting. 

The first and most fundamental difficulty goes back to the definition of b-divisors. 
In the complex analytic setting, it is not clear whether there is a reasonable definition of the restriction of b-divisors to open subsets. 
This is because the number of bimeromorphic morphisms increases under restriction, and we do not have a compactification of open analytic varieties. 
Hence, it is not clear whether the discriminant $\mathbb{R}$-b-divisor and the moduli $\mathbb{R}$-b-divisor associated to generalized lc-trivial fibrations are compatible with restriction over the base. 
However, to run a relative MMP or take a weak semistable reduction, restriction over the base is unavoidable. 
Therefore, the first step toward Theorem \ref{thm--gen-can-bundle-formula-full-intro} is to prove a suitable compatibility of the discriminant $\mathbb{R}$-b-divisor and the moduli $\mathbb{R}$-b-divisor with respect to restriction. 
This is achieved in Lemma \ref{lem--restriction-can-bundle-formula}. 
The second difficulty is the $\mathbb{R}$-b-Cartier property of the moduli $\mathbb{R}$-b-divisor. 
Note that \cite[Theorem 21.4]{fujino-analytic-bchm} does not explicitly assert the $\mathbb{Q}$-b-Cartier property of the moduli $\mathbb{Q}$-b-divisor. 
This is treated in Remark \ref{rem--descend-variety} and Theorem \ref{thm--klt-trivial-fib-start}, which is Theorem \ref{thm--gen-can-bundle-formula-full-intro} in the case of klt-trivial fibrations. 
After extending Theorem \ref{thm--klt-trivial-fib-start} to the case of generalized klt-trivial fibrations using Filipazzi's idea \cite{filipazzi-gen-can-bundle-formula}, we encounter the final difficulty: the extension to the generalized lc case. 
We overcome this difficulty by using the techniques of the MMP \cite{eh-analytic-mmp} together with Theorem \ref{thm--gen-can-bundle-formula-full-intro} in the case of generalized klt-trivial fibrations. 
For details, see the proof of Proposition \ref{prop--gen-can-bundle-formula-lc-Q-div}. 

By combining the proof of Theorem \ref{thm--gen-can-bundle-formula-full-intro} and \cite[Theorem 7.3]{bfmt}, we obtain the following result (see \cite[7.2.3]{bfmt}). 
We note again that the proof of Theorem \ref{thm--gen-can-bundle-formula-full-intro} does not rely on any result from \cite{bfmt}.

\begin{thm}[= Theorem \ref{thm--lc-trivial-bfmt}]\label{thm--lc-trivial-bfmt-intro}
Let $f \colon (X,\Delta) \to Z$ be an lc-trivial fibration over a complex analytic space $S$, and let $W \subset S$ be a Stein compact subset. 
Suppose that $\Delta|_{f^{-1}(U)}$ is effective for some Zariski open dense subset $U \subset Z$.
Then, after shrinking $S$ around $W$, there exists a projective bimeromorphic morphism $Z' \to Z$ such that for any open subset $\tilde{S} \subset S$ that does not necessarily contain $W$ (possibly $\tilde{S}=S$), if we put $\tilde{X} \subset X$, $\tilde{Z} \subset Z$, and $\tilde{Z}' \subset Z'$ as the inverse images of $\tilde{S}$, then the moduli $\mathbb{R}$-b-divisor $\tilde{\boldsymbol{\rm N}}$ associated to $(\tilde{X},\Delta|_{\tilde{X}}) \to \tilde{Z}$ descends to $\tilde{Z}'$ and $\tilde{\boldsymbol{\rm N}}_{\tilde{Z}'}$ is semi-ample over $\tilde{S}$. 
Moreover, we have $\boldsymbol{\rm N}|_{\tilde{Z}}=\tilde{\boldsymbol{\rm N}}$, where  $\boldsymbol{\rm N}$ is the moduli $\mathbb{R}$-b-divisor associated to $f \colon (X,\Delta) \to Z$. 
\end{thm}

Once we establish the canonical bundle formula and the MMP \cite{eh-analytic-mmp-2}, we expect that many results for projective varieties can be extended to projective morphisms between complex analytic varieties. Although we do not explore such applications in depth in this paper, we prove the boundedness of complements in the complex analytic setting. 
Roughly speaking, complements provide a way to measure singularities of generalized pairs.
The boundedness of complements was originally conjectured by Shokurov \cite{shokurov-flip} for the existence of threefold log flips. 
In addition to Birkar's breakthrough mentioned before, the boundedness of complements is applied to the ascending chain condition for minimal log discrepancies (\cite{hanliushokurov-acc-mld}, \cite{hanliuluo-acc-mld-3fold}), the moduli problem of Calabi--Yau pairs (\cite{abb+23}), and the Yau--Tian--Donaldson conjecture  and the projectivity of the K-moduli of log Fano pairs (\cite{liuxuzhuang-Kstability}). 
In this paper, we show the complex analytic analogs of results by Filipazzi--Moraga \cite{filipazzi-moraga} and Chen--Han--He--Xie \cite{chhx-compl}.

\begin{thm}[= Theorem \ref{thm--generalized-compl-main-1}]\label{thm--generalized-compl-main-1-intro}
Let $d$ and $p$ be positive integers and let $\Lambda  \subset [0,1] \cap \mathbb{Q}$ be a DCC set whose accumulation points are rational numbers. 
Then there exists $n \in \mathbb{Z}_{>0}$, depending only on $d$, $p$, and $\Lambda$, satisfying the following. 
Let $\pi \colon(X,B+M) \to Z$ be a generalized lc pair with the nef part $\boldsymbol{\rm M}$ such that 
\begin{itemize}
\item
$\pi \colon X \to Z$ is a contraction of normal analytic varieties  and ${\rm dim}\,X=d$, 
\item
$B \in \Lambda$ and $p \boldsymbol{\rm M}$ is b-Cartier, 
\item
$X$ is of Fano type over $Z$, and 
\item
$-(K_{X}+B+M)$ is nef over $Z$.
\end{itemize} 
Then, for any point $z \in Z$, after shrinking $Z$ around $z$ suitably, $(X,B+M)\to Z$ has an $n$-complement $(X,B^{+}+M)\to Z$ over $z$ such that $B^{+} \geq B$. 
\end{thm}

\begin{thm}[= Theorem \ref{thm--generalized-compl-main-2}]\label{thm--generalized-compl-main-2-intro}
Let $d$ and $p$ be positive integers. 
Then there exists a finite set $\mathcal{N}$ of positive integers, depending only on $d$ and $p$, satisfying the following.
Let $\pi \colon(X,B+M) \to Z$ be a generalized pair with the nef part $\boldsymbol{\rm M}$ and let $z \in Z$ be a point such that 
\begin{itemize}
\item
$\pi \colon X \to Z$ is a contraction of normal analytic varieties and ${\rm dim}\,X=d$, 
\item
$p \boldsymbol{\rm M}$ is b-Cartier, 
\item
$X$ is of Fano type over $Z$, and 
\item
after shrinking $Z$ around $z$, there exists an effective $\mathbb{R}$-divisor $\Delta$ on $X$ such that $(X,(B+\Delta)+M) \to Z$ is generalized klt and $K_{X}+B+\Delta+M \sim_{\mathbb{R}} 0$.
\end{itemize} 
Then, after shrinking $Z$ around $z$ suitably, $(X,B+M) \to Z$ has an $n$-complement over $z \in Z$ for some $n \in \mathcal{N}$. 
\end{thm}

One interesting case is $\boldsymbol{\rm M}=0$ and $X=Z$ in these theorems.

\begin{cor}[cf.~{\cite[Corollary 1.9]{birkar-compl}}]\label{cor--generalized-compl-main-1-intro}
Let $d$ be a positive integer and $\Lambda  \subset [0,1] \cap \mathbb{Q}$ a DCC set whose accumulation points are rational numbers. 
Then there exists $n \in \mathbb{Z}_{>0}$, depending only on $d$ and $\Lambda$, satisfying the following. 
Let $(X,B)$ be a $d$-dimensional lc pair such that the coefficients of $B$ belong to $\Lambda$ and $(X,\Delta)$ is klt for some $\mathbb{R}$-divisor $\Delta$ on $X$. 
Let $x \in X$ be a point. 
Then, after shrinking $X$ around $x$, there exists $B^{+} \geq B$ such that $(X,B^{+})$ is lc and $n(K_{X}+B^{+})$ is Cartier. 
In particular, if $x$ is a non-klt center of $(X,B)$, then $n(K_{X}+B)$ is Cartier. 
\end{cor}

For a related result, see Corollary \ref{cor--compl-fujinoconj}. 

\begin{cor}\label{cor--generalized-compl-main-2-intro}
Let $d$ be a positive integer. 
Then there exists a finite set $\mathcal{N}$ of positive integers, depending only on $d$, satisfying the following. 
Let $(X,B)$ be a $d$-dimensional klt pair, and let $x \in X$ be a point. 
Then, after shrinking $X$ around $x$, there exists an effective $\mathbb{Q}$-divisor $B^{+}$ on $X$ such that $(X,B^{+})$ is lc, $nB^{+} \geq n\lfloor B \rfloor+\lfloor (n+1)\{B\} \rfloor$, and $n(K_{X}+B^{+})$ is Cartier for some $n \in \mathcal{N}$. 
\end{cor}

As shown in \cite[Chapter 11, Example 11]{shokurov-compl}, we cannot extend Theorem \ref{thm--generalized-compl-main-2-intro} to the lc case. 

The contents of this paper are as follows: 
In Section \ref{sec--preliminaries}, we collect definitions and basic properties of generalized pairs. 
In Section \ref{sec--gen-can-bundle-formula}, we prove Theorem \ref{thm--gen-can-bundle-formula-full-intro}. 
In Section \ref{sec--compl}, we study the boundedness of complements. 

\begin{ack}
The author was partially supported by JSPS KAKENHI Grant Number JP23K20787. 
The author would like to express his deep gratitude to Professor Osamu Fujino for answering questions, carefully reading an earlier draft, and giving valuable advice. The author is grateful to Professors Benjamin Bakker, Guodu Chen, Makoto Enokizono, Stefano Filipazzi, Christopher D. Hacon, Jingjun Han, Yang He, Mirko Mauri, Mihai P\u{a}un, Jacob Tsimerman, and Lingyao Xie for answering questions.
\end{ack}

\section{Preliminaries}\label{sec--preliminaries}

Throughout this paper, complex analytic spaces are always assumed to be Hausdorff and second countable. 
{\em Analytic varieties} are reduced and irreducible complex analytic spaces. 
For notations and definitions used in this paper, see \cite{fujino-analytic-bchm}, \cite{eh-analytic-mmp}, and  \cite{eh-semistablereduction}. 

\subsection{Divisors and morphisms}
In this subsection we collect definitions and notations of divisors and morphisms. 

\begin{itemize}
\item
Let $\pi \colon X \to Z$ be a projective morphism from a normal analytic variety to a complex analytic space. 
We will freely use the definitions of $\pi$-nef $\mathbb{R}$-divisor, $\pi$-ample $\mathbb{R}$-divisor, $\pi$-semi-ample $\mathbb{R}$-Cartier divisor, $\pi$-base point free $\mathbb{R}$-Cartier divisor, and $\pi$-big $\mathbb{R}$-divisor as in \cite{fujino-analytic-bchm}. 
The real vector space spanned by the prime divisors on $X$ is denoted by ${\rm WDiv}_{\mathbb{R}}(X)$ (\cite[Definition 2.34]{fujino-analytic-bchm}). 

\item 
A {\em contraction} $f\colon X\to Y$ is a projective morphism of analytic varieties such that $f_{*}\mathcal{O}_{X}\cong \mathcal{O}_{Y}$. 
For an analytic variety $X$ and an $\mathbb{R}$-divisor $D$ on $X$, a {\em log resolution of} $(X,D)$ is a proper bimeromorphic morphism $g \colon V\to X$ from a smooth analytic variety $V$ such that $g$ is projective over a neighborhood of any compact subset of $X$, the exceptional locus ${\rm Ex}(g)$ of $g$ is of pure codimension one, and ${\rm Ex}(g)\cup {\rm Supp}\,g_{*}^{-1}D$ is a simple normal crossing divisor. 
We note that a log resolution $g \colon V\to X$ of $(X,D)$ always exists (\cite[Theorem 13.2]{resolution-1}). 

\item
We refer to \cite[Definition 2.30]{fujino-analytic-bchm} for {\em meromorphic map} and {\em bimeromorphic map}. 
We say that a bimeromorphic map $\phi\colon X \dashrightarrow X'$ of analytic varieties is a {\em bimeromorphic contraction} if $\phi^{-1}$ does not contract any divisor. 
\end{itemize}

\begin{defn}[Property (P), see {\cite{fujino-analytic-bchm}}]\label{defn--property(P)}
Let $\pi \colon X \to Y$ be a projective morphism of complex analytic spaces, and let $W \subset Y$ be a compact subset. 
In this paper, we will use the following conditions:
\begin{itemize}
\item[(P1)]
$X$ is a normal analytic variety,
\item[(P2)]
$Y$ is a Stein space,
\item[(P3)]
$W$ is a Stein compact subset of $Y$, and
\item[(P4)]
$W \cap Z$ has only finitely many connected components for any analytic subset $Z$
which is defined over an open neighborhood of $W$. 
\end{itemize}
We say that {\em $\pi \colon X \to Y$ and $W \subset Y$ satisfy (P)} if the conditions (P1)--(P4) hold. 
\end{defn}

\begin{defn}
Let $a$ be a real number. 
We define the {\em round up} $\lceil a \rceil$ to be the smallest integer not less than $a$, and we define the {\em round down} $\lfloor a \rfloor$ to be the largest integer not greater than $a$. 
It is easy to see that $\lfloor a \rfloor=- \lceil -a \rceil$. 

Let $D$ be an $\mathbb{R}$-divisor on an analytic variety and $D=\sum_{i}d_{i}D_{i}$ the decomposition of $D$ into distinct prime divisors. 
We define
\begin{equation*}
\begin{split}
\lceil D\rceil:=\sum_{i}\lceil d_{i} \rceil D_{i}, \qquad \lfloor D \rfloor:= \sum_{i}\lfloor d_{i} \rfloor D_{i}, \qquad {\rm and} \qquad \{D\}:=D-\lfloor D \rfloor.
\end{split}
\end{equation*}
We also define
\begin{equation*}
D^{<0} :=\sum_{d_{i}<0}d_{i}D_{i}, \quad 
D^{<1} :=\sum_{d_{i}<1}d_{i}D_{i}, \quad 
D^{= 1}:=\sum_{d_{i}= 1} D_{i}, \quad 
{\rm and} \quad 
D^{\geq 1}:=\sum_{d_{i}\geq 1}d_{i} D_{i}. 
\end{equation*}
By definition, we have $\lfloor D \rfloor = -\lceil -D \rceil$. 
\end{defn}

\begin{defn}[{\cite[Definition 2.2]{chhx-compl}}]\label{defn--hyperstandardset}
Let $\mathfrak{R}\subset \mathbb{R}$ be a subset. 
Throughout this paper, $\Phi(\mathfrak{R})$ is defined by
$$\Phi(\mathfrak{R})=\left\{1-\frac{r}{l}\;\middle| \;r\in \mathfrak{R}, \, l\in \mathbb{Z}_{>0}\right\},$$
and we call it a {\em hyperstandard set associated to $\mathfrak{R}$}. 

Let $\mathcal{N} \subset \mathbb{Z}_{>0}$ be a (possibly empty) finite subset and $\Phi:=\Phi(\mathfrak{R})$ a hyperstandard set associated to $\mathfrak{R} \subset \mathbb{R}\cap[0,1]$. 
If $\mathcal{N}$ is not empty, then we define
$$\Gamma(\mathcal{N}, \Phi):=\left\{1-\frac{r}{l}+\frac{1}{l}\sum_{n \in \mathcal{N}}\frac{m_{n}}{n+1}\;\middle| \;r\in \mathfrak{R}, \, l\in \mathbb{Z}_{>0}, \, m_{n} \in \mathbb{Z}_{\geq 0} \right\}\cap[0,1],$$
and otherwise we define $\Gamma(\mathcal{N}, \Phi):=\Phi$. 
When $\mathcal{N}=\{n\}$ for some positive integer $n$, we denote $\Gamma(\{n\}, \Phi)$ by $\Gamma(n, \Phi)$ for simplicity. 
\end{defn}

\begin{defn}
Let $\mathfrak{R} \subset \mathbb{R}$ be a subset and $D$ an $\mathbb{R}$-divisor on an analytic variety. 
Then $D \in \mathfrak{R}$ denotes the condition that every coefficient of $D$ belongs to $\mathfrak{R}$. 
\end{defn}

We define the notion of Fano type for contractions. 
The definition of klt pairs used below is the usual one. 
For the definition of singularities of pairs, see \cite{fujino-analytic-bchm} and  \cite{eh-analytic-mmp}.  

\begin{defn}\label{defn--Fano-type}
Let $\pi \colon X \to Z$ be a projective morphism from a normal analytic variety $X$ to a complex analytic space $Z$. 
We say that $X$ is {\em of Fano type over $Z$} if $\pi$ is a contraction and there exists an $\mathbb{R}$-divisor $\Delta$ on $X$ such that $(X,\Delta)$ is a klt pair and $-(K_{X}+\Delta)$ is $\pi$-ample. 
We note that if $X$ is of Fano type over $Z$, then $\pi^{-1}(U)$ is of Fano type over $U$ for any open subset $U \subset Z$. 
\end{defn}

\begin{lem}\label{lem--Fano-type-bimero-cont}
Let $\pi \colon X \to Z$ be a contraction such that $X$ is of Fano type over $Z$, and let $f\colon X \dashrightarrow X'$ be a  bimeromorphic contraction over $Z$, where $X'$ is projective over $Z$. 
Let $W \subset Z$ be a Stein compact subset. 
Then $X'$ is of Fano type over an open neighborhood of $W$.
\end{lem}

\begin{proof}
Pick $\Delta$ such that $(X,\Delta)$ is klt and $A:=-(K_{X}+\Delta)$ is ample over $Z$. 
By shrinking $Z$ around $W$, we can find a common resolution $g \colon Y \to X$ and $g' \colon Y \to X'$ of $f$ and an effective Cartier divisor $F$ on $Y$ such that ${\rm Supp}\,F={\rm Ex}(g)$ and $-F$ is $g$-ample. 
Let $H'$ be a $\mathbb{Q}$-Cartier divisor on $X'$ that is ample over $Z$. 
Shrinking $Z$ around $W$, we can find sufficiently small positive rational numbers $\epsilon$ and $\delta$ such that 
\begin{itemize}
\item
defining $\Delta_{Y}$ by $K_{Y}+\Delta_{Y}=g^{*}(K_{X}+\Delta)$, then $(Y,\Delta_{Y}+\epsilon F)$ is sub-klt, and 
\item
$-\epsilon F + g^{*}A-\delta g'^{*}H'$ is ample over $Z$. 
\end{itemize}
By Bertini's theorem and shrinking $Z$ around $W$, we can find an effective $\mathbb{R}$-divisor $G_{Y} \sim_{\mathbb{R},\,Z}-\epsilon F + g^{*}A-\delta g'^{*}H'$ on $Y$ such that $(Y,\Delta_{Y}+\epsilon F+G_{Y})$ is sub-klt. 
Replacing $H'$ if necessary, we may assume that $(Y,\Delta_{Y}+\epsilon F+G_{Y}+\delta g'^{*}H')$ is sub-klt. 
Now
\begin{equation*}
\begin{split}
K_{Y}+\Delta_{Y}+\epsilon F+G_{Y}+\delta g'^{*}H'\sim_{\mathbb{R},\,Z}K_{Y}+\Delta_{Y}+g^{*}A=g^{*}(K_{X}+\Delta+A) = 0.
\end{split}
\end{equation*}
Thus, putting $B':=g'_{*}(\Delta_{Y}+\epsilon F+G_{Y})=f_{*}\Delta+g'_{*}G_{Y}$, we have $K_{X'}+B'+\delta H' \sim_{\mathbb{R},\,Z}0$ and 
$$K_{Y}+\Delta_{Y}+\epsilon F+G_{Y}+\delta g'^{*}H'=g'^{*}(K_{X'}+B'+\delta H').$$ 
Therefore $(X',B'+\delta H')$ is klt. 
Then $(X',B')$ is klt, and $-(K_{X'}+B') \sim_{\mathbb{R},\,Z}\delta H'$ is ample over $Z$. 
This shows that $X'$ is of Fano type over $Z$. 
\end{proof}

The following result is useful since we do not need to shrink the base after focusing on a certain $\mathbb{R}$-Cartier divisor.

\begin{thm}\label{thm--mmp-Fano-type}
Let $\pi \colon X \to Z$ be a contraction such that $X$ is of Fano type over $Z$, and let $W \subset Z$ be a compact subset such that $\pi$ and $W$ satisfy (P). 
Then, there exists a Stein open neighborhood $\tilde{Z}$ of $W$ with the inverse image $\tilde{X}:=\pi^{-1}(\tilde{Z})$ satisfying the following property. 
For any $\mathbb{R}$-Cartier divisor $D$ on $X$ with $\tilde{D}:=D|_{\tilde{X}}$, there exists a bimeromorphic contraction $f \colon \tilde{X} \dashrightarrow \tilde{X}'$ over $\tilde{Z}$, where $\tilde{X}'$ is projective over $\tilde{Z}$, such that if we put $\tilde{D}':=f_{*}\tilde{D}$ then $\tilde{D}'$ is $\mathbb{R}$-Cartier and 
\begin{itemize}
\item
$\tilde{X}'$ is of Fano type over $\tilde{Z}$, 
\item
for any common resolution $g \colon V \to \tilde{X}$ and $g' \colon V \to \tilde{X}'$ of $f$, we have
 $$g^{*}\tilde{D}=g'^{*}\tilde{D}'+E$$ 
 for some effective $\mathbb{R}$-divisor $E$ on $V$ and ${\rm Supp}\,E$ contains the strict transform of every $f$-exceptional prime divisor on $\tilde{X}$, and
\item
$\tilde{D}'$ is semi-ample over $\tilde{Z}$ or there exists a contraction $\tilde{X}' \to Y$ over $\tilde{Z}$ such that ${\rm dim}\,\tilde{X}' > {\rm dim}\,Y$ and $-\tilde{D}'$ is ample over $Y$.  
\end{itemize}
\end{thm}

\begin{proof}
By \cite[Lemma 2.16]{fujino-analytic-bchm}, we get
$$W \subset \tilde{Z} \subset \tilde{W} \subset Z^{\circ} \subset W^{\circ} \subset  Z,$$
where $\tilde{Z}$ and $Z^{\circ}$ are Stein open subsets and $\tilde{W}$ and $W^{\circ}$ are Stein compact subsets that satisfy (P4) in Definition \ref{defn--property(P)}. 
We show that $\tilde{Z}$ is the desired Stein open neighborhood of $W$. 
For any $\mathbb{R}$-Cartier divisor $D$ on $X$, using \cite{fujino-analytic-bchm} we can find a Stein open subset $Z_{D} \subset Z$ containing $W^{\circ}$ and a finite sequence of steps of a $D$-MMP over $Z_{D}$
$$f_{D}\colon \pi^{-1}(Z_{D})=:X_{D} \dashrightarrow X'_{D}$$
such that $D':=f_{D*}(D|_{X_{D}})$ is nef over $W^{\circ}$ or $X'_{D}$ has the structure $X'_{D} \to Y_{D}$ of a Mori fiber space over $Z_{D}$ for $D'$. 
Then $\tilde{Z}$ satisfies the second condition of Theorem \ref{thm--mmp-Fano-type}, and $\tilde{Z}$ also satisfies the first condition of Theorem \ref{thm--mmp-Fano-type} since Lemma \ref{lem--Fano-type-bimero-cont} shows that $X'_{D}$ is of Fano type over an open neighborhood of $W^{\circ}$. 
Furthermore, the third condition of Theorem \ref{thm--mmp-Fano-type} is satisfied if there exists the structure $X'_{D} \to Y_{D}$ of a Mori fiber space over $Z_{D}$ for $D'$. 
Thus, we assume that $D'$ is nef over $W^{\circ}$. 
In this case, $D'$ is nef over $Z^{\circ}$, and Lemma \ref{lem--Fano-type-bimero-cont} shows that $X'_{D}$ is of Fano type over $Z^{\circ}$. 
By applying \cite[Theorem 8.1]{fujino-analytic-bchm} to $\tilde{W} \subset Z^{\circ}$, we see that $D'$ is semi-ample over a neighborhood of $\tilde{W}$. 
In particular, we see that if $D'$ is nef over $W^{\circ}$ then $D'$ is semi-ample over $\tilde{Z}$. 
Hence, the third condition of Theorem \ref{thm--mmp-Fano-type} is satisfied for this $\tilde{Z}$. 
We complete the proof. 
\end{proof}

\subsection{Generalized pairs}

In this subsection we define b-divisors and generalized pairs. 
After that, we introduce some basic results.  

\begin{defn}[b-divisor]\label{defn--b-divisor}
Let $X$ be a normal analytic variety. 
Consider all proper bimeromorphic morphisms $f \colon Y \to X$ from normal analytic varieties $Y$ and the sets ${\rm WDiv}_{\mathbb{R}}(Y)$ of $\mathbb{R}$-divisors on $Y$. 
For two proper bimeromorphic morphisms $Y_{1} \to X$ and $Y_{2} \to X$ such that the induced map $\phi \colon Y_{1} \dashrightarrow Y_{2}$ is a bimeromorphic contraction, we define the map
$$\phi_{*} \colon {\rm WDiv}_{\mathbb{R}}(Y_{1}) \longrightarrow {\rm WDiv}_{\mathbb{R}}(Y_{2})$$
by taking the strict transform. 
Then an {\em $\mathbb{R}$-b-divisor} $\boldsymbol{\rm D}$ on $X$ is an element of the inverse limit
$$\boldsymbol{\rm WDiv}_{\mathbb{R}}(X):=\underset{Y \to X}{\rm lim}{\rm WDiv}_{\mathbb{R}}(Y)$$
defined with all such $\phi_{*}$ as above. 
Similarly, a {\em $\mathbb{Q}$-b-divisor} on $X$ is an element of the inverse limit $\underset{Y \to X}{\rm lim}{\rm WDiv}_{\mathbb{Q}}(Y)$, where ${\rm WDiv}_{\mathbb{Q}}(Y)$  is the set of $\mathbb{Q}$-divisors on $Y$. 

For an $\mathbb{R}$-b-divisor $\boldsymbol{\rm D}$ on $X$ and a proper bimeromorphic morphism $Y \to X$ from a normal analytic variety $Y$, the image of $\boldsymbol{\rm D}$ by the projection $\boldsymbol{\rm WDiv}_{\mathbb{R}}(X) \to {\rm WDiv}_{\mathbb{R}}(Y)$ is called the {\em trace of $\boldsymbol{\rm D}$ on $Y$}, and it is denoted by $\boldsymbol{\rm D}_{Y}$. 

Let $\boldsymbol{\rm D}$ be an $\mathbb{R}$-b-divisor on $X$. 
We say that $\boldsymbol{\rm D}$ is {\em $\mathbb{R}$-b-Cartier} if there exist a proper bimeromorphic morphism $f \colon Y \to X$ from a normal analytic variety $Y$ and an $\mathbb{R}$-Cartier divisor $D$ on $Y$ satisfying the following: 
For any normal analytic variety $Y'$ together with a proper bimeromorphic morphism $f' \colon Y' \to X$, taking  proper bimeromorphic morphisms $g \colon V \to Y$ and $g' \colon V \to Y'$ from a normal analytic variety $V$ such that $f \circ g = f' \circ g'$, we have $\boldsymbol{\rm D}_{Y'}=g'_{*}g^{*}D$. 
In this situation, we say that {\em $\boldsymbol{\rm D}$ descends to $Y$}, we say that $\boldsymbol{\rm D}$ is the {\em Cartier closure of $D$}, and we use the notation $\overline{D}$ to denote $\boldsymbol{\rm D}$. 
When $\boldsymbol{\rm D}=\overline{D}$ and $D$ is $\mathbb{Q}$-Cartier, we say that $\boldsymbol{\rm D}$ is {\em $\mathbb{Q}$-b-Cartier}. 
If $D$ is Cartier, then we say that $\boldsymbol{\rm D}$ is {\em b-Cartier}. 

Suppose that $X$ has a projective morphism $X\to Z$ to a complex analytic space $Z$. 
We say that $\boldsymbol{\rm D}$ is {\em b-nef$/Z$} if $\boldsymbol{\rm D}$ is the Cartier closure of an $\mathbb{R}$-Cartier divisor that is nef over $Z$.   
\end{defn}

\begin{defn}[Restriction of $\mathbb{R}$-b-Cartier $\mathbb{R}$-b-divisor]\label{defn--rest-b-divisor}
Let $X$ be a normal analytic variety and $\boldsymbol{\rm D}$ an $\mathbb{R}$-b-Cartier $\mathbb{R}$-b-divisor on $X$. 
Let $U \subset X$ be an open subset. 
Then, we define the {\em restriction of $\boldsymbol{\rm D}$ over $U$}, which is an element of $\boldsymbol{\rm WDiv}_{\mathbb{R}}(U)$ and we denote by $\boldsymbol{\rm D}|_{U}$, as follows: 
Pick a proper bimeromorphic morphism $f \colon Y \to X$ from a normal analytic variety $Y$ such that $\boldsymbol{\rm D}=\overline{\boldsymbol{\rm D}_{Y}}$. 
Then we define $\boldsymbol{\rm D}|_{U}$ to be the Cartier closure of $(\boldsymbol{\rm D}_{Y})|_{f^{-1}(U)}$. We remark that the restriction $\boldsymbol{\rm D}|_{U}$ is well defined. 
In other words, $\boldsymbol{\rm D}|_{U}$ does not depend on the choice of $f \colon Y \to X$.  
\end{defn}

\begin{rem}\label{rem--rest-b-divisor}
It is not easy to give a reasonable definition of the restriction $\boldsymbol{\rm D}|_{U}$ for an $\mathbb{R}$-b-divisor that is not $\mathbb{R}$-b-Cartier. 
This is because a proper bimeromorphic morphism $V \to U$ is not necessarily extended to a proper bimeromorphic morphism $Y \to X$.    
\end{rem}

The following statement is important to define the moduli $\mathbb{R}$-b-divisors associated to generalized lc-trivial fibrations in Definition \ref{defn--can-bundle-formula}. 

\begin{thm}\label{thm--welldef-can-b-div}
Let $\pi \colon X \to S$ be a projective morphism from a normal analytic variety $X$ to a complex analytic space $S$, and let $W \subset S$ be a Stein compact subset. 
Then, after shrinking $S$ around $W$ suitably, we can globally define the canonical b-divisor $\boldsymbol{\rm K}$ on $X$, which is a $\mathbb{Z}$-b-divisor whose traces are the canonical divisors.
In other words, for any proper bimeromorphic morphism $X' \to X$ from a normal analytic variety $X'$, the canonical divisor $K_{X'}$ on $X'$ is well defined as a Weil divisor on $X'$, and furthermore, we can define these canonical divisors so that for any two proper bimeromorphic morphisms $X_{1} \to X$ and $X_{2} \to X$ such that the induced bimeromorphic map $\phi \colon X_{1} \dashrightarrow X_{2}$ is a bimeromorphic contraction, we have $\phi_{*}K_{X_{1}}=K_{X_{2}}$ as Weil divisors on $X_{2}$. 
\end{thm}

\begin{proof}
Let $f \colon Y \to X$ be a resolution of $X$. 
By shrinking $S$ around $W$, we may assume that $S$ is Stein, $\pi \circ f \colon Y \to S$ is projective,   and there exists a $(\pi \circ f)$-very ample divisor $A$ on $Y$ that defines an embedding of $Y$ into $S \times \mathbb{P}^{N}$ for some $N$ (\cite[Lemma 2.2]{eh-semistablereduction}). 
From now on, we do not shrink $S$ anymore. 

Let $\omega_{Y}$ be the canonical bundle on $Y$. 
We fix a positive integer $m$ such that some analytically general fiber $F$ of $\pi \circ f$ satisfies $H^{i}(F,(\omega_{Y}\otimes \mathcal{O}_{Y}(mA))|_{F}) = 0$ for all $i>0$ and $H^{0}(F,(\omega_{Y}\otimes \mathcal{O}_{Y}(mA))|_{F}) \neq 0$.
By the cohomology and base change theorem, we have $(\pi \circ f)_{*}(\omega_{Y}\otimes \mathcal{O}_{Y}(mA))\neq 0$. 
By Cartan's Theorem A (cf.~\cite[Theorem 2.6]{fujino-analytic-bchm}), we see that $(\pi \circ f)_{*}\bigl(\omega_{Y}\otimes \mathcal{O}_{Y}(mA)\bigr)$ is generated by global sections. 
Thus, $\omega_{Y}\otimes \mathcal{O}_{Y}(mA)$ has a non-zero global section $u$. 
Considering the local expression of $u$, we can easily check that $u$ defines a Weil divisor $D$ on $Y$. 
On the other hand, the choice of $A$ shows that $\mathcal{O}_{Y}(mA)$ has a global meromorphic section $v$ with ${\rm div}(v)=mA$. 
Then $K_{Y}:=D-mA$ is the canonical divisor on $Y$ because the multiplication by $\frac{u}{v}$ defines an isomorphism $\mathcal{O}_{Y}(K_{Y})\to \omega_{Y}$. 
Note that the divisor $K_{Y}$ depends on the choice of $u$. 
We fix $K_{Y}$ with the corresponding meromorphic section $\frac{u}{v}$.

For any proper bimeromorphic morphism $X' \to X$ from a normal analytic variety $X'$, let $Y'$ be the normalization of the main component of $Y \times_{X}X'$ and let $p \colon Y' \to Y$ and $q \colon Y' \to X'$ be the projections. 
Then $p^{*}\bigl(\frac{u}{v}\bigr)$ is a meromorphic section of the canonical sheaf $\omega_{Y'}$, and therefore it globally defines the canonical divisor $K_{Y'}$ on $Y'$ such that $p_{*}K_{Y'}=K_{Y}$. 
Hence, $K_{X'}:=q_{*}K_{Y'}$ is also well defined, and these $K_{X'}$ for all $X' \to X$ form the canonical b-divisor $\boldsymbol{\rm K}$.
\end{proof}

\begin{defn}[Generalized pair, {cf.~\cite{bz}}]\label{defn--gen-pair}
A {\em generalized sub-pair}, which we denote by $\pi \colon (X,B+M) \to Z$, consists of 
\begin{itemize}
\item
a projective morphism $\pi \colon X\to Z$ from a normal analytic variety $X$ to a complex analytic space $Z$, 
\item
an $\mathbb{R}$-divisor $B$ on $X$, and
\item
the trace $M$ of a b-nef$/Z$ $\mathbb{R}$-b-Cartier $\mathbb{R}$-b-divisor $\boldsymbol{\rm M}$ on $X$
\end{itemize}
such that $K_{X}+B+M$ is $\mathbb{R}$-Cartier. 
The $\mathbb{R}$-divisor $B$ is called the {\em boundary part}, and $\boldsymbol{\rm M}$ is called the {\em nef part}. 
A {\em generalized pair} is a generalized sub-pair whose boundary part is an effective $\mathbb{R}$-divisor. 
When the nef part is zero, a generalized (sub-)pair $(X,B) \to Z$ is a usual (sub-)pair $(X,B)$ together with $X\to Z$. 
When $Z$ is a point, we simply denote a generalized (sub-)pair by $(X,B+M)$, and this is nothing but a generalized (sub-)pair in the algebraic setting. 

Let $(X,B+M) \to Z$ be a generalized sub-pair and $P$ a prime divisor over $X$. 
Let $f \colon X' \to X$ be a proper bimeromorphic morphism from a normal analytic variety $X'$ such that $P$ appears as a prime divisor on $X'$. 
Let $M'$ be the trace of the nef part on $X'$. 
As in the case of usual pairs, there exists an $\mathbb{R}$-divisor $B'$ on $X'$ such that
$$K_{X'}+B'+M'=f^{*}(K_{X}+B+M).$$
Then we call $1-{\rm coeff}_{P}(B')$ the {\em generalized log discrepancy $a(P,X,B+M)$ of $P$ with respect to $(X,B+M)\to Z$}. 
When the nef part is zero, the generalized log discrepancy $a(P,X,B)$ coincides with the log discrepancy of $P$ with respect to the pair $(X,B)$. 

A generalized sub-pair $(X,B+M) \to Z$ is {\em generalized sub-klt} (resp.~{\em generalized sub-lc}) if we have $a(P,X,B+M)>0$ (resp.~$a(P,X,B+M)\geq 0$) for every prime divisor $P$ over $X$. 
A {\em generalized klt} (resp.~{\em generalized lc}) {\em pair} is a generalized pair that is generalized sub-klt (resp.~generalized sub-lc). 
We say that a generalized lc pair $(X,B+M) \to Z$ is {\em generalized plt} if $a(P,X,B+M)>0$ for all exceptional prime divisors $P$ over $X$. 

A {\em generalized non-klt locus} of a generalized sub-pair $(X,B+M) \to Z$ is a union of images on $X$ of all prime divisors $P$ over $X$ satisfying $a(P,X,B+M) \leq 0$. 
\end{defn}

\begin{defn}[Generalized dlt pair, {cf.~\cite[2.13]{birkar-compl}}]\label{defn--gen-dlt}
Let $(X,B+M) \to Z$ be a generalized pair with the nef part $\boldsymbol{\rm M}$. 
We say that $(X,B+M) \to Z$ is a {\em generalized dlt pair} if this is generalized lc and there exists a Zariski open dense subset $U \subset X$ satisfying the following.
\begin{itemize}
\item
$U$ is smooth and $B|_{U}$ has a simple normal crossing support, 
\item
$\boldsymbol{\rm M}|_{U}$ descends to $U$ (see Definition \ref{defn--rest-b-divisor} and Definition \ref{defn--b-divisor}), and
\item
any prime divisor $P$ over $X$ mapped into $X \setminus U$ satisfies $a(P,X,B+M)>0$.  
\end{itemize}
\end{defn}

The following result is useful, although we do not use it in this paper.

\begin{thm}[{cf.~\cite[Theorem 6.1]{has-iitakafibration}}] Let $(X,B+M) \to Z$ be a generalized lc pair whose nef part $\boldsymbol{\rm M}$ is a finite $\mathbb{R}_{>0}$-linear combination of b-nef$/Z$ $\mathbb{Q}$-b-Cartier $\mathbb{Q}$-b-divisors on $X$. Then the following conditions are equivalent: \begin{enumerate} \item$(X,B+M) \to Z$ is generalized dlt, and \item there exist a log resolution $f \colon X' \to X$ of $(X,B)$ and a Zariski open dense subset $V \subset X$ such that $\boldsymbol{\rm M}$ descends to $X'$, $f$ is a biholomorphism over $V$, and any prime divisor $P$ over $X$ mapped into $X \setminus V$ satisfies $a(P,X,B+M)>0$.  \end{enumerate} \end{thm}

\begin{proof} The argument of \cite[Proof of Theorem 6.1]{has-iitakafibration} works without any changes. In that proof, we used \cite[Theorem 1.2]{has-class} for the reduction to the case where $K_{X}$ and $M$ are $\mathbb{R}$-Cartier. The complex analytic analog of \cite[Theorem 1.2]{has-class} can be proved by using \cite[Theorem 1.3]{eh-analytic-mmp}. We recall that we do not assume that $X$ is quasi-projective in the algebraic case (\cite[Theorem 6.1]{has-iitakafibration}). Similarly, we do not need to shrink $X$ in the reduction. After that, we applied \cite[Lemma 6.2]{has-iitakafibration} and a suitable log resolution. The complex analytic analog of \cite[Lemma 6.2]{has-iitakafibration} can be proved by using the relative MMP for generalized klt pairs, which is a simple application of \cite{fujino-analytic-bchm}, and the log resolution can be found in \cite[Section 13]{resolution-1}. We note that we do not need to shrink $X$ in the latter part of the proof, as shown in the algebraic case. \end{proof}

\begin{thm}[Generalized dlt blow-up, {cf.~\cite[2.13 (3)]{birkar-compl}, \cite[Theorem 2.6]{hacon-xie-acc-lct}}]\label{thm--gen-dlt-model}
Let $\pi \colon (X,B+M) \to Z$ be a generalized pair with the nef part $\boldsymbol{\rm M}$, and let $W \subset Z$ be a compact subset such that $\pi \colon X \to Z$ and $W$ satisfy {\rm (P)}. 
Then, after shrinking $Z$ around $W$, there exist a projective bimeromorphic morphism $f \colon X' \to X$ and a generalized pair $\pi \circ f \colon (X',B'+M') \to Z$  with the nef part $\boldsymbol{\rm M}$ such that
\begin{itemize}
\item
$X'$ is $\mathbb{Q}$-factorial over $W$, 
\item
any $f$-exceptional prime divisor $E$ on $X'$ satisfies $a(E,X,B+M) \leq 0$, 
\item
$K_{X'}+B'+M'=f^{*}(K_{X}+B+M)$, and
\item
$(X', (B'^{<1}+{\rm Supp}\,B'^{\geq 1})+M') \to Z$ is a generalized dlt pair with the nef part $\boldsymbol{\rm M}$. 
\end{itemize}
We call $f \colon (X',B'+M') \to (X,B+M)$ a {\rm generalized dlt blow-up} and $(X',B'+M') \to Z$ a {\em generalized dlt model} of $(X,B+M) \to Z$. 
\end{thm}

\begin{proof}
This is an application of the relative MMP for generalized dlt pairs. 
For details, see \cite[Proof of Theorem 2.6]{hacon-xie-acc-lct}. 
\end{proof}

We use the following result, so-called connectedness principle, in Section \ref{sec--compl}. 

\begin{lem}[cf.~{\cite[Lemma 2.14]{birkar-compl}, \cite[Proposition 3.1]{filipazzi-svaldi}, \cite[Lemma 3.3]{birkar-connect}}]\label{lem--conn-princi}
Let $\pi \colon (X,B+M) \to Z$ be a generalized sub-pair such that $\pi \colon X \to Z$ is a contraction, the codimension of $\pi({\rm Supp}\,B^{<0})$ in $Z$ is at least two, and $-(K_{X}+B+M)$ is $\pi$-nef and $\pi$-big. 
Let $z \in Z$ be a point. 
Then the generalized non-klt locus of $(X,B+M) \to Z$ is connected around $\pi^{-1}(z)$. 
In particular, if $\pi \colon (X,B+M) \to Z$ is generalized plt and $\pi(\lfloor B\rfloor) \ni z$, then there exists a (not necessarily Stein) open subset $U \ni z$ of $Z$ such that $\lfloor B\rfloor|_{\pi^{-1}(U)}$ is a prime divisor. 
\end{lem}

\begin{proof}
The first statement follows from the Kawamata--Viehweg vanishing theorem in the complex analytic setting (\cite[II, 5.12 Corollary]{nakayama}). 
When $\pi \colon (X,B+M) \to Z$ is generalized plt, $\lfloor B\rfloor$ is a disjoint union of its irreducible components. 
Hence, for any point $z \in Z$ such that $\pi(\lfloor B\rfloor) \ni z$, there exists a unique prime divisor $T \subset \lfloor B\rfloor$ such that $\pi(T) \ni z$. 
Then, $U:=Z\setminus \pi(\lfloor B\rfloor-T)$ is the desired open subset. 
\end{proof}

We recall the definition of generalized adjunction. 

\begin{defn}[cf.~{\cite[Definition 4.7]{bz}}]\label{defn--gen-adjunction}
Let $(X,B+M)\to Z$ be a generalized pair with the nef part $\boldsymbol{\rm M}$, and let $\phi \colon X' \to X$ be a log resolution of $(X,B)$ such that $\boldsymbol{\rm M}$ descends to $X'$. 
Let $T$ be a component of $B^{=1}$, set $T':=\phi_{*}^{-1}T$, and let $T^{\nu}$ be the normalization of $T$. 
Let $\phi_{T^{\nu}} \colon T' \to T^{\nu}$ denote the induced morphism. 
We define $B'$ by 
$$K_{X'}+T'+B'+\boldsymbol{\rm M}_{X'}=\phi^{*}(K_{X}+B+M).$$
Suppose that there exists an $\mathbb{R}$-divisor $M' \sim_{\mathbb{R},\,Z}\boldsymbol{\rm M}_{X'}$ on $X'$ such that $T' \not\subset {\rm Supp}\,M'$. 
Note that this condition always holds when $Z$ is Stein. 
We define $\mathbb{R}$-divisors $B_{T'}$ and $B_{T^{\nu}}$ and an $\mathbb{R}$-b-divisor $\boldsymbol{\rm N}$ by
\begin{equation*}
B_{T'}:=B'|_{T'},\quad B_{T^{\nu}}:=\phi_{T^{\nu}*}B_{T'}, \qquad {\rm and}\qquad \boldsymbol{\rm N}:=\overline{(M'|_{T'})} 
\end{equation*}
Then $T^{\nu}$, $B_{T^{\nu}}$, and $\boldsymbol{\rm N}$ define a generalized pair $(T^{\nu}, B_{T^{\nu}}+\boldsymbol{\rm N}_{T^{\nu}}) \to Z$ with the nef part $\boldsymbol{\rm N}$ such that
$$K_{T^{\nu}}+B_{T^{\nu}}+\boldsymbol{\rm N}_{T^{\nu}} \sim_{\mathbb{R},\,Z}(K_{X}+B+M)|_{T^{\nu}}.$$ 
We call this process a {\em generalized adjunction}. 
By definition, it is easy to check that if $(X,B+M)\to Z$ is generalized lc (resp.~generalized plt) then $(T^{\nu}, B_{T^{\nu}}+\boldsymbol{\rm N}_{T^{\nu}}) \to Z$ is also generalized lc (resp.~generalized klt). 
\end{defn}

The following lemma is a variant of the generalized inversion of adjunction. 

\begin{lem}[cf.~{\cite[Lemma 3.2]{birkar-compl}}]\label{lem--gen-inv-adj}
Let $\pi \colon (X,B+M) \to Z$ be a generalized pair, let $T$ be a component of $B^{=1}$, and let $W \subset Z$ be a compact subset.   
Suppose that $X$ is $\mathbb{Q}$-factorial over $W$ and $(X,T)$ is plt.
Let $\pi|_{T}\colon (T,B_{T}+M_{T}) \to Z$ be the generalized pair defined by using the generalized adjunction $K_{T}+B_{T}+M_{T}\sim_{\mathbb{R},\,Z}(K_{X}+B+M)|_{T}$. 
In this situation, if $(T,B_{T}+M_{T}) \to Z$ is generalized lc around $\pi^{-1}(W) \cap T$, then so is $(X,B+M) \to Z$. 
\end{lem}

\begin{proof}
We can apply \cite[Proof of Lemma 3.2]{birkar-compl} because we can use the inversion of adjunction \cite[Theorem 1.1]{fujino-analytic-inv-adjunction}. 
\end{proof}

We close this section with a result on how the coefficients of the boundary parts are inherited under generalized adjunction.

\begin{lem}[{\cite[Proposition 3.2]{chhx-compl}}]\label{lem--adj-coeff}
Fix a positive integer $p$ and a finite subset $\mathfrak{R} \subset \mathbb{Q}\cap[0,1]$. 
Then there exists a finite subset $\mathfrak{R}' \subset \mathbb{Q}\cap[0,1]$, depending only on $p$ and $\mathfrak{R}$, satisfying the following.
Let $(X,B+M)\to Z$ be a generalized lc pair with the nef part $\boldsymbol{\rm M}$ such that $B \in \Phi(\mathfrak{R})$ and $p \boldsymbol{\rm M}$ is b-Cartier. 
Let $T$ be an irreducible component of $\lfloor B \rfloor$ with the normalization $T^{\nu} \to T$, and let $(T^{\nu}, B_{T^{\nu}}+M_{T^{\nu}}) \to Z$ be the generalized lc pair defined by using the generalized adjunction $K_{T^{\nu}}+B_{T^{\nu}}+M_{T^{\nu}}\sim_{\mathbb{R},\,Z}(K_{X}+B+M)|_{T^{\nu}}$. 
Then $B_{T^{\nu}} \in \Phi(\mathfrak{R}')$. 
Moreover, if we have $B \in \Gamma(n, \Phi(\mathfrak{R}))$ for some positive integer $n$, then $B_{T^{\nu}} \in \Gamma(n, \Phi(\mathfrak{R}'))$. 
\end{lem}

\begin{proof}
The argument of \cite[Proof of Proposition 3.2]{chhx-compl} works with no changes. 
In that proof, we only need \cite[Lemma 3.3]{birkar-compl}. 
The proof of \cite[Lemma 3.3]{birkar-compl} only needs \cite[3.9.~Proposition, 3.10.~Corollary]{shokurov-flip}, which hold even in the complex analytic setting. 
\end{proof}

\section{Generalized canonical bundle formula}\label{sec--gen-can-bundle-formula}

The goal of this section is to establish the generalized canonical bundle formula for projective morphisms between analytic varieties over non-compact bases. 
For a more general setup, see \cite[Conjecture 2.1]{haconpaun}.

\begin{defn}[{cf.~\cite[Section 4]{filipazzi-gen-can-bundle-formula}, \cite[Subsection 2.5]{filipazzi-svaldi}}]\label{defn--can-bundle-formula}
Let $(X,B+M) \to S$ be a generalized sub-pair with the nef part $\boldsymbol{\rm M}$.  
We define an $\mathbb{R}$-b-divisor $\boldsymbol{\rm A}^{*}(X,B+M)$ on $X$ by the condition that for any proper bimeromorphic morphism $Y \to X$, the trace of $\boldsymbol{\rm A}^{*}(X,B+M)$ on $Y$ is 
$$\boldsymbol{\rm A}^{*}(X,B+M)_{Y}:=\sum_{a(P,X,B+M)>0}(a(P,X,B+M)-1)P,$$
where $P$ runs over prime divisors on $Y$. 
The sheaf $\mathcal{O}_{X}(\lceil \boldsymbol{\rm A}^{*}(X,B+M)\rceil)$ is defined by
$$\mathcal{O}_{X}(\lceil \boldsymbol{\rm A}^{*}(X,B+M)\rceil)(V):=\left\{
\begin{array}{l}
\text{meromorphic function $\varphi$ on $V$ such that}\\
\text{${\rm div}(h^{*}\varphi)|_{h^{-1}(V)}+\lceil \boldsymbol{\rm A}^{*}(X,B+M)_{Y}\rceil|_{h^{-1}(V)} \geq 0$ for}\\
\text{any proper bimeromorphic morphism $h \colon Y \to X$}
\end{array}\right\}$$
for any open subset $V \subset X$. 

Let $f \colon X \to Z$ be a contraction over $S$ such that 
\begin{itemize}
\item
the structure morphism $Z \to S$ is projective, 
\item
$K_{X}$ is globally defined as a Weil divisor on $X$ and $K_{X}+B+M \sim_{\mathbb{R},\,Z}0$, 
\item
there exists a Zariski open dense subset $U \subset Z$ such that if we put $X_{U}:=f^{-1}(U)$ then
 $(X_{U},B|_{X_{U}}+M|_{X_{U}}) \to U$ is generalized sub-lc, and
 \item
${\rm rank}f_{*}\mathcal{O}_{X}(\lceil \boldsymbol{\rm A}^{*}(X,B+M)\rceil)=1$.
\end{itemize}
We call the structure of $f \colon X \to Z$ over $S$ together with $(X,B+M) \to S$ a {\em generalized lc-trivial fibration} over $S$, and we denote it by $f\colon (X,B+M) \to Z$ if there is no risk of confusion. 
If $(X_{U},B|_{X_{U}}+M|_{X_{U}}) \to U$ is generalized sub-klt in the third condition, then we call $f\colon (X,B+M) \to Z$ a {\em generalized klt-trivial fibration} over $S$. 
When $\boldsymbol{\rm M}=0$ and the sub-pair $(X_{U},B|_{X_{U}})$ is sub-lc (resp.~sub-klt), we call $f\colon (X,B) \to Z$  an {\em lc-trivial fibration} (resp.~{\em klt-trivial fibration}) over $S$. 

We define the {\em discriminant $\mathbb{R}$-b-divisor} and the {\em moduli $\mathbb{R}$-b-divisor} on $Z$ associated to $f$ as follows: 
For any proper bimeromorphic morphism $\tau \colon Z' \to Z$, we take a resolution $\tau_{X} \colon X' \to X$ of $X$ such that the induced meromorphic map $f' \colon X' \dashrightarrow Z'$ is a morphism. 
Then we have the diagram
 $$
\xymatrix{
X' \ar@{->}[d]_{f'} \ar[r]^{\tau_{X}}&X\ar@{->}[d]^{f}\\
Z'\ar[r]_{\tau}&Z. 
}
$$
Although $X' \to S$ is only locally projective, we still regard $(X',B'+M') \to S$ as a generalized sub-pair, where $M':=\boldsymbol{\rm M}_{X'}$ and $B'$ is defined to be 
$$K_{X'}+B'+M'=\tau_{X}^{*}(K_{X}+B+M).$$
For any prime divisor $Q$ on $Z'$, after shrinking $Z'$ around a general point of $Q$ we can define the generalized lc threshold $t_{Q}$ of $f'^{*}Q$ with respect to $(X',B'+M')\to S$. 
Then, the discriminant $\mathbb{R}$-b-divisor $\boldsymbol{\rm G}$ associated to $f$ is defined by 
$$\boldsymbol{\rm G}_{Z'}:=\sum_{Q}(1-t_{Q})Q,$$
where $Q$ runs over prime divisors on $Z'$. 
To define the moduli $\mathbb{R}$-b-divisor, we fix an $\mathbb{R}$-Cartier divisor $D$ on $Z$ such that $K_{X}+B+M\sim_{\mathbb{R}}f^{*}D$. 
Then, the moduli $\mathbb{R}$-b-divisor $\boldsymbol{\rm N}$ associated to $f$ is defined by 
$$\boldsymbol{\rm N}_{Z'}:= \tau^{*}D-(K_{Z'}+\boldsymbol{\rm G}_{Z'}).$$
A priori, $\boldsymbol{\rm N}$ is not globally defined as an $\mathbb{R}$-b-divisor on $Z$ because the canonical divisors are only locally defined in general. 
However, once we shrink $S$ around a Stein compact subset appropriately, $\boldsymbol{\rm N}$ is globally defined (see Theorem \ref{thm--welldef-can-b-div}).  
We remark that the discriminant $\mathbb{R}$-b-divisor depends only on $(X,B+M)\to S$ and $f \colon X \to Z$. 
On the other hand, to define the moduli $\mathbb{R}$-b-divisor, we need to fix such $D$ as above. 
Hence, the moduli $\mathbb{R}$-b-divisor is defined up to $\mathbb{R}$-linear equivalence over $S$. 

By definition, we have
$$K_{X}+B+M\sim_{\mathbb{R}}f^{*}(K_{Z}+\boldsymbol{\rm G}_{Z}+\boldsymbol{\rm N}_{Z}).$$
If $\boldsymbol{\rm N}$ is a b-nef$/S$ $\mathbb{R}$-b-Cartier $\mathbb{R}$-b-divisor, in other words, if $(Z, \boldsymbol{\rm G}_{Z}+\boldsymbol{\rm N}_{Z}) \to S$ is a generalized sub-pair with the nef part $\boldsymbol{\rm N}$, then we call the above relation the {\em generalized canonical bundle formula}.  
\end{defn}

\begin{rem}\label{rem--can-bundle-formula-rest}
With notations as in Definition \ref{defn--can-bundle-formula}, pick an open subset $\tilde{S} \subset S$, and let $\tilde{X} \subset X$ and $\tilde{Z} \subset Z$ be the inverse images of $\tilde{S}$. 
By the restriction, we get a generalized sub-pair $(\tilde{X},\tilde{B}+\tilde{M}) \to \tilde{S}$ with the nef part $\boldsymbol{\rm M}|_{\tilde{X}}$ and a contraction $\tilde{f} \colon \tilde{X} \to \tilde{Z}$. 
Hence, using the restriction $D|_{\tilde{Z}}$, we can define the discriminant $\mathbb{R}$-b-divisor $\tilde{\boldsymbol{\rm G}}$ and the moduli $\mathbb{R}$-b-divisor $\tilde{\boldsymbol{\rm N}}$ associated to $\tilde{f}$. 
It is  natural to consider relations between $\boldsymbol{\rm G}$ (resp.~$\boldsymbol{\rm N}$) and $\tilde{\boldsymbol{\rm G}}$ (resp.~$\tilde{\boldsymbol{\rm N}}$). 
However, we cannot directly compare these $\mathbb{R}$-b-divisors because the restriction of $\mathbb{R}$-b-divisors to an open subset cannot be defined well (Remark \ref{rem--rest-b-divisor}). 
\end{rem}

\begin{lem}\label{lem--restriction-can-bundle-formula}
With notations as in Definition \ref{defn--can-bundle-formula} and Remark \ref{rem--can-bundle-formula-rest}, fix  an $\mathbb{R}$-Cartier divisor $D$ on $Z$ such that $K_{X}+B+M\sim_{\mathbb{R}}f^{*}D$. 
Then the following properties hold.
\begin{enumerate}[$(a)$]
\item\label{lem--restriction-can-bundle-formula-(a)}
For any proper bimeromorphic morphism $\tau \colon Z' \to Z$ with $\tilde{Z}':=\tau^{-1}(\tilde{Z})$, we have 
$$(\boldsymbol{\rm G}_{Z'})|_{\tilde{Z}'}=\tilde{\boldsymbol{\rm G}}_{\tilde{Z}'} \qquad {\rm and} \qquad (\boldsymbol{\rm N}_{Z'})|_{\tilde{Z}'}=\tilde{\boldsymbol{\rm N}}_{\tilde{Z}'}$$
as $\mathbb{R}$-divisors on $\tilde{Z}'$.

\item\label{lem--restriction-can-bundle-formula-(b)}
Let $\boldsymbol{\rm K}$ {\rm (}resp.~$\tilde{\boldsymbol{\rm K}}${\rm)} be the canonical b-divisor on $Z$ {\rm (}resp.~$\tilde{Z}${\rm )}, that is, a $\mathbb{Z}$-b-divisor whose traces are the canonical divisors. 
Suppose that $\boldsymbol{\rm N}$ {\rm (}or, equivalently, $\boldsymbol{\rm K}+\boldsymbol{\rm G}${\rm )} is $\mathbb{R}$-b-Cartier. 
Define $\boldsymbol{\rm G}|_{\tilde{Z}}:=(\boldsymbol{\rm K}+\boldsymbol{\rm G})|_{\tilde{Z}}-\tilde{\boldsymbol{\rm K}}$. 
Then 
$$\boldsymbol{\rm G}|_{\tilde{Z}}\geq \tilde{\boldsymbol{\rm G}} \qquad {\rm and} \qquad \boldsymbol{\rm N}|_{\tilde{Z}}\leq \tilde{\boldsymbol{\rm N}}$$
as $\mathbb{R}$-b-divisors on $\tilde{Z}$. 

\item\label{lem--restriction-can-bundle-formula-(c)}
Suppose that both $\boldsymbol{\rm N}$ and $\tilde{\boldsymbol{\rm N}}$ are $\mathbb{R}$-b-Cartier and $\tilde{\boldsymbol{\rm N}}$ is b-nef$/\tilde{S}$. 
Then
$$(\boldsymbol{\rm K}+\boldsymbol{\rm G})|_{\tilde{Z}} = \tilde{\boldsymbol{\rm K}}+\tilde{\boldsymbol{\rm G}} \qquad {\rm and} \qquad \boldsymbol{\rm N}|_{\tilde{Z}}=\tilde{\boldsymbol{\rm N}}$$
as $\mathbb{R}$-b-divisors on $\tilde{Z}$. 
\end{enumerate}
\end{lem}

\begin{proof}
Because we will give the detailed arguments, the proof and the notations will be quite complicated. 
However, the ideas are standard and well known to the experts. 
The first assertion is shown by modifying $f \colon X \to Z$ to the situation of \cite[Theorem 2]{kawamata-subadjunction-II} and calculations of lc thresholds. 
The second assertion is shown by computations of log discrepancies and generalized lc thresholds after using the weak semistable reduction as in \cite[Theorem 1.17]{haconpaun} (cf.~\cite[Subsection 4.5]{eh-semistablereduction}) and the finite base change property (\cite[Theorem 3.2]{ambro-phdthesis}). 
The third assertion follows from the second assertion and the negativity lemma. 
In each assertion, we only need to show one of two relations. 
In Lemma \ref{lem--restriction-can-bundle-formula} (\ref{lem--restriction-can-bundle-formula-(a)}) and Lemma \ref{lem--restriction-can-bundle-formula} (\ref{lem--restriction-can-bundle-formula-(b)}) we focus on the discriminant $\mathbb{R}$-b-divisors, whereas we deal with the moduli $\mathbb{R}$-b-divisors in Lemma \ref{lem--restriction-can-bundle-formula} (\ref{lem--restriction-can-bundle-formula-(c)}). 

Since $\tilde{\boldsymbol{\rm G}}$ and $\boldsymbol{\rm G}$ are $\mathbb{R}$-b-divisors, to prove Lemma \ref{lem--restriction-can-bundle-formula} (\ref{lem--restriction-can-bundle-formula-(a)}), we may replace $\tau \colon Z' \to Z$ by a suitable bimeromorphic morphism (\cite[Corollary 2]{hironaka-flattening}, \cite[Theorem 13.2]{resolution-1}) so that the induced diagram 
 $$
\xymatrix{
X' \ar@{->}[d]_{f'} \ar[r]^{\tau_{X}}&X\ar@{->}[d]^{f}\\
Z'\ar[r]_{\tau}&Z
}
$$
of locally projective morphisms, where $X'$ is a resolution of the main component of $X \times_{Z}Z'$ on which the nef part $\boldsymbol{\rm M}$ descends, satisfies the following: 
If we put $M':=\boldsymbol{\rm M}_{X'}$ and define $B'$ by
$$K_{X'}+B'+M'=\tau_{X}^{*}(K_{X}+B+M),$$
then there exist simple normal crossing divisors $\Sigma_{X'} \subset X'$ and $\Sigma_{Z'} \subset Z'$ such that
\begin{itemize}
\item
 $\Sigma_{Z'}\supset {\rm Supp}\,\boldsymbol{\rm G}_{Z'} \cup f'(\Sigma_{X'}^{v})$ and $\Sigma_{X'} \supset {\rm Supp}\,B' \cup f'^{-1 }(\Sigma_{Z'})$, where $\Sigma_{X'}^{v}$ denotes the sum of all components of $\Sigma_{X'}$ that do not dominate $Z'$, and
\item
$f'$ is smooth over  $Z' \setminus \Sigma_{Z'}$ and $\Sigma_{X'}$ is simple normal crossing over $Z' \setminus \Sigma_{Z'}$. 
\end{itemize}
We note that we do not need to shrink $S$ to obtain the above diagram since we do not assume that $\tau$ or $\tau_{X}$ is a finite sequence of blow-ups. 
Let $\tilde{Q}$ be a prime divisor on $\tilde{Z}'$. 
If ${\tilde{Q}} \not\subset \Sigma_{Z'}|_{\tilde{Z}'}$, then the second condition immediately implies 
$${\rm coeff}_{\tilde{Q}}\bigl((\boldsymbol{\rm G}_{Z'})|_{\tilde{Z}'}\bigr)={\rm coeff}_{\tilde{Q}}\bigl(\tilde{\boldsymbol{\rm G}}_{\tilde{Z}'}\bigr)=0.$$
Note that this equality holds regardless of whether $\tilde{Q}$ is a component of $P'|_{\tilde{Z}'}$ for some prime divisor $P'$ on $Z'$. 
If ${\tilde{Q}} \subset \Sigma_{Z'}|_{\tilde{Z}'}$, then $\tilde{Q}$ is a component of $Q'|_{\tilde{Z}'}$ for some prime divisor $Q' \subset \Sigma_{Z'}$. 
Let $t_{Q'}$ (resp.~$t_{\tilde{Q}}$) be the generalized lc threshold of $f'^{*}Q'$ (resp.~$f'^{*}\tilde{Q}$) with respect to $(X',B'+M') \to S$ over a general point of $Q'$ (resp.~$\tilde{Q}$) after shrinking $Z'$ (resp.~$\tilde{Z}'$). 
Since all the relevant divisors have simple normal crossing supports, we can easily check that $t_{Q'}=t_{\tilde{Q}}$. 
This shows 
$${\rm coeff}_{\tilde{Q}}\bigl((\boldsymbol{\rm G}_{Z'})|_{\tilde{Z}'}\bigr)= 1-t_{Q'} = 1-t_{\tilde{Q}} ={\rm coeff}_{\tilde{Q}}\bigl(\tilde{\boldsymbol{\rm G}}_{\tilde{Z}'}\bigr).$$
Thus, we have $(\boldsymbol{\rm G}_{Z'})|_{\tilde{Z}'}=\tilde{\boldsymbol{\rm G}}_{\tilde{Z}'}$, and therefore Lemma \ref{lem--restriction-can-bundle-formula} (\ref{lem--restriction-can-bundle-formula-(a)}) holds. 

For Lemma \ref{lem--restriction-can-bundle-formula} (\ref{lem--restriction-can-bundle-formula-(b)}), 
we fix a proper bimeromorphic morphism $\overline{Z} \to \tilde{Z}$. 
We will show $(\boldsymbol{\rm G}|_{\tilde{Z}})_{\overline{Z}} \geq \tilde{\boldsymbol{\rm G}}_{\overline{Z}}$. 
Since it is sufficient to show this relation locally, by Lemma \ref{lem--restriction-can-bundle-formula} (\ref{lem--restriction-can-bundle-formula-(a)}), we may shrink $S$ freely whenever necessary. 
In particular,  using \cite[Corollary 2]{hironaka-flattening} we can assume that every morphism that appears in this paragraph is projective. 
Replacing $Z$ we may assume that $Z$ is smooth, $\boldsymbol{\rm G}_{Z}$ is a simple normal crossing $\mathbb{R}$-divisor on $Z$, and $\boldsymbol{\rm K}+\boldsymbol{\rm G}$ descends to $Z$. 
Replacing $X$ by a suitable resolution, we may assume that $X$ is smooth and the nef part $\boldsymbol{\rm M}$ of the generalized sub-pair $(X,B+M) \to S$ descends to $X$. 
We apply \cite[Theorem 1.17]{haconpaun} after we shrink $S$ (see also \cite[Theorem 9.5]{karu-phdthesis} and \cite[Subsection 4.5]{eh-semistablereduction}). 
By this theorem and replacing $Z$, we obtain a diagram
 $$
\xymatrix{
X^{\rm ws} \ar@{->}[r]^{\mu_{X}}\ar@{->}[d]_{f^{\rm ws}} & X \ar@{->}[d]^{f}\\
Z^{\rm ws}\ar[r]_{\mu}&Z
}
$$
such that 
\begin{itemize}
\item
$\mu \colon Z^{\rm ws} \to Z$ is a finite surjective morphism and $\mu_{X} \colon X^{\rm ws} \to X$ is a generically finite surjective morphism, and
\item
there exist reduced divisors $\Delta_{X^{\rm ws}}$ and $\Delta_{Z^{\rm ws}}$ such that 
$$f^{\rm ws}\colon (X^{\rm ws},\Delta_{X^{\rm ws}}) \to (Z^{\rm ws},\Delta_{Z^{\rm ws}})$$
is weakly semistable as a toroidal morphism with good horizontal divisors (see \cite[Definition 9.1]{karu-phdthesis}), and furthermore, if we define 
$B^{\rm ws}$ and $M^{\rm ws}$ by 
$$K_{X^{\rm ws}}+B^{\rm ws}=\mu^{*}_{X}(K_{X}+B) \qquad {\rm and} \qquad M^{\rm ws}:=\mu^{*}_{X}M$$ respectively, then  $\Delta_{X^{\rm ws}} \supset {\rm Supp}\,B^{\rm ws} \cup {\rm Ex}(\mu_{X})$ and $\Delta_{Z^{\rm ws}}$ contains the union of ${\rm Supp}\,\mu^{*}\boldsymbol{\rm G}_{Z}$ and the ramification locus of $\mu \colon Z^{\rm ws} \to Z$.   
\end{itemize}
By applying \cite[Theorem 13.2]{resolution-1} after shrinking $S$, we may assume that $\overline{Z} \to \tilde{Z}$ is a finite sequence of blow-ups
$$\overline{Z}:=\overline{Z}_{n} \to \overline{Z}_{n-1} \to \cdots \to \overline{Z}_{1} \to \overline{Z}_{0}:=\tilde{Z}$$
such that the center of each blow-up is admissible (\cite[(1.2)]{resolution-1}) with respect to the inverse image of a union of ${\rm Supp}\,(\boldsymbol{\rm G}_{Z}|_{\tilde{Z}})$ and the branch locus of $\mu \colon Z^{\rm ws} \to Z$. 
By the complex analytic analog of \cite[Chapter 2, Lemma 8.3]{karu-phdthesis} and taking the base change of $f^{\rm ws}\colon (X^{\rm ws},\Delta_{X^{\rm ws}}) \to (Z^{\rm ws},\Delta_{Z^{\rm ws}})$ by this sequence over $Z$, for any $0\leq i \leq n$, we get a toroidal and weakly semistable morphism with good horizontal divisors
$$\overline{f}^{\rm ws}_{i}\colon  (\overline{X}^{\rm ws}_{i},\Delta_{\overline{X}^{\rm ws}_{i}}) \to (\overline{Z}^{\rm ws}_{i},\Delta_{\overline{Z}^{\rm ws}_{i}})$$
 satisfying the same property as above, where $\overline{Z}^{\rm ws}_{i}=Z^{\rm ws} \times_{Z}\overline{Z}_{i}$. 
In this situation, we can apply the complex analytic analog of \cite[Theorem 3.2]{ambro-phdthesis} and direct computations of log discrepancies (\cite[Lemma 2.29]{kollar-mori}) and generalized lc thresholds by using the weak semistability with good horizontal divisors (\cite[Definition 9.1]{karu-phdthesis}).  
By carrying out a similar argument as in \cite[Proof of Theorem 4.13]{filipazzi-gen-can-bundle-formula}, we obtain $(\boldsymbol{\rm G}|_{\tilde{Z}})_{\overline{Z}} \geq \tilde{\boldsymbol{\rm G}}_{\overline{Z}}$. 

To prove Lemma \ref{lem--restriction-can-bundle-formula} (\ref{lem--restriction-can-bundle-formula-(c)}), we fix a proper bimeromorphic morphism $Z'' \to Z$ such that $\boldsymbol{\rm N}$ descends to $Z''$. 
Let $\tilde{Z}''$ be the inverse image of $\tilde{S}$ in $Z''$. 
We also fix a proper bimeromorphic morphism $\sigma \colon \tilde{Z}''' \to \tilde{Z}''$ such that $\tilde{\boldsymbol{\rm N}}$ descends to $\tilde{Z}'''$ and $\tilde{\boldsymbol{\rm N}}_{\tilde{Z}'''}$ is nef over $\tilde{S}$. 
$$
\xymatrix{
\tilde{Z}''' \ar@{->}[r]^{\sigma} & \tilde{Z}'' \ar@{_{(}->}[d]\ar[r]&\tilde{Z}\ar@{_{(}->}[d]\ar[r]&\tilde{S}\ar@{_{(}->}[d]\\
&Z''\ar[r]&Z\ar[r]&S.
}
$$
By Definition \ref{defn--rest-b-divisor}, it is enough to prove the equality $(\boldsymbol{\rm N}|_{\tilde{Z}})_{\tilde{Z}'''}=\tilde{\boldsymbol{\rm N}}_{\tilde{Z}'''}$ of the traces on $\tilde{Z}'''$. 
As in the proof of Lemma \ref{lem--restriction-can-bundle-formula} (\ref{lem--restriction-can-bundle-formula-(b)}), using Lemma \ref{lem--restriction-can-bundle-formula} (\ref{lem--restriction-can-bundle-formula-(a)}) and \cite[Corollary 2]{hironaka-flattening}, we can assume that every morphism that appears in this paragraph is projective. 
By Lemma \ref{lem--restriction-can-bundle-formula} (\ref{lem--restriction-can-bundle-formula-(b)}), we have $(\boldsymbol{\rm N}|_{\tilde{Z}})_{\tilde{Z}'''}+E=\tilde{\boldsymbol{\rm N}}_{\tilde{Z}'''}$ for some effective $\mathbb{R}$-divisor $E$ on $\tilde{Z}'''$, and $\sigma_{*}E=0$ by Lemma \ref{lem--restriction-can-bundle-formula} (\ref{lem--restriction-can-bundle-formula-(a)}).  
Furthermore, we have $(\boldsymbol{\rm N}|_{\tilde{Z}})_{\tilde{Z}'''}=\sigma^{*}\bigl( (\boldsymbol{\rm N}|_{\tilde{Z}})_{\tilde{Z}''}\bigr)$ by Definition \ref{defn--rest-b-divisor}. 
From these arguments, we have
$$\tilde{\boldsymbol{\rm N}}_{\tilde{Z}'''}=(\boldsymbol{\rm N}|_{\tilde{Z}})_{\tilde{Z}'''}+E\sim_{\mathbb{R},\,\tilde{Z}''}E,$$
where $E$ is effective and $\sigma$-exceptional. 
Since $\tilde{\boldsymbol{\rm N}}_{\tilde{Z}'''}$ is nef over $\tilde{S}$, the negativity lemma implies $E=0$. 
Hence, we have $(\boldsymbol{\rm N}|_{\tilde{Z}})_{\tilde{Z}'''}=\tilde{\boldsymbol{\rm N}}_{\tilde{Z}'''}$, and Lemma \ref{lem--restriction-can-bundle-formula} (\ref{lem--restriction-can-bundle-formula-(c)}) holds.  
\end{proof}

\begin{rem}[{cf.~\cite[Theorem 4.13]{filipazzi-gen-can-bundle-formula}}]\label{rem--descend-variety}
Let $(X,B+M) \to S$ be a generalized sub-pair with the nef part $\boldsymbol{\rm M}$, and let $f\colon (X,B+M) \to Z$ be a generalized lc-trivial fibration over $S$. 
Let $W \subset S$ be a Stein compact subset. 
Suppose that the following properties hold.
\begin{itemize}
\item
$\boldsymbol{\rm M}$ descends to $X$,  
\item
$Z$ is smooth and the trace $\boldsymbol{\rm G}_{Z}$ of the discriminant $\mathbb{R}$-b-divisor $\boldsymbol{\rm G}$ on $Z$ has simple normal crossing support,
\item
there exists a diagram
 $$
\xymatrix{
X^{\rm ws} \ar@{->}[r]^{\mu_{X}}\ar@{->}[d]_{f^{\rm ws}} & X \ar@{->}[d]^{f}\\
Z^{\rm ws}\ar[r]_{\mu}&Z
}
$$
from normal analytic varieties $X^{\rm ws}$ and $Z^{\rm ws}$ such that 
\begin{itemize}
\item
$Z^{\rm ws} \to Z$ is a finite surjective morphism and $X^{\rm ws} \to X$ is a generically finite surjective morphism, and
\item
there exist divisors $\Delta_{X^{\rm ws}}$ and $\Delta_{Z^{\rm ws}}$ such that the morphism
$$f^{\rm ws}\colon (X^{\rm ws},\Delta_{X^{\rm ws}}) \to (Z^{\rm ws},\Delta_{Z^{\rm ws}})$$
is weakly semistable as a toroidal morphism with good horizontal divisors, and furthermore, if we define $B^{\rm ws}$ and $M^{\rm ws}$ by $K_{X^{\rm ws}}+B^{\rm ws}=\mu^{*}_{X}(K_{X}+B)$ and $M^{\rm ws}:=\mu^{*}_{X}M$ respectively, then $\Delta_{X^{\rm ws}} \supset {\rm Supp}\,B^{\rm ws} \cup {\rm Ex}(\mu_{X})$ and $\Delta_{Z^{\rm ws}}$ contains the union of ${\rm Supp}\,\mu^{*}\boldsymbol{\rm G}_{Z}$ and the ramification locus of $\mu \colon Z^{\rm ws} \to Z$.  
\end{itemize}
\end{itemize}
As in the proof of Lemma \ref{lem--restriction-can-bundle-formula} (\ref{lem--restriction-can-bundle-formula-(b)}), after shrinking $S$ around $W$ and taking an appropriate modification we may always assume these properties. 
Suppose the following condition holds (cf.~\cite[Theorem 21.4 (4)]{fujino-analytic-bchm}):
\begin{itemize}
\item[($*$)]
For any projective bimeromorphic morphism $Z' \to Z$ and any point $s \in S$, there exist an open neighborhood $\tilde{S}$ of $s$ with the restrictions $\tilde{f} \colon \tilde{X} \to \tilde{Z}$ and $\tilde{Z}' \to \tilde{Z}$ over $\tilde{S}$ and a projective bimeromorphic morphism $\tilde{Z}'' \to \tilde{Z}'$ such that the trace $\tilde{\boldsymbol{\rm N}}_{\tilde{Z}''}$ of the moduli $\mathbb{R}$-b-divisor $\tilde{\boldsymbol{\rm N}}$ associated to $\tilde{f}$ is nef over $\tilde{S}$. 
\end{itemize}
In this situation, it follows that the moduli $\mathbb{R}$-b-divisor $\boldsymbol{\rm N}$ associated to $f$ is a b-nef$/S$ $\mathbb{R}$-b-Cartier $\mathbb{R}$-b-divisor on $Z$ and it descends to $Z$.
Indeed, we fix an arbitrary proper bimeromorphic morphism $\tau \colon Z' \to Z$. 
Replacing $Z'$ if necessary, we may assume that $Z'$ is smooth, ${\rm Supp}\,\tau^{*}\boldsymbol{\rm G}_{Z} \cup {\rm Ex}(\tau)$ is a normal crossing divisor, and $\tau$ is locally a finite sequence of blow-ups with admissible centers (\cite[(1.2)]{resolution-1}). 
By Lemma \ref{lem--restriction-can-bundle-formula} (\ref{lem--restriction-can-bundle-formula-(a)}) and the same argument as in the proof of Lemma \ref{lem--restriction-can-bundle-formula} (\ref{lem--restriction-can-bundle-formula-(b)}), we can locally check 
$$K_{Z'}+\boldsymbol{\rm G}_{Z'} \leq \tau^{*}(K_{Z}+\boldsymbol{\rm G}_{Z}).$$ 
Therefore $\boldsymbol{\rm N}_{Z'} \geq \tau^{*}\boldsymbol{\rm N}_{Z}$. 
On the other hand, for any point $s \in S$, the assumption ($*$) says that there exist an open neighborhood $\tilde{S}$ of $s$ with the restrictions $\tilde{f} \colon \tilde{X} \to \tilde{Z}$ and $\tilde{\tau} \colon \tilde{Z}' \to \tilde{Z}$ over $\tilde{S}$ and a projective bimeromorphic morphism $\tilde{\tau}' \colon \tilde{Z}'' \to \tilde{Z}'$ such that the trace $\tilde{\boldsymbol{\rm N}}_{\tilde{Z}''}$ of the moduli $\mathbb{R}$-b-divisor $\tilde{\boldsymbol{\rm N}}$ associated to $\tilde{f}$ is nef over $\tilde{S}$
$$\xymatrix{\tilde{Z}'' \ar@{->}[r]^{\tilde{\tau}'} & \tilde{Z}' \ar@{_{(}->}[d]\ar[r]^{\tilde{\tau}}&\tilde{Z}\ar@{_{(}->}[d]\ar[r]&\tilde{S}\ar@{_{(}->}[d]\\&Z'\ar[r]_{\tau}&Z\ar[r]&S.}$$
Then 
$\tilde{\tau}'^{*}\tilde{\tau}^{*}(\boldsymbol{\rm N}_{Z}|_{\tilde{Z}}) \leq \tilde{\tau}'^{*}(\boldsymbol{\rm N}_{Z'}|_{\tilde{Z}'})\leq \tilde{\boldsymbol{\rm N}}_{\tilde{Z}''}$ by the previous discussion. 
Since $\boldsymbol{\rm N}_{Z}|_{\tilde{Z}}=\tilde{\boldsymbol{\rm N}}_{\tilde{Z}}$ by Lemma \ref{lem--restriction-can-bundle-formula} (\ref{lem--restriction-can-bundle-formula-(a)}), the negativity lemma implies $\tilde{\boldsymbol{\rm N}}_{\tilde{Z}''}\leq \tilde{\tau}'^{*}\tilde{\tau}^{*}(\boldsymbol{\rm N}_{Z}|_{\tilde{Z}})$. 
Then
$$\tilde{\tau}'^{*}\tilde{\tau}^{*}(\boldsymbol{\rm N}_{Z}|_{\tilde{Z}}) \leq \tilde{\tau}'^{*}(\boldsymbol{\rm N}_{Z'}|_{\tilde{Z}'})\leq \tilde{\boldsymbol{\rm N}}_{\tilde{Z}''}\leq \tilde{\tau}'^{*}\tilde{\tau}^{*}(\boldsymbol{\rm N}_{Z}|_{\tilde{Z}}),$$
which shows $(\boldsymbol{\rm N}_{Z'})|_{\tilde{Z}'} = (\tau^{*}\boldsymbol{\rm N}_{Z})|_{\tilde{Z}'}$. 
Since $s \in S$ is arbitrary and $\tilde{Z}'$ contains the fiber over $s$, we see that $\boldsymbol{\rm N}_{Z'} = \tau^{*}\boldsymbol{\rm N}_{Z}$. 
Therefore $\boldsymbol{\rm N}$ descends to $Z$. 
The property of being b-nef$/S$ can also be checked locally by using ($*$).  
\end{rem}

The following result is a variant of \cite[Theorem 21.4]{fujino-analytic-bchm}. 

\begin{thm}[{cf.~\cite[Theorem 21.4]{fujino-analytic-bchm}}]\label{thm--klt-trivial-fib-start}
Let $\pi \colon X \to S$ be a projective morphism from a normal analytic variety $X$ to a Stein space $S$, and let $W \subset S$ be a Stein compact subset. 
Let $f \colon (X,\Delta) \to Z$ be a klt-trivial fibration over $S$ such that $\Delta$ is a $\mathbb{Q}$-divisor on $X$. 
Then, replacing $S$ with a suitable open neighborhood of $W$, we have the following properties.
\begin{enumerate}[(i)]
\item \label{thm--klt-trivial-fib-start-(i)}
Let $\boldsymbol{\rm G}$ (resp.~$\boldsymbol{\rm N}$) be the discriminant $\mathbb{Q}$-b-divisor (resp.~the moduli $\mathbb{Q}$-b-divisor) associated to $f$. 
Then $\boldsymbol{\rm N}$ is a b-nef$/S$ $\mathbb{Q}$-b-Cartier $\mathbb{Q}$-b-divisor on $Z$. 
In particular, $(Z, \boldsymbol{\rm G}_{Z} +\boldsymbol{\rm N}_{Z}) \to S$ is a generalized sub-pair with the nef part $\boldsymbol{\rm N}$. 
Furthermore, if $(X,\Delta)$ is klt, then $(Z, \boldsymbol{\rm G}_{Z} +\boldsymbol{\rm N}_{Z}) \to S$ is generalized klt. 
\item \label{thm--klt-trivial-fib-start-(ii)}
For any open subset $\tilde{S} \subset S$ that does not necessarily contain $W$, if we put $\tilde{X} \subset X$ and $\tilde{Z} \subset Z$ as the inverse images of $\tilde{S}$, then the klt-trivial fibration $(\tilde{X}, \Delta|_{\tilde{X}}) \to \tilde{Z}$ over $\tilde{S}$ satisfies (\ref{thm--klt-trivial-fib-start-(i)}). 
\end{enumerate}
In particular, the discriminant $\mathbb{Q}$-b-divisor and the moduli $\mathbb{Q}$-b-divisor associated to $f$ are compatible with the restriction over any open subset (see Lemma \ref{lem--restriction-can-bundle-formula}).  
\end{thm}

\begin{proof}
The first assertion follows from \cite[Theorem 21.4]{fujino-analytic-bchm} and Remark \ref{rem--descend-variety}. 
In the following discussion we will show 
(\ref{thm--klt-trivial-fib-start-(ii)}) only using (\ref{thm--klt-trivial-fib-start-(i)}) and Lemma \ref{lem--restriction-can-bundle-formula}. 

Let $\tilde{\boldsymbol{\rm N}}$ be the moduli $\mathbb{Q}$-b-divisor associated to $(\tilde{X}, \Delta|_{\tilde{X}}) \to \tilde{Z}$. 
Then it is enough to show that $\tilde{\boldsymbol{\rm N}}= \boldsymbol{\rm N}|_{\tilde{Z}}$. 
Pick any point $s \in \tilde{S}$. 
By (\ref{thm--klt-trivial-fib-start-(i)}), there exists an open neighborhood $S' \subset \tilde{S}$ of $s$ such that if we put $X' \subset \tilde{X}$ and $Z' \subset \tilde{Z}$ as the inverse images of $S'$, then the moduli $\mathbb{Q}$-b-divisor $\boldsymbol{\rm N}'$ associated to $(X', \Delta|_{X'}) \to Z'$ is b-nef$/S'$. 
Then, for every proper bimeromorphic morphism $\overline{Z} \to \tilde{Z}$ and the inverse image $\overline{Z}' \subset \overline{Z}$ over $S' \subset \tilde{S}$, we have the diagram
$$
\xymatrix{
 \overline{Z}' \ar@{_{(}->}[d]\ar[r]&Z'\ar@{_{(}->}[d]\ar[r]&S'\ar@{_{(}->}[d]\\
 \overline{Z} \ar[r]&\tilde{Z}\ar[r]&\tilde{S},\\
}
$$
and 
\begin{equation*}
\begin{split}
(\boldsymbol{\rm N}|_{Z'})_{\overline{Z}'}=\bigl((\boldsymbol{\rm N}|_{\tilde{Z}})_{\overline{Z}}\bigr)|_{\overline{Z}'} \leq  \bigl(\tilde{\boldsymbol{\rm N}}_{\overline{Z}}\bigr)|_{\overline{Z}'} = \boldsymbol{\rm N}'_{\overline{Z}'}=(\boldsymbol{\rm N}|_{Z'})_{\overline{Z}'}.
\end{split}
\end{equation*}
Here, the first equality immediately follows from Definition \ref{defn--rest-b-divisor}, and the other relations follow from Lemma \ref{lem--restriction-can-bundle-formula}. 
This shows $\bigl((\boldsymbol{\rm N}|_{\tilde{Z}})_{\overline{Z}}\bigr)|_{\overline{Z}'} =  \bigl(\tilde{\boldsymbol{\rm N}}_{\overline{Z}}\bigr)|_{\overline{Z}'}$. 
Since $s \in \tilde{S}$ is arbitrary and $\overline{Z}' \subset \overline{Z}$ is the inverse image of an open neighborhood $S'$ of $s$, we have 
$(\boldsymbol{\rm N}|_{\tilde{Z}})_{\overline{Z}} =  \tilde{\boldsymbol{\rm N}}_{\overline{Z}}$
for any proper bimeromorphic morphism $\overline{Z} \to \tilde{Z}$. 
Therefore, we have $\tilde{\boldsymbol{\rm N}}= \boldsymbol{\rm N}|_{\tilde{Z}}$, from which we see that Theorem \ref{thm--klt-trivial-fib-start} (\ref{thm--klt-trivial-fib-start-(ii)}) holds. 
\end{proof}

From here until Proposition \ref{prop--gen-can-bundle-formula-klt-Q-div}, we follow the strategy of \cite{filipazzi-gen-can-bundle-formula}. 

\begin{rem}[cf.~{\cite[Remark 4.8]{filipazzi-gen-can-bundle-formula}}]\label{rem--description-div-gen-can-bundle-formula}
Let $(X,B+M) \to S$ be a generalized sub-pair with the nef part $\boldsymbol{\rm M}$ such that  $B$ is a $\mathbb{Q}$-divisor on $X$ and $\boldsymbol{\rm M}$ is a $\mathbb{Q}$-b-divisor on $X$. 
Let $f\colon (X,B+M) \to Z$ be a generalized klt-trivial fibration over $S$. 
Let $W \subset S$ be a Stein compact subset.
We fix a $\mathbb{Q}$-Cartier divisor $D$ on $Z$ such that $K_{X}+B+M \sim_{\mathbb{Q}}f^{*}D$. 
Suppose that after shrinking $S$ around $W$, there exists a projective bimeromorphic morphism $\phi \colon X' \to X$ such that $\boldsymbol{\rm M}$ descends to $X'$ and $M':=\boldsymbol{\rm M}_{X'}$ is semi-ample over $Z$. 
Following \cite[Remark 4.8]{filipazzi-gen-can-bundle-formula}, we describe the discriminant $\mathbb{Q}$-b-divisor and the moduli $\mathbb{Q}$-b-divisor of $f$ over an open neighborhood of $W$ by using the corresponding $\mathbb{Q}$-b-divisors of other klt-trivial fibrations. 

By \cite[Lemma 2.16]{fujino-analytic-bchm}, we get
$$W \subset S^{\circ} \subset W^{\circ} \subset S,$$
where $S^{\circ}\subset S$ is a Stein open subset and $W^{\circ}\subset S$ is a Stein compact subset. 
We fix a Cartier divisor $H$ on $Z$ that is ample over $S$. 
As in \cite[Proof of Proposition 4.7]{filipazzi-gen-can-bundle-formula} (see also \cite[Theorem 4.10]{fujino-analytic-bchm}), for any $\epsilon \in \mathbb{Q}_{>0}$, there exist a Stein open neighborhood $S_{\epsilon}$ of $W^{\circ}$ and $m_{\epsilon} \in \mathbb{Z}_{>0}$ such that $m_{\epsilon}(M'+\epsilon (f \circ \phi)^{*}H)|_{X'_{\epsilon}}$ is Cartier and base point free over $S_{\epsilon}$, where $X'_{\epsilon} \subset X'$ is the inverse image of $S_{\epsilon}$. 
By replacing $S$ with $S^{\circ}$ and replacing $(X,B+M) \to S$ and $f \colon X \to Z$ accordingly, we may assume that $m_{\epsilon}(M'+\epsilon (f \circ \phi)^{*}H)$ is Cartier and base point free over $S$. 

By \cite[Lemma 2.16]{fujino-analytic-bchm}, we get
$$W \subset \tilde{S} \subset \tilde{W} \subset S,$$
where $\tilde{S}\subset S$ is a Stein open subset and $\tilde{W}\subset S$ is a Stein compact subset. 
We define a $\mathbb{Q}$-divisor $B'$ on $X'$ by 
$$K_{X'}+B'+M' = \phi^{*}(K_{X}+B+M).$$ 
For any $\epsilon \in \mathbb{Q}_{>0}$, we consider every element $A' \in |M'+\epsilon(f \circ \phi)^{*}H/S|_{\mathbb{Q}}$ that defines a klt-trivial fibration $(X',B'+A') \to Z$ over $S$ such that 
$$K_{X'}+B'+A' \sim_{\mathbb{Q}}(f \circ \phi)^{*}(D+\epsilon H).$$ 
Such an $A'$  exists because $m_{\epsilon}(M'+\epsilon (f \circ \phi)^{*}H)$ is base point free over $S$. 
Let $\tilde{X}' \subset X'$ and $\tilde{Z} \subset Z$ be the inverse images of $\tilde{S}$, and let $\tilde{\boldsymbol{\rm G}}^{\underset{A'}{}}_{\epsilon}$ (resp.~$\tilde{\boldsymbol{\rm N}}^{\underset{A'}{}}_{\epsilon}$) be the discriminant $\mathbb{Q}$-b-divisor (resp.~the moduli $\mathbb{Q}$-b-divisor) associated to $(\tilde{X}',B'|_{\tilde{X}'}+A'|_{\tilde{X}'}) \to \tilde{Z}$. 
By Theorem \ref{thm--klt-trivial-fib-start} (\ref{thm--klt-trivial-fib-start-(ii)}) for $\tilde{S} $ after applying Theorem \ref{thm--klt-trivial-fib-start} (\ref{thm--klt-trivial-fib-start-(i)}) to $(X',B'+A') \to Z$ and $\tilde{W} \subset S$, we see that $\tilde{\boldsymbol{\rm N}}^{\underset{A'}{}}_{\epsilon}$ is a b-nef$/\tilde{S}$ $\mathbb{Q}$-b-Cartier $\mathbb{Q}$-b-divisor on $\tilde{Z}$. 

Now we are ready to carry out the same argument as in the algebraic setting. 
Let $\tilde{\boldsymbol{\rm G}}$ and $\tilde{\boldsymbol{\rm N}}$ (resp.~$\tilde{\boldsymbol{\rm G}}_{\epsilon}$ and $\tilde{\boldsymbol{\rm N}}_{\epsilon}$) be the discriminant $\mathbb{Q}$-b-divisor and the moduli $\mathbb{Q}$-b-divisor associated to $(\tilde{X}',B'|_{\tilde{X}'}+M'|_{\tilde{X}'}) \to \tilde{Z}$ (resp.~$\bigl(\tilde{X}',B'|_{\tilde{X}'}+(M'+\epsilon(f \circ \phi)^{*}H)|_{\tilde{X}'}\bigr) \to \tilde{Z}$). 
By the same argument as in \cite[Remark 4.8]{filipazzi-gen-can-bundle-formula}, we have
$$\tilde{\boldsymbol{\rm G}}=\tilde{\boldsymbol{\rm G}}_{\epsilon}=\underset{A'}{\rm inf }\,\tilde{\boldsymbol{\rm G}}^{\underset{A'}{}}_{\epsilon}$$
for any $\epsilon \in \mathbb{Q}_{>0}$, and 
$$\tilde{\boldsymbol{\rm N}}=\underset{\epsilon \to 0}{\rm lim}\,\tilde{\boldsymbol{\rm N}}_{\epsilon}=\underset{\epsilon \to 0}{\rm lim}\,\left(\underset{A'}{\rm sup }\,\tilde{\boldsymbol{\rm N}}^{\underset{A'}{}}_{\epsilon}\right),$$
where $A'$ runs over elements of $ |M'+\epsilon(f \circ \phi)^{*}H/S|_{\mathbb{Q}}$ such that $(X',B'+A') \to Z$ is a klt-trivial fibration over $S$. 
Here, we used the base point freeness of $m_{\epsilon}(M'+\epsilon (f \circ \phi)^{*}H)$ over $S$ for the final equalities.  
Recall that $\tilde{\boldsymbol{\rm G}}^{\underset{A'}{}}_{\epsilon}$ and $\tilde{\boldsymbol{\rm N}}^{\underset{A'}{}}_{\epsilon}$ are the discriminant $\mathbb{Q}$-b-divisor and the moduli  $\mathbb{Q}$-b-divisor associated to $(\tilde{X}',B'|_{\tilde{X}'}+A'|_{\tilde{X}'}) \to \tilde{Z}$, respectively, and $\tilde{\boldsymbol{\rm N}}^{\underset{A'}{}}_{\epsilon}$ is a b-nef$/\tilde{S}$ $\mathbb{Q}$-b-Cartier $\mathbb{Q}$-b-divisor for every $\epsilon \in \mathbb{Q}_{>0}$ and $A'$. 
Hence, $\tilde{\boldsymbol{\rm N}}_{\epsilon}$ is almost b-nef$/\tilde{S}$ (\cite[Definition 4.9, Proposition 4.10]{filipazzi-gen-can-bundle-formula}). 
Moreover, the arguments in \cite{filipazzi-gen-can-bundle-formula} are valid for the generalized klt-trivial fibration $(\tilde{X}',B'|_{\tilde{X}'}+M'|_{\tilde{X}'}) \to \tilde{Z}$ over $\tilde{S}$.  
\end{rem}

\begin{prop}[cf.~{\cite[Theorem 4.12]{filipazzi-gen-can-bundle-formula}}]\label{prop--gen-can-bundle-formula-semiample}
Let $\pi \colon (X,B+M) \to S$ be a generalized sub-pair with the nef part $\boldsymbol{\rm M}$, let  $f\colon (X,B+M) \to Z$ be a generalized klt-trivial fibration over $S$, and let $W \subset S$ be a Stein compact subset. 
Suppose that
\begin{itemize}
\item
$B$ is a $\mathbb{Q}$-divisor on $X$ and $\boldsymbol{\rm M}$ is a $\mathbb{Q}$-b-divisor on $X$, 
\item
there exists a Zariski open dense subset $U \subset Z$ such that $B|_{f^{-1}(U)}$ is effective, and
\item
after shrinking $S$ around $W$, there exists a resolution $X' \to X$ of $X$ such that $\boldsymbol{\rm M}$ descends to $X'$ and $M':=\boldsymbol{\rm M}_{X'}$ is semi-ample over $Z$. 
\end{itemize}
Then, after shrinking $S$ around $W$ suitably, we have the following.
Let $\boldsymbol{\rm G}$ (resp.~$\boldsymbol{\rm N}$) be the discriminant $\mathbb{Q}$-b-divisor (resp.~the moduli $\mathbb{Q}$-b-divisor) associated to $f$. 
Then $\boldsymbol{\rm N}$ is a b-nef$/S$ $\mathbb{Q}$-b-Cartier $\mathbb{Q}$-b-divisor on $Z$. 
In particular, $(Z, \boldsymbol{\rm G}_{Z} +\boldsymbol{\rm N}_{Z}) \to S$ is a generalized sub-pair with the nef part $\boldsymbol{\rm N}$. 
Furthermore, if $(X,B+M) \to S$ is generalized klt, then so is $(Z, \boldsymbol{\rm G}_{Z} +\boldsymbol{\rm N}_{Z}) \to S$. 
\end{prop}

\begin{proof}
By using Remark \ref{rem--description-div-gen-can-bundle-formula} and the weak semistable reduction that is constructed by using \cite[Theorem 9.5]{karu-phdthesis} and \cite[Subsection 4.5]{eh-semistablereduction} after shrinking $S$, we may carry out \cite[Proof of Theorem 4.12]{filipazzi-gen-can-bundle-formula} with no changes. 
\end{proof}

\begin{prop}\label{prop--gen-can-bundle-formula-klt-Q-div}
Let $\pi \colon (X,B+M) \to S$ be a generalized sub-pair with the nef part $\boldsymbol{\rm M}$, and let $f\colon (X,B+M) \to Z$ be a generalized klt-trivial fibration over $S$. 
Let $W \subset S$ be a compact subset such that $\pi \colon X \to S$ and $W$ satisfy (P). 
Suppose that
\begin{itemize}
\item
$B$ is a $\mathbb{Q}$-divisor on $X$ and $\boldsymbol{\rm M}$ is a $\mathbb{Q}$-b-divisor on $X$, and
\item
there exists a Zariski open dense subset $U \subset Z$ such that $B|_{f^{-1}(U)}$ is effective.
\end{itemize}
Then, after replacing $S$ with a suitable open neighborhood of $W$, we have the following properties.
\begin{enumerate}[(i)]
\item \label{prop--gen-can-bundle-formula-klt-Q-div-(i)}
Let $\boldsymbol{\rm G}$ (resp.~$\boldsymbol{\rm N}$) be the discriminant $\mathbb{Q}$-b-divisor (resp.~the moduli $\mathbb{Q}$-b-divisor) associated to $f$. 
Then $\boldsymbol{\rm N}$ is a b-nef$/S$ $\mathbb{Q}$-b-Cartier $\mathbb{Q}$-b-divisor on $Z$. 
In particular, $(Z, \boldsymbol{\rm G}_{Z} +\boldsymbol{\rm N}_{Z}) \to S$ is a generalized sub-pair with the nef part $\boldsymbol{\rm N}$. 
Furthermore, if $(X,B+M) \to S$ is generalized klt, then so is $(Z, \boldsymbol{\rm G}_{Z} +\boldsymbol{\rm N}_{Z}) \to S$. 
\item \label{prop--gen-can-bundle-formula-klt-Q-div-(ii)}
For any open subset $\tilde{S} \subset S$ that does not necessarily contain $W$, if we put $\tilde{X} \subset X$ and $\tilde{Z} \subset Z$ as the inverse images of $\tilde{S}$, then the generalized klt-trivial fibration $(\tilde{X}, B|_{\tilde{X}}+M|_{\tilde{X}}) \to \tilde{Z}$ over $\tilde{S}$ satisfies (\ref{prop--gen-can-bundle-formula-klt-Q-div-(i)}). 
\end{enumerate}
In particular, the discriminant $\mathbb{Q}$-b-divisor and the moduli $\mathbb{Q}$-b-divisor associated to $f$ are compatible with the restriction over any open subset (see Lemma \ref{lem--restriction-can-bundle-formula}). 
\end{prop}

\begin{proof}
We only prove the first assertion since the second assertion follows from the first assertion and the same argument as in the proof of Theorem \ref{thm--klt-trivial-fib-start} (\ref{thm--klt-trivial-fib-start-(ii)}). 
We can apply the idea of \cite[Proof of Theorem 1.4]{filipazzi-gen-can-bundle-formula} to our situation since we may use \cite{fujino-analytic-bchm} and \cite{eh-semistablereduction} to reduce Proposition \ref{prop--gen-can-bundle-formula-klt-Q-div} to the situation of Theorem \ref{thm--klt-trivial-fib-start} or  Proposition \ref{prop--gen-can-bundle-formula-semiample}. 
However, 
we write the detailed argument because we will apply a similar argument in the proof of Theorem \ref{thm--gen-can-bundle-formula-full} later. 

We fix a $\mathbb{Q}$-Cartier divisor $D$ on $Z$ such that $K_{X}+B+M \sim_{\mathbb{Q}}f^{*}D$. 
We divide the proof into several steps. 

\begin{step3}\label{step1--prop--gen-can-bundle-formula-klt-Q-div}
In this step we reduce the proposition to the case where $X$ is $\mathbb{Q}$-factorial over $W$ and $(X,B+M) \to S$ is generalized dlt. 

By the equi-dimensional reduction (\cite{eh-semistablereduction}), we obtain the diagram
 $$
\xymatrix{
X \ar@{->}[d]_{f}& \ar@{->}[l]_{\phi}X' \ar@{->}[d]^{f'}\\
Z &Z', \ar[l]^{\varphi}
}
$$
where $X'$ is a normal analytic variety that is $\mathbb{Q}$-factorial over $W$, such that
\begin{itemize}
\item
$\boldsymbol{\rm M}$ descends to $X'$, 
\item
 the pair $(X',0)$ is klt and there exists a reduced divisor $\Sigma_{X'}$ on $X'$ such that  $(X',\Sigma_{X'})$ is an lc pair and $\Sigma_{X'}\supset {\rm Supp}\,\phi^{-1}_{*}B\cup {\rm Ex}(\phi)$, and
\item
all fibers of $f'$ have the same dimension. 
\end{itemize}
Let $B_{U}$ be a $\mathbb{Q}$-divisor on $X$ that is the sum of all the components of $B$ intersecting $f^{-1}(U)$. 
Note that the coefficients of $B_{U}$ belong to $[0,1)$. 
We may write
$$K_{X'}+B'+\boldsymbol{\rm M}_{X'}=\phi^{*}(K_{X}+B+M)+E',$$
where $B'$ is the sum of $\phi^{-1}_{*}B_{U}$ and the reduced $\phi$-exceptional divisor on $X'$. 
Note that $E'$ may not be effective or $\phi$-exceptional because $(X,B+M) \to S$ is not necessarily a generalized lc pair, whereas $E'|_{(f \circ \phi)^{-1}(U)}$ is effective and exceptional over $f^{-1}(U)$. 
By construction, $(X',B'+\boldsymbol{\rm M}_{X'}) \to S$ is generalized lc. 

Let $H_{Z'}$ be an ample Cartier divisor on $Z'$, and we set $H':=3 \cdot ({\rm dim}\,X) f'^{*}H_{Z'}$. 
We run a $(K_{X'}+B'+\boldsymbol{\rm M}_{X'}+H')$-MMP over $S$ around $W$ with scaling of an ample $\mathbb{R}$-divisor. 
By the length of extremal rays (\cite[Theorem 1.1.6 (5)]{fujino-analytic-conethm}), the induced meromorphic map $X'_{i} \dashrightarrow Z'$ from any variety $X'_{i}$ appearing in this MMP is a contraction morphism. 
By the same argument as in \cite[Proof of Lemma 5.4]{eh-analytic-mmp}, after shrinking $S$ around $W$, we obtain a bimeromorphic contraction 
$$\psi \colon X' \dashrightarrow X''$$
over $Z'$ and the induced contraction $f'' \colon X'' \to Z'$ such that if we put $B'':=\psi_{*}B'$ and $M'':=\psi_{*}\boldsymbol{\rm M}_{X'}$, then $K_{X''}+B''+M''\sim_{\mathbb{Q},\,Z'}\psi_{*}E'$ is the limit of movable divisors over $Z'$ and $\psi_{*}E'|_{(\varphi \circ f'')^{-1}(U)}=0$. 
By the equi-dimensionality of $f'\colon X' \to Z'$, we can find a $\mathbb{Q}$-Cartier divisor $L'$ on $Z'$ such that $F'':=\psi_{*}E'-f''^{*}L'$ is an effective $\mathbb{Q}$-divisor on $X''$ that is very exceptional over $Z'$.  
We apply the negativity lemma for very exceptional divisors (\cite[Corollary 2.15]{eh-analytic-mmp}) to $F''$ and $f''$, and we see that $F'' =0$. 
Therefore, we have $\psi_{*}E'=f''^{*}L'$ as $\mathbb{Q}$-divisors. 
Then
\begin{equation*}
\begin{split}
K_{X''}+B''+M''=&\psi_{*}\phi^{*}(K_{X}+B+M)+\psi_{*}E'\\
\sim_{\mathbb{Q}}& \psi_{*}(f \circ \phi)^{*}D+\psi_{*}E'
=f''^{*}(\varphi^{*}D+L').
\end{split}
\end{equation*}
By construction, we can check that $f'' \colon (X'',B''+M'') \to Z'$ is a generalized klt-trivial fibration over $S$ whose moduli $\mathbb{Q}$-b-divisor coincides with $\boldsymbol{\rm N}$. 
Therefore, we may replace $f \colon (X,B+M) \to Z$ by $f'' \colon (X'',B''+M'') \to Z'$. 
By this replacement and taking a generalized dlt-blow-up as in Theorem \ref{thm--gen-dlt-model}, we may assume that $X$ is $\mathbb{Q}$-factorial over $W$ and $(X,B+M) \to S$ is generalized dlt. 
\end{step3}

By shrinking $S$ around $W$, we may assume that $(X,B)$ is a dlt pair. 

\begin{step3}\label{step2--prop--gen-can-bundle-formula-klt-Q-div}
When $K_{X}+B$ is not pseudo-effective over $Z$, we may apply \cite{fujino-analytic-bchm} and the argument of \cite[Proof of Lemma 5.2 and Theorem 1.4]{filipazzi-gen-can-bundle-formula} to our situation, and we can reduce Proposition \ref{prop--gen-can-bundle-formula-klt-Q-div} to Proposition \ref{prop--gen-can-bundle-formula-semiample}. 
Therefore, we may assume that $K_{X}+B$ is pseudo-effective over $Z$. 
\end{step3}

\begin{step3}\label{step3--prop--gen-can-bundle-formula-klt-Q-div}
In this step we prove that we may assume $M \sim_{\mathbb{Q},\,Z}0$. 

Let $H_{Z}$ be an ample Cartier divisor on $Z$ and $H$ a general member of the $\mathbb{R}$-linear system $|3 \cdot ({\rm dim}\,X) f^{*}H_{Z}|_{\mathbb{R}}$ such that $(X,B+H)$ is a dlt pair whose lc centers are those  of $(X,B)$. 
We prove the existence of a finite sequence of steps of a $(K_{X}+B+H)$-MMP over $S$ around $W$
$$(X,B+H) \dashrightarrow (\tilde{X},\tilde{B}+\tilde{H})$$
such that the induced meromorphic map $\tilde{X} \dashrightarrow Z$ is a morphism and the divisor $K_{\tilde{X}}+\tilde{B}$ is semi-ample over $Z$. 
To prove this assertion, we show that every dlt pair appearing in a $(K_{X}+B+H)$-MMP over $S$ around $W$ satisfies the property of being log abundant over $S$ around $W$ (\cite[Definition 2.23]{eh-analytic-mmp}), and then we apply \cite[Theorem 1.3]{eh-analytic-mmp}, \cite[Theorem 1.5]{fujino-analytic-lcabundance}, and \cite[Theorem 1.1]{eh-analytic-mmp-2}.

Let $(X,B+H) \dashrightarrow (X_{i},B_{i}+H_{i})$ be a finite sequence of steps of a $(K_{X}+B+H)$-MMP over $S$ around $W$. 
We may assume that this MMP is represented by bimeromorphic contractions over $S$. 
By the length of extremal rays (\cite[Theorem 1.1.6 (5)]{fujino-analytic-conethm}), the induced meromorphic map $X_{i} \dashrightarrow Z$ is a contraction morphism. 
Let $T_{i}$ be $X_{i}$ or an lc center of $(X_{i},B_{i}+H_{i})$, and let $T$ be $X$ or the corresponding lc center of $(X,B+H)$ admitting an induced bimeromorphic map $T \dashrightarrow T_{i}$. 
Let $(T,\Delta_{T})$ and $(T,B_{T}+M_{T})$ be (generalized) dlt pairs defined by applying (generalized) adjunction to $(X, B)$ and $(X,B+M)$, respectively.  
By considering a contribution of $M$ to the singularities of generalized pairs, we have $\Delta_{T}\leq B_{T}$. 
Since $K_{T}+B_{T}+M_{T} \sim_{\mathbb{Q},\,Z}0$, we see that 
$$-(K_{X}+B)|_{T}=-(K_{T}+\Delta_{T})\sim_{\mathbb{Q},\,Z}B_{T}-\Delta_{T}+M_{T}$$
is pseudo-effective over $Z$.
This fact and the definition of a $(K_{X}+B+H)$-MMP imply the inequalities
$$\kappa_{\sigma}(T_{i}/Z, (K_{X_{i}}+B_{i})|_{T_{i}}) \leq \kappa_{\sigma}(T/Z, (K_{X}+B)|_{T}) \leq 0$$
of the relative numerical dimensions (\cite[Definition 2.22]{eh-analytic-mmp}). 
By \cite{gongyo1} and restricting $(K_{X_{i}}+B_{i})|_{T_{i}}$ to an analytically sufficiently general fiber of the Stein factorization of $T_{i} \to Z$, we see that $(K_{X_{i}}+B_{i})|_{T_{i}}$ is abundant over $Z$ (see \cite[Definition 2.23]{eh-analytic-mmp}). 
By \cite[Lemma 2.19]{has-finite} and restricting $(K_{X_{i}}+B_{i}+H_{i})|_{T_{i}}$ to an analytically sufficiently general fiber of the Stein factorization of $T_{i} \to S$, we see that $(K_{X_{i}}+B_{i}+H_{i})|_{T_{i}}$ is abundant over $S$. 
Hence, $(X_{i},B_{i}+H_{i})$ is a log abundant lc pair over $S$ around $W$. 
This fact shows that every dlt pair appearing in a $(K_{X}+B+H)$-MMP over $S$ around $W$ satisfies the property of being log abundant over $S$ around $W$. 
As stated above, applying \cite[Theorem 1.3]{eh-analytic-mmp}, \cite[Theorem 1.5]{fujino-analytic-lcabundance}, and \cite[Theorem 1.1]{eh-analytic-mmp-2}, we obtain a finite sequence of steps of a $(K_{X}+B+H)$-MMP over $S$ around $W$
$$(X,B+H) \dashrightarrow (\tilde{X},\tilde{B}+\tilde{H})$$
such that after shrinking $S$ around $W$ suitably, $K_{\tilde{X}}+\tilde{B}+\tilde{H}$ is semi-ample over $S$. 
By the length of extremal rays (\cite[Theorem 1.1.6 (5)]{fujino-analytic-conethm}), the induced meromorphic map $\tilde{X} \dashrightarrow Z$ is a contraction morphism. 
Then $K_{\tilde{X}}+\tilde{B}+\tilde{H} \sim_{\mathbb{R},\,Z}K_{\tilde{X}}+\tilde{B}$, and therefore $K_{\tilde{X}}+\tilde{B}$ is semi-ample over $Z$. 
Therefore, the map $(X,B+H) \dashrightarrow (\tilde{X},\tilde{B}+\tilde{H})$ is what we wanted. 

Let $\tilde{f} \colon \tilde{X} \to \tilde{Z}$ be the contraction over $Z$ induced by $K_{\tilde{X}}+\tilde{B}$. 
Then it is easy to see that we may replace $f\colon (X,B+M) \to Z$ by $\tilde{f} \colon (\tilde{X},\tilde{B}+M_{\tilde{X}}) \to \tilde{Z}$. 
By this replacement, we may in particular assume that $M \sim_{\mathbb{Q},\,Z}-(K_{X}+B) \sim_{\mathbb{Q},\,Z}0$. 
\end{step3}

\begin{step3}\label{step4--prop--gen-can-bundle-formula-klt-Q-div}
Finally we use the argument in \cite[Proof of Theorem 1.4]{filipazzi-gen-can-bundle-formula}. 
Since we have $M \sim_{\mathbb{Q},\,Z}0$, we can avoid using the notion of the generic fiber in our setting. 
By \cite{eh-semistablereduction}, we can reduce Proposition \ref{prop--gen-can-bundle-formula-klt-Q-div} to the situation of Theorem \ref{thm--klt-trivial-fib-start}. 
Hence, Proposition \ref{prop--gen-can-bundle-formula-klt-Q-div} holds. 
\end{step3}

We finish the proof. 
\end{proof}

We extend Proposition \ref{prop--gen-can-bundle-formula-klt-Q-div} to generalized lc-trivial fibrations.   

\begin{prop}[cf.~{\cite[Theorem 1.4]{filipazzi-gen-can-bundle-formula}}]\label{prop--gen-can-bundle-formula-lc-Q-div}
Let $\pi \colon (X,B+M) \to S$ be a generalized sub-pair with the nef part $\boldsymbol{\rm M}$, let $f\colon (X,B+M) \to Z$ be a generalized lc-trivial fibration over $S$, and let $W \subset S$ be a compact subset such that $\pi \colon X \to S$ and $W$ satisfy (P). 
Suppose that
\begin{itemize}
\item
$B$ is a $\mathbb{Q}$-divisor on $X$ and $\boldsymbol{\rm M}$ is a $\mathbb{Q}$-b-divisor on $X$, and
\item
there exists a Zariski open dense subset $U \subset Z$ such that $B|_{f^{-1}(U)}$ is effective.
\end{itemize}
Then, after replacing $S$ with a suitable open neighborhood of $W$, we have the following properties.
\begin{enumerate}[(i)]
\item \label{prop--gen-can-bundle-formula-lc-Q-div-(i)}
Let $\boldsymbol{\rm G}$ (resp.~$\boldsymbol{\rm N}$) be the discriminant $\mathbb{Q}$-b-divisor (resp.~the moduli $\mathbb{Q}$-b-divisor) associated to $f$. 
Then $\boldsymbol{\rm N}$ is a b-nef$/S$ $\mathbb{Q}$-b-Cartier $\mathbb{Q}$-b-divisor on $Z$. 
In particular, $(Z, \boldsymbol{\rm G}_{Z} +\boldsymbol{\rm N}_{Z}) \to S$ is a generalized sub-pair with the nef part $\boldsymbol{\rm N}$. 
Furthermore, if $(X,B+M) \to S$ is generalized lc, then so is $(Z, \boldsymbol{\rm G}_{Z} +\boldsymbol{\rm N}_{Z}) \to S$. 
\item \label{prop--gen-can-bundle-formula-lc-Q-div-(ii)}
For any open subset $\tilde{S} \subset S$ that does not necessarily contain $W$, if we put $\tilde{X} \subset X$ and $\tilde{Z} \subset Z$ as the inverse images of $\tilde{S}$, then the generalized lc-trivial fibration $(\tilde{X}, B|_{\tilde{X}}+M|_{\tilde{X}}) \to \tilde{Z}$ over $\tilde{S}$ satisfies (\ref{prop--gen-can-bundle-formula-lc-Q-div-(i)}). 
\end{enumerate}
In particular, the discriminant $\mathbb{Q}$-b-divisor and the moduli $\mathbb{Q}$-b-divisor associated to $f$ are compatible with the restriction over any open subset.   
\end{prop}

\begin{proof}
We fix a $\mathbb{Q}$-Cartier divisor $D$ on $Z$ such that $K_{X}+B+M \sim_{\mathbb{Q}}f^{*}D$. 
By the same argument as in Step \ref{step1--prop--gen-can-bundle-formula-klt-Q-div} in the proof of Proposition \ref{prop--gen-can-bundle-formula-klt-Q-div}, we may assume that $X$ is $\mathbb{Q}$-factorial over $W$ and $(X,B+M) \to S$ is generalized dlt. 
By shrinking $S$ around $W$ suitably, we may assume that $(X, \{B\}+M) \to S$ is generalized klt. 

We divide the proof into two steps. 

\begin{step2}\label{step2--prop-gen-can-bundle-formula-lc-Q-div}
In this step, we prove that we may assume $\lfloor B \rfloor$ is $f$-ample. 

If $f(\lfloor B \rfloor)\subsetneq Z$, then $f \colon (X,B+M) \to Z$ is a generalized klt-trivial fibration over $S$. 
In this case, Proposition \ref{prop--gen-can-bundle-formula-lc-Q-div} is nothing but Proposition \ref{prop--gen-can-bundle-formula-klt-Q-div} and there is nothing to prove. 
Therefore, we may assume that $f(\lfloor B \rfloor)=Z$. 
We note that this condition shows ${\rm dim}\,X > {\rm dim}\,Z$. 
Let $H_{Z}$ be a Cartier divisor on $Z$ that is ample over $S$, and we set $H:=3 \cdot ({\rm dim}\,X) f^{*}H_{Z}$. 
We run a $(K_{X}+\{B\}+M+H)$-MMP over $S$ with scaling of an ample $\mathbb{R}$-divisor. 
By the length of extremal rays (\cite[Theorem 1.1.6 (5)]{fujino-analytic-conethm}), the induced meromorphic map $X_{i} \dashrightarrow Z$ from any variety $X_{i}$ appearing in this MMP is a morphism. 
By shrinking $S$ around $W$ and replacing $X$, we may assume that $X$ has the structure of a Mori fiber space $g \colon X \to Y$ over $S$ for $K_{X}+\{B\}+M$ such that the induced meromorphic map $Y \dashrightarrow Z$ is a morphism. 
Then ${\rm dim}\,X>{\rm dim}\,Y$, and
$$\lfloor B \rfloor = K_{X}+B+M-(K_{X}+\{B\}+M)\sim_{\mathbb{Q},\,Z} -(K_{X}+\{B\}+M)$$ 
is $g$-ample. 

By \cite[Lemma 5.1]{filipazzi-gen-can-bundle-formula} and an induction on the dimension of $X$, it is sufficient to prove Proposition \ref{prop--gen-can-bundle-formula-lc-Q-div} for $g \colon (X,B+M) \to Y$. 
Indeed, assume that Proposition \ref{prop--gen-can-bundle-formula-lc-Q-div} holds for $g \colon (X,B+M) \to Y$. 
Let $\boldsymbol{\rm G}'$ (resp.~$\boldsymbol{\rm N}'$) be the discriminant $\mathbb{Q}$-b-divisor (resp.~the moduli $\mathbb{Q}$-b-divisor) associated to $g$. 
Then $(Y,\boldsymbol{\rm G}'_{Y}+\boldsymbol{\rm N}'_{Y}) \to S$ is a generalized lc pair with the nef part $\boldsymbol{\rm N}'$, and 
$(Y,\boldsymbol{\rm G}'_{Y}+\boldsymbol{\rm N}'_{Y}) \to Z$ defines the structure of a generalized lc-trivial fibration over $S$. 
By applying the induction hypothesis of Proposition \ref{prop--gen-can-bundle-formula-lc-Q-div} to $(Y,\boldsymbol{\rm G}'_{Y}+\boldsymbol{\rm N}'_{Y}) \to Z$, we obtain the discriminant $\mathbb{Q}$-b-divisor and the moduli $\mathbb{Q}$-b-divisor. 
Then \cite[Lemma 5.1]{filipazzi-gen-can-bundle-formula} shows that they are equal to the discriminant $\mathbb{Q}$-b-divisor $\boldsymbol{\rm G}$ and the moduli $\mathbb{Q}$-b-divisor $\boldsymbol{\rm N}$ associated to $f\colon (X,B+M) \to Z$, respectively. 
From this discussion, we may replace $f\colon (X,B+M) \to Z$ by $g \colon (X,B+M) \to Y$, and we may assume that $\lfloor B \rfloor$ is $f$-ample. 
\end{step2}

\begin{step2}\label{step3--prop-gen-can-bundle-formula-lc-Q-div}
By shrinking $S$ around $W$, we can find an ample $\mathbb{Q}$-Cartier divisor $A$ on $Z$ such that $\lfloor B \rfloor+f^{*}A$ is ample over $S$. 
By \cite[Lemma 2.16]{fujino-analytic-bchm}, we get
$$W \subset S^{\circ} \subset W^{\circ} \subset S,$$
where $S^{\circ} \subset S$ is a Stein open subset and $W^{\circ} \subset S$ is a Stein compact subset such that $X \to S$ and $W^{\circ} \subset S$ satisfy (P). 
In this step we show that Proposition \ref{prop--gen-can-bundle-formula-lc-Q-div} holds if we shrink $S$ to $S^{\circ}$. 

For each positive integer $m$, we put $B_{m}:=B-\frac{1}{m}\lfloor B \rfloor$, $\boldsymbol{\rm L}_{m}:=\boldsymbol{\rm M}+\frac{1}{m}\overline{(\lfloor B \rfloor+f^{*}A)}$, and $L_{m}:=(\boldsymbol{\rm L}_{m})_{X}$. 
Then $(X,B_{m}+L_{m}) \to Z$ is a generalized klt-trivial fibration over $S$ for every $m \in \mathbb{Z}_{>0}$. 
Let $\boldsymbol{\rm G}_{m}$ (resp.~$\boldsymbol{\rm N}_{m}$) be the discriminant $\mathbb{Q}$-b-divisor (resp.~the moduli $\mathbb{Q}$-b-divisor) associated to $(X,B_{m}+L_{m}) \to Z$. 
Since 
$$K_{X}+B_{m}+L_{m}=K_{X}+B+M+\frac{1}{m}f^{*}A\sim_{\mathbb{Q}}f^{*}(D+\frac{1}{m}A),$$
we have $$\boldsymbol{\rm G}=\underset{m \to \infty}{\rm lim}\boldsymbol{\rm G}_{m} \qquad {\rm and} \qquad \boldsymbol{\rm N}=\underset{m \to \infty}{\rm lim}\boldsymbol{\rm N}_{m},$$where the limit means that
every prime divisor $P$ over $Z$ satisfies 
$${\rm coeff}_{P}(\boldsymbol{\rm G})=\underset{m \to \infty}{\rm lim}{\rm coeff}_{P}(\boldsymbol{\rm G}_{m}) \qquad {\rm and} \qquad {\rm coeff}_{P}(\boldsymbol{\rm N})=\underset{m \to \infty}{\rm lim}{\rm coeff}_{P}(\boldsymbol{\rm N}_{m}).$$
Furthermore, we have ${\rm Supp}\,B_{m}\subset {\rm Supp}\,B$ for every $m \in \mathbb{Z}_{>0}$. 
By Proposition \ref{prop--gen-can-bundle-formula-klt-Q-div} (\ref{prop--gen-can-bundle-formula-klt-Q-div-(i)}) and Remark \ref{rem--descend-variety}, there exists a projective bimeromorphic morphism $\overline{Z} \to Z$ such that 
\begin{itemize}
\item
$\overline{Z}$ is smooth, and 
\item
for each $m \in \mathbb{Z}_{>0}$, there exists an open neighborhood $S_{m}$ of $W^{\circ}$ such that 
after restricting $(X,B_{m}+L_{m}) \to Z$ over $S_{m}$, the associated moduli $\mathbb{Q}$-b-divisor descends to $\overline{Z} $. 
\end{itemize}
By Proposition \ref{prop--gen-can-bundle-formula-klt-Q-div} (\ref{prop--gen-can-bundle-formula-klt-Q-div-(ii)}) and replacing $S$ by $S^{\circ}$, we may assume that all $\boldsymbol{\rm N}_{m}$ descend to $\overline{Z} $. 
Then, $\boldsymbol{\rm N}_{\overline{Z}}$ is $\mathbb{Q}$-Cartier and this is the limit of $(\boldsymbol{\rm N}_{m})_{\overline{Z}}$, which are nef over $S$. 
Hence, $\boldsymbol{\rm N}_{\overline{Z}}$ is nef over $S$. 
Moreover, for any proper bimeromorphic morphism $\tau \colon \widetilde{Z} \to \overline{Z}$, we have
$$\boldsymbol{\rm N}_{\widetilde{Z}}-\tau^{*}\boldsymbol{\rm N}_{\overline{Z}}=\underset{m \to \infty}{\rm lim}\bigl((\boldsymbol{\rm N}_{m})_{\widetilde{Z}}-\tau^{*}((\boldsymbol{\rm N}_{m})_{\overline{Z}})\bigr)=0.$$
Therefore, $\boldsymbol{\rm N}$ descends to $\overline{Z}$. 
In particular, $\boldsymbol{\rm N}$ is b-nef$/S$ $\mathbb{Q}$-b-Cartier $\mathbb{Q}$-b-divisor. 
This shows Proposition \ref{prop--gen-can-bundle-formula-lc-Q-div} (\ref{prop--gen-can-bundle-formula-lc-Q-div-(i)}). 
We can easily check Proposition \ref{prop--gen-can-bundle-formula-lc-Q-div} (\ref{prop--gen-can-bundle-formula-lc-Q-div-(ii)}) by using the same argument as in the proof of Theorem \ref{thm--klt-trivial-fib-start} (\ref{thm--klt-trivial-fib-start-(ii)}). 
\end{step2}
As discussed in Step \ref{step2--prop-gen-can-bundle-formula-lc-Q-div}, we complete the proof. 
\end{proof}

Now we are ready to prove the main results of this section.

\begin{thm}\label{thm--gen-can-bundle-formula-lc}
Let $\pi \colon (X,B+M) \to S$ be a generalized sub-pair with the nef part $\boldsymbol{\rm M}$, and let $f\colon (X,B+M) \to Z$ be a generalized lc-trivial fibration over $S$. 
Let $W \subset S$ be a Stein compact subset. 
Suppose that
\begin{itemize}
\item
$\boldsymbol{\rm M}$ is a finite $\mathbb{R}_{>0}$-linear combination of b-nef$/S$ $\mathbb{Q}$-b-Cartier $\mathbb{Q}$-b-divisors on $X$, and
\item
there exists a Zariski open dense subset $U \subset Z$ such that $B|_{f^{-1}(U)}$ is effective.
\end{itemize}
Then, after replacing $S$ with a suitable open neighborhood of $W$, we have the following properties.
\begin{enumerate}[(i)]
\item \label{thm--gen-can-bundle-formula-lc-(i)}
Let $\boldsymbol{\rm G}$ (resp.~$\boldsymbol{\rm N}$) be the discriminant $\mathbb{R}$-b-divisor (resp.~the moduli $\mathbb{R}$-b-divisor) associated to $f$. 
Then $\boldsymbol{\rm N}$ is a finite $\mathbb{R}_{>0}$-linear combination of b-nef$/S$ $\mathbb{Q}$-b-Cartier $\mathbb{Q}$-b-divisors on $Z$. 
In particular, $(Z, \boldsymbol{\rm G}_{Z} +\boldsymbol{\rm N}_{Z}) \to S$ is a generalized sub-pair with the nef part $\boldsymbol{\rm N}$. 
Furthermore, if $(X,B+M) \to S$ is generalized lc (resp.~generalized klt), then $(Z, \boldsymbol{\rm G}_{Z} +\boldsymbol{\rm N}_{Z}) \to S$ is also generalized lc (resp.~generalized klt). 
\item \label{thm--gen-can-bundle-formula-lc-(ii)}
For any open subset $\tilde{S} \subset S$ that does not necessarily contain $W$, if we put $\tilde{X} \subset X$ and $\tilde{Z} \subset Z$ as the inverse images of $\tilde{S}$, then the generalized lc-trivial fibration $(\tilde{X}, B|_{\tilde{X}}+M|_{\tilde{X}}) \to \tilde{Z}$ over $\tilde{S}$ satisfies (\ref{thm--gen-can-bundle-formula-lc-(i)}). 
\end{enumerate}
In particular, the discriminant $\mathbb{R}$-b-divisor and the moduli $\mathbb{R}$-b-divisor associated to $f$ are compatible with the restriction over any open subset.  
\end{thm}

\begin{proof}
By Proposition \ref{prop--gen-can-bundle-formula-lc-Q-div} and Remark \ref{rem--descend-variety}, this theorem holds if $B$ is a $\mathbb{Q}$-divisor on $X$ and $\boldsymbol{\rm M}$ is a $\mathbb{Q}$-b-divisor on $X$. 
By the same argument as in Step \ref{step1--prop--gen-can-bundle-formula-klt-Q-div} in the proof of Proposition \ref{prop--gen-can-bundle-formula-klt-Q-div}, we may assume that $(X,B+M) \to S$ is generalized lc. 

We fix an $\mathbb{R}$-Cartier divisor $D$ on $Z$ such that $K_{X}+B+M \sim_{\mathbb{R}}f^{*}D$. 
Since $\boldsymbol{\rm M}$ is a finite $\mathbb{R}_{>0}$-linear combination of b-nef$/S$ $\mathbb{Q}$-b-Cartier $\mathbb{Q}$-b-divisors on $X$, we can write $\boldsymbol{\rm M}=\sum_{i=1}^{l} r_{i}\boldsymbol{\rm L}_{i}$ for some $r_{1},\,\cdots r_{l} \in \mathbb{R}_{>0}$ and b-nef$/S$ $\mathbb{Q}$-b-Cartier $\mathbb{Q}$-b-divisors $\boldsymbol{\rm L}_{1},\,\cdots,\,\boldsymbol{\rm L}_{l}$ on $X$.  
Then there exist a log resolution $\phi \colon X' \to X$ of $(X,B)$ and $\mathbb{Q}$-Cartier divisors $L'_{1},\,\cdots,\,L'_{l}$ on $X'$ such that $L'_{i}$ is nef over $S$ and $\boldsymbol{\rm L}_{i}=\overline{L'_{i}}$ for all $1 \leq i \leq l$. 
By shrinking $S$ around $W$ suitably, we may assume that all the $\mathbb{R}$-divisors appearing so far have  finitely many components and all appearing $\mathbb{R}$-Cartier divisors are globally $\mathbb{R}$-Cartier. 
By using \cite[Proof of Lemma 5.4]{fujino-hashizume-adjunction}, which is an argument from convex geometry, we obtain $\mathbb{Q}$-divisors $B_{1},\,\cdots, B_{m}$ on $X$, $\mathbb{Q}$-Cartier divisors $D_{1},\,\cdots,\,D_{m}$ on $Z$, $\mathbb{Q}$-Cartier divisors $M'_{1},\,\cdots,\,M'_{m}$ on $X'$ that are nef over $S$, and positive real numbers $\alpha_{1},\,\cdots,\,\alpha_{m}$ such that
\begin{itemize}
\item
$\sum_{j=1}^{m}\alpha_{j}=1$, $\sum_{j=1}^{m}\alpha_{j}B_{j}=B$, $\sum_{j=1}^{m}\alpha_{j}M'_{j}=\boldsymbol{\rm M}_{X'}$, and $\sum_{j=1}^{m}\alpha_{j}D_{j}=D$, 
\item
${\rm Supp}\,B_{j} = {\rm Supp}\,B$ and $B_{j}$ is effective for every $1 \leq j \leq m$, 
\item
if we put $\boldsymbol{\rm M}_{j}:=\overline{M'_{j}}$ and $M_{j}:=\phi_{*}M'_{j}=(\boldsymbol{\rm M}_{j})_{X}$ for each $1 \leq j \leq m$, then every $(X,B_{j}+M_{j}) \to S$ is a generalized sub-pair with the nef part $\boldsymbol{\rm M}_{j}$, 
\item
every $f \colon (X,B_{j}+M_{j}) \to Z$ is a generalized lc-trivial fibration over $S$ such that $K_{X}+B_{j}+M_{j} \sim_{\mathbb{Q}}f^{*}D_{j}$ for every $1\leq j \leq m$. 
\end{itemize}

We put
$$\mathcal{P}:=\left\{\boldsymbol{t}=(t_{1},\,\cdots, t_{m})\in (\mathbb{R}_{\geq 0})^{m}\,\middle| \,\sum_{j=1}^{m}t_{j}=1\right\}$$
and we define 
$$B_{\boldsymbol{t}}:=\sum_{j=1}^{m}t_{j}B_{j}, \qquad \boldsymbol{\rm M}_{\boldsymbol{t}}:=\sum_{j=1}^{m}t_{j}\boldsymbol{\rm M}_{j}, \qquad {\rm and} \qquad M_{\boldsymbol{t}}:=\sum_{j=1}^{m}t_{j}M_{j}=(\boldsymbol{\rm M}_{\boldsymbol{t}})_{X}$$
for $\boldsymbol{t} \in \mathcal{P}$. 
Then $f \colon (X, B_{\boldsymbol{t}}+M_{\boldsymbol{t}}) \to Z$ is a generalized lc-trivial fibration over $S$ such that $B_{\boldsymbol{t}}|_{f^{-1}(U)}$ is effective. 
Let $\boldsymbol{\rm G}_{\boldsymbol{t}}$ (resp.~$\boldsymbol{\rm N}_{\boldsymbol{t}}$) be the discriminant $\mathbb{R}$-b-divisor (resp.~the moduli $\mathbb{R}$-b-divisor) associated to $f \colon (X, B_{\boldsymbol{t}}+M_{\boldsymbol{t}}) \to Z$, where $\boldsymbol{\rm N}_{\boldsymbol{t}}$ is defined by using $\sum_{j=1}^{m}t_{j}D_{j}$. 
By \cite[Lemma 2.16]{fujino-analytic-bchm}, we obtain a Stein open subset $S^{\circ}\subset S$ and a Stein compact subset $W^{\circ}\subset S$ such that
$$W \subset S^{\circ} \subset W^{\circ} \subset S.$$
By Theorem \ref{thm--gen-can-bundle-formula-lc} (\ref{thm--gen-can-bundle-formula-lc-(i)}) in the rational case, for every $\boldsymbol{t} \in \mathcal{P} \cap {\mathbb{Q}}^{m}$, shrinking $S$ around $W^{\circ}$ suitably, $\boldsymbol{\rm N}_{\boldsymbol{t}}$ becomes a b-nef$/S$ $\mathbb{Q}$-b-Cartier $\mathbb{Q}$-b-divisor on $Z$. 
By Theorem \ref{thm--gen-can-bundle-formula-lc} (\ref{thm--gen-can-bundle-formula-lc-(ii)}) in the rational case and Lemma \ref{lem--restriction-can-bundle-formula} (\ref{lem--restriction-can-bundle-formula-(c)}), replacing $S$ by $S^{\circ}$ and restricting other objects over $S^{\circ}$, we may assume that $\boldsymbol{\rm N}_{\boldsymbol{t}}$ is a b-nef$/S$ $\mathbb{Q}$-b-Cartier $\mathbb{Q}$-b-divisor  on $Z$ for any $\boldsymbol{t} \in \mathcal{P} \cap {\mathbb{Q}}^{m}$. 
Note that this property is preserved after any further shrink of $S$. 

For each $1\leq j \leq m$, set $\boldsymbol{\rm G}_{j}:=\boldsymbol{\rm G}_{\boldsymbol{e}_{j}}$ for the $j$-th standard basis $\boldsymbol{e}_{j}$.  
In other words, $\boldsymbol{\rm G}_{j}$ denotes the discriminant $\mathbb{Q}$-b-divisor associated to $f \colon (X, B_{j}+M_{j}) \to Z$. 
 By using \cite[Theorem 1.17]{haconpaun} after shrinking $S$ around $W$ suitably, we can get a diagram
  $$
\xymatrix{
X^{\rm ws} \ar@{->}[r]^{\mu_{X}}\ar@{->}[d]_{f^{\rm ws}} & X^{\vee} \ar@{->}[r]^{\tau_{X}}\ar@{->}[d]^{f^{\vee}}&X \ar@{->}[d]^{f}\\
Z^{\rm ws}\ar[r]_{\mu}&Z^{\vee} \ar@{->}[r]_{\tau} &Z,
}
$$
where $\tau$ and $\tau_{X}$ are projective bimeromorphic morphisms, $\mu$ is finite and surjective, and $\mu_{X}$ is a generically finite surjective morphism, such that 
\begin{itemize}
\item
all $\boldsymbol{\rm M}_{j}$ descend to $X^{\vee}$, 
\item
$Z^{\vee}$ is smooth and there exists a simple normal crossing divisor $\Sigma^{\vee}$ on $Z^{\vee}$ that contains the union of ${\rm Ex}(\tau)$ and all ${\rm Supp}\,(\boldsymbol{\rm G}_{j})_{Z^{\vee}}$, and 
\item
there exist divisors $\Delta_{X^{\rm ws}}$ and $\Delta_{Z^{\rm ws}}$ such that the morphism
$$f^{\rm ws}\colon (X^{\rm ws},\Delta_{X^{\rm ws}}) \to (Z^{\rm ws},\Delta_{Z^{\rm ws}})$$
is weakly semistable as a toroidal morphism with good horizontal divisors, and furthermore, if we define $B^{\rm ws}_{j}$ and $M^{\rm ws}_{j}$ by 
$$K_{X^{\rm ws}}+B^{\rm ws}_{j}+M^{\rm ws}_{j}=(\tau_{X}\circ\mu_{X})^{*}(K_{X}+B_{j}+M_{j}) \quad {\rm and} \quad M^{\rm ws}_{j}:=\mu^{*}_{X}(\boldsymbol{\rm M}_{j})_{X^{\vee}}$$ respectively, then $\Delta_{X^{\rm ws}} \supset {\rm Supp}\,B^{\rm ws}_{j} \cup {\rm Ex}(\mu_{X})$ and $\Delta_{Z^{\rm ws}}$ contains the union of ${\rm Supp}\,\mu^{*}\Sigma^{\vee}$ and the ramification locus of $\mu \colon Z^{\rm ws} \to Z^{\vee}$.  
\end{itemize}
For $\boldsymbol{t} \in \mathcal{P}$, we define $B^{\rm ws}_{\boldsymbol{t}}$ and $M^{\rm ws}_{\boldsymbol{t}}$ by 
$$K_{X^{\rm ws}}+B^{\rm ws}_{\boldsymbol{t}}+M^{\rm ws}_{\boldsymbol{t}}=(\tau_{X}\circ\mu_{X})^{*}(K_{X}+B_{\boldsymbol{t}}+M_{\boldsymbol{t}}) \qquad {\rm and} \qquad M^{\rm ws}_{\boldsymbol{t}}:=\mu_{X}^{*}(\boldsymbol{\rm M}_{\boldsymbol{t}})_{X^{\vee}}.$$
Then we have ${\rm Supp}\,B^{\rm ws}_{\boldsymbol{t}} \subset \bigcup_{j=1}^{m}{\rm Supp}\,B^{\rm ws}_{j} \subset \Delta_{X^{\rm ws}}$. 
Furthermore, the definition of the discriminant $\mathbb{R}$-b-divisors shows ${\rm Supp}\,(\boldsymbol{\rm G}_{\boldsymbol{t}})_{Z^{\vee}} \subset \Sigma^{\vee}$ for every rational point $\boldsymbol{t}$. 
Hence, $(\boldsymbol{\rm G}_{\boldsymbol{t}})_{Z^{\vee}}$ has simple normal crossing support and $\Delta_{Z^{\rm ws}} \supset {\rm Supp}\,\mu^{*}(\boldsymbol{\rm G}_{\boldsymbol{t}})_{Z^{\vee}}$. 
Thus, we may apply Remark \ref{rem--descend-variety} to $f \colon (X, B_{\boldsymbol{t}}+M_{\boldsymbol{t}}) \to Z$ and $\boldsymbol{\rm N}_{\boldsymbol{t}}$ for any $\boldsymbol{t} \in \mathcal{P} \cap {\mathbb{Q}}^{m}$, and we see that $\boldsymbol{\rm N}_{\boldsymbol{t}}$ descends to $Z^{\vee}$ for all rational points $\boldsymbol{t}$. 
In particular, $\boldsymbol{\rm K}+\boldsymbol{\rm G}_{\boldsymbol{t}}$ descends to $Z^{\vee}$ for all rational points $\boldsymbol{t}$, where $\boldsymbol{\rm K}$ is the canonical b-divisor on $Z$.  

Now we have $\boldsymbol{\alpha}:=(\alpha_{1},\,\cdots,\,\alpha_{m}) \in \mathcal{P}$. 
We pick an arbitrary proper bimeromorphic morphism $\sigma \colon Z^{\vee\vee} \to Z^{\vee}$. 
By applying \cite[Proof of Lemma 5.5]{fujino-hashizume-adjunction} to our situation, we can find rational points $\boldsymbol{q}_{1},\,\cdots,\,\boldsymbol{q}_{n} \in \mathcal{P}$ and positive real numbers $c_{1},\, \cdots,\, c_{n}$ such that $\sum_{k=1}^{n}c_{k}=1$ and $\sum_{k=1}^{n}c_{k}\boldsymbol{q}_{k}=\boldsymbol{\alpha}$. 
Then $\sum_{k=1}^{n}c_{k}(\boldsymbol{\rm G}_{\boldsymbol{q}_{k}})_{Z^{\vee\vee}}=(\boldsymbol{\rm G}_{\boldsymbol{\alpha}})_{Z^{\vee\vee}}=\boldsymbol{\rm G}_{Z^{\vee\vee}}$, and
\begin{equation*}
\begin{split}
\sigma^{*}(K_{Z^{\vee}}+\boldsymbol{\rm G}_{Z^{\vee}})=&\sigma^{*}\left(\sum_{k=1}^{n}c_{k}\bigl(K_{Z^{\vee}}+(\boldsymbol{\rm G}_{\boldsymbol{q}_{k}})_{Z^{\vee}}\bigr)\right)=\sum_{k=1}^{n}c_{k}\sigma^{*}\bigl(K_{Z^{\vee}}+(\boldsymbol{\rm G}_{\boldsymbol{q}_{k}})_{Z^{\vee}}\bigr)\\
=&\sum_{k=1}^{n}c_{k}\bigl(K_{Z^{\vee\vee}}+(\boldsymbol{\rm G}_{\boldsymbol{q}_{k}})_{Z^{\vee\vee}}\bigr)=K_{Z^{\vee\vee}}+\boldsymbol{\rm G}_{Z^{\vee\vee}}. 
\end{split}
\end{equation*}
Since $\sigma \colon Z^{\vee\vee} \to Z^{\vee}$ is arbitrary, $\boldsymbol{\rm K}+\boldsymbol{\rm G}$ is $\mathbb{R}$-b-Cartier and this descends to $Z^{\vee}$. 
Thus, $\boldsymbol{\rm N}$ is also $\mathbb{R}$-b-Cartier and this descends to $Z^{\vee}$. 
Since $\boldsymbol{\rm N}_{\boldsymbol{q}_{k}}$ is b-nef$/S$ and descends to $Z^{\vee}$ for every $1\leq k \leq n$ and $\boldsymbol{\rm N}_{Z^{\vee}}=\sum_{k=1}^{n}c_{k}(\boldsymbol{\rm N}_{\boldsymbol{q}_{k}})_{Z^{\vee}}$, we have 
$$\boldsymbol{\rm N}=\overline{\boldsymbol{\rm N}_{Z^{\vee}}}=\sum_{k=1}^{n}c_{k}\overline{(\boldsymbol{\rm N}_{\boldsymbol{q}_{k}})_{Z^{\vee}}}=\sum_{k=1}^{n}c_{k}\boldsymbol{\rm N}_{\boldsymbol{q}_{k}}.$$ 
Thus, we see that Theorem \ref{thm--gen-can-bundle-formula-lc} (\ref{thm--gen-can-bundle-formula-lc-(i)}) holds. 

Theorem \ref{thm--gen-can-bundle-formula-lc} (\ref{thm--gen-can-bundle-formula-lc-(ii)}) follows from the argument in the proof of Theorem \ref{thm--klt-trivial-fib-start} (\ref{thm--klt-trivial-fib-start-(ii)}). 
\end{proof}

\begin{thm}\label{thm--gen-can-bundle-formula-full}
Let $\pi \colon (X,B+M) \to S$ be a generalized sub-pair with the nef part $\boldsymbol{\rm M}$, and let $f\colon (X,B+M) \to Z$ be a generalized lc-trivial fibration over $S$. 
Let $W \subset S$ be a Stein compact subset. 
Suppose that there exists a Zariski open dense subset $U \subset Z$ such that $B|_{f^{-1}(U)}$ is effective.
Then, after replacing $S$ with a suitable open neighborhood of $W$, we have the following properties.
\begin{enumerate}[(i)]
\item \label{thm--gen-can-bundle-formula-full-(i)}
Let $\boldsymbol{\rm G}$ (resp.~$\boldsymbol{\rm N}$) be the discriminant $\mathbb{R}$-b-divisor (resp.~the moduli $\mathbb{R}$-b-divisor) associated to $f$. 
Then $\boldsymbol{\rm N}$ is a b-nef$/S$ $\mathbb{R}$-b-Cartier $\mathbb{R}$-b-divisor on $Z$. 
In particular, $(Z, \boldsymbol{\rm G}_{Z} +\boldsymbol{\rm N}_{Z}) \to S$ is a generalized sub-pair with the nef part $\boldsymbol{\rm N}$. 
Furthermore, if $(X,B+M) \to S$ is generalized lc (resp.~generalized klt), then $(Z, \boldsymbol{\rm G}_{Z} +\boldsymbol{\rm N}_{Z}) \to S$ is also generalized lc (resp.~generalized klt). 
\item \label{thm--gen-can-bundle-formula-full-(ii)}
For any open subset $\tilde{S} \subset S$ that does not necessarily contain $W$, if we put $\tilde{X} \subset X$ and $\tilde{Z} \subset Z$ as the inverse images of $\tilde{S}$, then the generalized lc-trivial fibration $(\tilde{X}, B|_{\tilde{X}}+M|_{\tilde{X}}) \to \tilde{Z}$ over $\tilde{S}$ satisfies (\ref{thm--gen-can-bundle-formula-full-(i)}). 
\end{enumerate}
In particular, the discriminant $\mathbb{R}$-b-divisor and the moduli $\mathbb{R}$-b-divisor associated to $f$ are compatible with the restriction over any open subset. 
\end{thm}

\begin{proof}
By \cite[Lemma 2.16]{fujino-analytic-bchm}, we get
$$W \subset W' \subset S,$$
where $W'\subset S$ is a Stein compact satisfying the property (P4) in Definition \ref{defn--property(P)}. 
Then Theorem \ref{thm--gen-can-bundle-formula-full} for $W$ follows from Theorem \ref{thm--gen-can-bundle-formula-full} (\ref{thm--gen-can-bundle-formula-full-(ii)}) for $W'$. 
Hence, we may replace $W$ by $W'$, and therefore we may assume that $\pi$ and $W$ satisfy the property (P). 
Then Theorem \ref{thm--gen-can-bundle-formula-full} can be proved similarly to Proposition \ref{prop--gen-can-bundle-formula-klt-Q-div} by using the case of $\boldsymbol{\rm M}=0$ in Theorem \ref{thm--gen-can-bundle-formula-lc}, \cite{eh-semistablereduction}, and \cite{fujino-analytic-bchm}. 
We note that the MMP for generalized lc pairs in that argument can be run without assuming that the nef part is a finite $\mathbb{R}_{>0}$-linear combination of b-nef $\mathbb{Q}$-b-Cartier $\mathbb{Q}$-b-divisors (\cite[Section 4]{bz}). 
\end{proof}

\begin{proof}[Proof of Theorem \ref{thm--gen-can-bundle-formula-full-intro}]
This theorem follows from Theorem \ref{thm--gen-can-bundle-formula-lc} and Theorem \ref{thm--gen-can-bundle-formula-full}. 
\end{proof}

By combining the proof of Theorem \ref{thm--gen-can-bundle-formula-lc} with \cite[Theorem 7.3]{bfmt}, we obtain a result related to \cite[7.2.3]{bfmt}.

\begin{thm}\label{thm--lc-trivial-bfmt}
Let $f \colon (X,\Delta) \to Z$ be an lc-trivial fibration over a complex analytic space $S$, and let $W \subset S$ be a Stein compact subset. 
Suppose that $\Delta|_{f^{-1}(U)}$ is effective for some Zariski open dense subset $U \subset Z$.
Then, after shrinking $S$ around $W$, there exists a projective bimeromorphic morphism $Z' \to Z$ such that for any open subset $\tilde{S} \subset S$ that does not necessarily contain $W$ (possibly $\tilde{S}=S$), if we put $\tilde{X} \subset X$, $\tilde{Z} \subset Z$, and $\tilde{Z}' \subset Z'$ as the inverse images of $\tilde{S}$, then the moduli $\mathbb{R}$-b-divisor $\tilde{\boldsymbol{\rm N}}$ associated to $(\tilde{X},\Delta|_{\tilde{X}}) \to \tilde{Z}$ descends to $\tilde{Z}'$ and $\tilde{\boldsymbol{\rm N}}_{\tilde{Z}'}$ is semi-ample over $\tilde{S}$. 
Moreover, we have $\boldsymbol{\rm N}|_{\tilde{Z}}=\tilde{\boldsymbol{\rm N}}$, where  $\boldsymbol{\rm N}$ is the moduli $\mathbb{R}$-b-divisor associated to $f \colon (X,\Delta) \to Z$. 
\end{thm}

\begin{proof}
Let $\boldsymbol{\rm N}$ be the moduli $\mathbb{R}$-b-divisor associated to $f$. 
By the same argument as in the proof of Theorem \ref{thm--gen-can-bundle-formula-lc}, after shrinking $S$ around $W$ we can find $r_{1},\,\cdots,\, r_{l}\in \mathbb{R}_{>0}$ and $\mathbb{Q}$-divisors $\Delta_{1},\,\cdots,\,\Delta_{l}$ on $X$ such that
\begin{itemize}
\item
$\sum_{i=1}^{l}r_{i}=1$ and $\sum_{i=1}^{l}r_{i}\Delta_{i}=\Delta$, 
\item
all $(X,\Delta_{i})\to Z$ are lc-trivial fibrations over $S$ and $\Delta_{i}|_{f^{-1}(U)}$ are effective, and 
\item
if we define $\boldsymbol{\rm N}_{i}$ to be the moduli $\mathbb{Q}$-b-divisor associated to $(X,\Delta_{i})\to Z$ for each $1 \leq i \leq l$, then $\sum_{i=1}^{l}r_{i}\boldsymbol{\rm N}_{i}=\boldsymbol{\rm N}$. 
\end{itemize}
By Theorem \ref{thm--gen-can-bundle-formula-lc} (see also Remark \ref{rem--descend-variety} and Lemma \ref{lem--restriction-can-bundle-formula}) and shrinking $S$ around $W$, we can find a projective bimeromorphic morphism $Z' \to Z$ such that $\boldsymbol{\rm N}$ and $\boldsymbol{\rm N}_{i}$ descend to $Z'$ and for any open subset $\tilde{S} \subset S$ (possibly $\tilde{S}=S$ and not containing $W$) with the inverse images $\tilde{X}$, $\tilde{Z}$, and $\tilde{Z}'$, the moduli $\mathbb{R}$-b-divisor $\tilde{\boldsymbol{\rm N}}$ associated to $(\tilde{X},\Delta|_{\tilde{X}}) \to \tilde{Z}$ satisfies $\tilde{\boldsymbol{\rm N}}=\boldsymbol{\rm N}|_{\tilde{Z}}$. 
Since $\sum_{i=1}^{l}r_{i}\boldsymbol{\rm N}_{i}=\boldsymbol{\rm N}$, by replacing $(X,\Delta) \to Z$ by $(X,\Delta_{i})\to Z$ for any $i$, we may assume that $\Delta$ is a $\mathbb{Q}$-divisor. 
Then \cite[Proof of Theorem 7.3]{bfmt} can be applied to our setting, and we see that $\boldsymbol{\rm N}_{Z'}$ is semi-ample over $S$. 
Note that \cite[Remark 7.4]{bfmt} can be recovered by applying \cite[Proof of Theorem 1.1]{fg-lctrivial} and Theorem \ref{thm--gen-can-bundle-formula-lc}. 
We also note that shrinking $S$ around $W$ allows us to find positive integers $k$ and $k'$ such that $Z$ and $X$ can be embedded into $S \times \mathbb{P}^{k}$ and $S \times \mathbb{P}^{k}\times \mathbb{P}^{k'}$, respectively (\cite[Lemma 2.2]{eh-semistablereduction}). 
In this way, we see that Theorem \ref{thm--lc-trivial-bfmt} holds. 
\end{proof}

\section{Boundedness of complements}\label{sec--compl}

In this section, we prove the boundedness of complements. 

\begin{defn}[Complement, cf.~{\cite[Definition 5.1]{shokurov-flip}}]
Let $n$ be a positive integer, let $\pi \colon (X,B+M) \to Z$ be a generalized pair with the nef part $\boldsymbol{\rm M}$ such that $\pi \colon X \to Z$ is a contraction, and let $z \in Z$ be a point. 
Then an {\em $n$-complement of $(X,B+M) \to Z$ over $z \in Z$} is a generalized pair $(X,B^{+}+M) \to Z$ such that after shrinking $Z$ around $z$, we have
\begin{itemize}
\item
$(X,B^{+}+M) \to Z$ is generalized lc, 
\item
$nB^{+} \geq n\lfloor B \rfloor+\lfloor (n+1)\{B\}\rfloor$, and
\item
$n(K_{X}+B^{+}+M)\sim 0$ and $n \boldsymbol{\rm M}$ is b-Cartier. 
\end{itemize}
These properties are obviously preserved under further shrinking of $Z$ around $z$.
\end{defn}

\subsection{Statements}\label{subsec--state-comple}

In this subsection, we collect the main results proved in this section.  

\begin{thm}[Boundedness of complements I, cf.~{\cite[Theorem 1.2]{filipazzi-moraga}}]\label{thm--generalized-compl-main-1}
Let $d$ and $p$ be positive integers and let $\Lambda  \subset [0,1] \cap \mathbb{Q}$ be a DCC set whose accumulation points are rational numbers. 
Then there exists $n \in \mathbb{Z}_{>0}$, depending only on $d$, $p$, and $\Lambda$, satisfying the following. 
Let $\pi \colon(X,B+M) \to Z$ be a generalized lc pair with the nef part $\boldsymbol{\rm M}$ such that 
\begin{itemize}
\item
$\pi \colon X \to Z$ is a contraction of normal analytic varieties and ${\rm dim}\,X=d$, 
\item
$B \in \Lambda$ and $p \boldsymbol{\rm M}$ is b-Cartier, 
\item
$X$ is of Fano type over $Z$, and 
\item
$-(K_{X}+B+M)$ is nef over $Z$.
\end{itemize} 
Then, for any point $z \in Z$, after shrinking $Z$ around $z$ suitably, $(X,B+M)\to Z$ has an $n$-complement $(X,B^{+}+M)\to Z$ over $z$ such that $B^{+} \geq B$. 
\end{thm}

\begin{thm}[Boundedness of complements I\!I, cf.~{\cite[Theorem 1.1]{chhx-compl}}]\label{thm--generalized-compl-main-2}
Let $d$ and $p$ be positive integers. 
Then there exists a finite set $\mathcal{N}$ of positive integers, depending only on $d$ and $p$, satisfying the following.
Let $\pi \colon(X,B+M) \to Z$ be a generalized pair with the nef part $\boldsymbol{\rm M}$ and let $z \in Z$ be a point such that 
\begin{itemize}
\item
$\pi \colon X \to Z$ is a contraction of normal analytic varieties and ${\rm dim}\,X=d$, 
\item
$p \boldsymbol{\rm M}$ is b-Cartier, 
\item
$X$ is of Fano type over $Z$, and 
\item
after shrinking $Z$ around $z$, there exists an effective $\mathbb{R}$-divisor $\Delta$ on $X$ such that $(X,(B+\Delta)+M) \to Z$ is generalized klt and $K_{X}+B+\Delta+M \sim_{\mathbb{R}} 0$.
\end{itemize} 
Then, after shrinking $Z$ around $z$ suitably, $(X,B+M) \to Z$ has an $n$-complement over $z \in Z$ for some $n \in \mathcal{N}$. 
\end{thm}

\begin{cor}[cf.~{\cite[Theorem 1.8]{birkar-compl}}]\label{cor--compl-log-pair}
Let $d$ be a positive integer and $\mathfrak{R} \subset [0,1] \cap \mathbb{Q}$ a finite set. 
Then there exists $n \in \mathbb{Z}_{>0}$, depending only on $d$ and $\mathfrak{R}$, satisfying the following. 
Let $\pi \colon X \to Z$ be a contraction of normal analytic varieties and let $(X,B)$ be an lc pair such that 
\begin{itemize}
\item
${\rm dim}\,X=d$, 
\item
$B \in \Phi(\mathfrak{R})$, 
\item
$X$ is of Fano type over $Z$, and 
\item
$-(K_{X}+B)$ is nef over $Z$.
\end{itemize} 
Then, for any point $z \in Z$, there exists an $n$-complement $(X,B^{+})$ of $(X,B)$ over $z$ such that $B^{+} \geq B$. 
Moreover, the complement is an $mn$-complement for any $m \in \mathbb{Z}_{>0}$. 
\end{cor}

\subsection{Extension of complements}\label{subsec--extend-comple}

In this subsection, we prove auxiliary results to globally extend complements of subvarieties. 

\begin{lem}\label{lem--compl-extend-main}
Let $n$ be a positive integer, let $\pi \colon (X,B+M) \to Z$ be a generalized lc pair with the nef part $\boldsymbol{\rm M}$, and let $z \in Z$ be a point such that
\begin{itemize}
\item
$\pi \colon X \to Z$ is a contraction and $X$ is $\mathbb{Q}$-factorial over $z$, 
\item
$n\boldsymbol{\rm M}$ is b-Cartier, 
\item
$-(K_{X}+B+M)$ is $\pi$-nef.  
\end{itemize}
Suppose that there exists a generalized plt pair $\pi \colon (X,\Gamma+N) \to Z$ such that
\begin{itemize}
\item
$-(K_{X}+\Gamma+N)$ is $\pi$-nef and $\pi$-big,
\item
$S:=\lfloor \Gamma \rfloor$ is a prime divisor on $X$ and it is a component of $\lfloor B \rfloor$, 
\item
$\pi(S) \ni z$, and
\item
$\Gamma^{<1}\geq B^{<1}-\{(n+1)B^{<1}\}$. 
\end{itemize}
Let $\pi|_{S}\colon (S,B_{S}+M_{S}) \to Z$ be the generalized pair defined by the generalized adjunction $K_{S}+B_{S}+M_{S}\sim_{\mathbb{R},\,Z}(K_{X}+B+M)|_{S}$. 
Then, for any $n$-complement $(S,B^{+}_{S}+M_{S}) \to Z$ of $(S,B_{S}+M_{S}) \to Z$ over $z \in Z$ such that $B^{+}_{S} \geq B_{S}$, after shrinking $Z$ around $z$ suitably, there exists an $n$-complement $(X,B^{+}+M) \to Z$  of $(X,B+M) \to Z$ over $z \in Z$ such that $(K_{X}+B^{+}+M)|_{S}=K_{S}+B^{+}_{S}+M_{S}$. 
\end{lem}

\begin{proof}
The idea can be found in \cite[Proof of Proposition 8.1]{birkar-compl}. 
However, notations in the proof are slightly more complicated since we deal with generalized pairs. 
See also \cite[Proof of Proposition 4.1]{filipazzi-moraga} or \cite[Proof of Proposition 3.5]{chhx-compl} for details. 

We remark that we may harmlessly replace $Z$ by a Zariski open neighborhood of $z$ in each discussion. 
Hence, shrinking $Z$ without explicit mention, we always assume that every component of every $\mathbb{R}$-divisor appearing in the proof intersects the fiber of $z$. 
Since $X$ is $\mathbb{Q}$-factorial over $z$, we in particular assume that all the $\mathbb{R}$-divisors appearing in the proof are $\mathbb{R}$-Cartier. 

We divide the proof into several steps. 

\begin{step1}\label{step1--lem-compl-extend-main}
In this step we fix notations. 

Let $\boldsymbol{\rm N}$ be the nef part of $(X,\Gamma+N) \to Z$. 
We take a log resolution 
$$\phi \colon X' \to X$$
of $(X,B+\Gamma)$ such that both $\boldsymbol{\rm M}$ and $\boldsymbol{\rm N}$ descend to $X'$. 
We set $S':=\phi^{-1}_{*} S$, $M':=\boldsymbol{\rm M}_{X'}$, and $M_{S'}:=M'|_{S'}$. 
We denote the induced morphism $S' \to S$ by $\phi_{S}$. 

Let $B'$ be an $\mathbb{R}$-divisor on $X'$ defined by $K_{X'}+B'+M'=\phi^{*}(K_{X}+B+M)$. 
We set
\begin{equation}\label{lem--compl-extend-main-(TDelta)}
T':=B'^{=1} \qquad {\rm and}\qquad \Delta'= B'-T'.
\end{equation}
We also set
\begin{equation}\label{lem--compl-extend-main-(LH)}
H':=-(K_{X'}+B'+M') \qquad {\rm and}\qquad L':=-nK_{X'}-nT'-\lfloor (n+1)\Delta'\rfloor-nM'.
\end{equation}
Let $\Gamma'$ be an $\mathbb{R}$-divisor on $X'$ defined by $K_{X'}+\Gamma'+\boldsymbol{\rm N}_{X'}=\phi^{*}(K_{X}+\Gamma+N)$. 
Shrinking $Z$ around $z$ and replacing $\Gamma$ (resp.~$N$) by $(1-a)\Gamma+aB$ (resp.~$(1-a)N+aM$)  for some $a \in (0,1)$ sufficiently close to one, we may assume that every coefficient of $B'-\Gamma'$ has sufficiently small absolute value. 
Note that the fourth condition of $(X,\Gamma+N) \to Z$ in Lemma \ref{lem--compl-extend-main} is preserved after this replacement because 
\begin{equation*}
\begin{split}
(1-a)\Gamma^{<1}+aB^{<1} \geq & (1-a)\bigl(B^{<1}-\{(n+1)B^{<1}\}\bigr)+aB^{<1}\\
=&B^{<1}-(1-a)\{(n+1)B^{<1}\}\geq B^{<1}-\{(n+1)B^{<1}\} 
\end{split}
\end{equation*}
for any $a \in (0,1)$. 
\end{step1}

\begin{step1}\label{step2--lem-compl-extend-main}
In this step we define a Weil divisor $\Psi'$ on $X'$ and study its properties.  

We define a Weil divisor $\Psi'$ on $X'$ by
\begin{equation}\label{lem--compl-extend-main-(Psi)}
\Psi' := S' - \lfloor (\Gamma' +n\Delta' - \lfloor (n+1)\Delta' \rfloor)\rfloor = S' - \lfloor (\Gamma' -\Delta' + \{ (n+1)\Delta' \})\rfloor.
\end{equation}
We also set
\begin{equation}\label{lem--compl-extend-main-(Lambda)}
\Lambda':=\Gamma' +n\Delta' - \lfloor (n+1)\Delta' \rfloor  +\Psi'=S'+\{(\Gamma' +n\Delta' - \lfloor (n+1)\Delta'\rfloor)\}.
\end{equation}
We check that $\Psi'$ is $\phi$-exceptional and all the coefficients of $\Psi'$ belong to $[0,1]$. 
Because the latter part follows from \cite[Step 4 in the proof of Proposition 8.1]{birkar-compl}, we only show that $\Psi'$ is $\phi$-exceptional. 
Since $\lfloor \Gamma \rfloor =S$ and $\phi_{*}\Delta'=\{B\}=B^{<1}$, we have 
$$\phi_{*}\Psi'=S - \lfloor (\Gamma -B^{<1} + \{ (n+1)B^{<1} \})\rfloor=- \lfloor (\Gamma^{<1} -B^{<1} + \{ (n+1)B^{<1} \})\rfloor.$$
By recalling that $\Gamma^{<1}\geq B^{<1}-\{(n+1)B^{<1}\}$, which is our hypothesis of Lemma \ref{lem--compl-extend-main}, we have $\Gamma^{<1} -B^{<1} + \{ (n+1)B^{<1} \} \geq 0$. 
On the other hand, since every coefficient of $B'-\Gamma'$ has sufficiently small absolute value, all the coefficients of $\Gamma^{<1} -B^{<1} + \{ (n+1)B^{<1} \}$ are contained in $[0,1)$. 
This implies $\phi_{*}\Psi' = 0$, and therefore $\Psi'$ is $\phi$-exceptional. 
We also note that to show the latter part, we use the fact that every coefficient of $B'-\Gamma'$ has sufficiently small absolute value. 
\end{step1}

\begin{step1}\label{step3--lem-compl-extend-main}
By (\ref{lem--compl-extend-main-(TDelta)}) and (\ref{lem--compl-extend-main-(LH)}), we have
\begin{equation}\label{lem--compl-extend-main-(L-calc)}
\begin{split}
L'+\Psi'=&n\Delta' - \lfloor (n+1)\Delta' \rfloor+nH'+\Psi',
\end{split}
\end{equation}
and (\ref{lem--compl-extend-main-(Lambda)}) shows
\begin{equation*}
\begin{split}
L'+\Psi'=&n\Delta' - \lfloor (n+1)\Delta' \rfloor+nH'+\Psi'\\
=& K_{X'}+\Gamma'+\boldsymbol{\rm N}_{X'}-\phi^{*}(K_{X}+\Gamma+N)+(\Lambda'-\Gamma')+nH'\\
=& K_{X'}+\Lambda'+\boldsymbol{\rm N}_{X'}-\phi^{*}(K_{X}+\Gamma+N)+nH'.
\end{split}
\end{equation*}
We set $\pi':=\pi \circ \phi \colon X' \to Z$. 
Since $\boldsymbol{\rm N}_{X'}-\phi^{*}(K_{X}+\Gamma+N)+nH'$ is $\pi'$-nef and $\pi'$-big and $(X',\Lambda')$ is a log smooth plt pair with $\lfloor \Lambda' \rfloor = S'$, we can apply \cite[II, 5.12 Corollary]{nakayama} to $\pi'$, $(X',\Lambda')$, and $L'+\Psi'-S'$, and we have $R^{1}\pi'_{*}\mathcal{O}_{X'}(L'+\Psi'-S')=0$. 
By considering a long cohomology exact sequence, we see that the morphism
$$\pi'_{*}\mathcal{O}_{X'}(L'+\Psi')\longrightarrow \pi'_{*}\mathcal{O}_{S'}((L'+\Psi')|_{S'})$$
is surjective. 
\end{step1}

\begin{step1}\label{step4--lem-compl-extend-main}
In this step we define some $\mathbb{R}$-divisors. 

Recall the hypothesis of Lemma \ref{lem--compl-extend-main} that $(S,B_{S}+M_{S}) \to Z$ has an $n$-complement $(S,B^{+}_{S}+M_{S}) \to Z$ over $z \in Z$ such that $B^{+}_{S} \geq B_{S}$, where $(S,B_{S}+M_{S}) \to Z$ is defined by using the generalized adjunction $K_{S}+B_{S}+M_{S}\sim_{\mathbb{R},\,Z}(K_{X}+B+M)|_{S}$. 
We define 
$$R_{S}:=B^{+}_{S}-B_{S}.$$ 
By shrinking $Z$ around $z$, we may assume $n(K_{S}+B^{+}_{S}+M_{S}) \sim 0$. 
Then we have
$$-n(K_{S}+B_{S}+M_{S})=-n(K_{S}+B^{+}_{S}+M_{S}+B_{S}-B^{+}_{S}) \sim nR_{S} \geq 0.$$
Put $R_{S'}:=\phi^{*}_{S}R_{S}$. 
Then
$$nH'|_{S'}=-n(K_{X'}+B'+M')|_{S'} = -n \phi^{*}_{S}(K_{S}+B_{S}+M_{S}) \sim nR_{S'} \geq 0.$$
By (\ref{lem--compl-extend-main-(L-calc)}), we have
\begin{equation}\label{lem--compl-extend-main-(G)}
\begin{split}
(L'+\Psi')|_{S'}=&(n\Delta' - \lfloor (n+1)\Delta' \rfloor+nH'+\Psi')|_{S'} \\
\sim &G_{S'}:=nR_{S'}+n\Delta_{S'} - \lfloor (n+1)\Delta_{S'} \rfloor+\Psi_{S'},
\end{split}
\end{equation}
where $\Delta_{S'}:=\Delta'|_{S'}$ and $\Psi_{S'}:=\Psi'|_{S'}$. 
We note that $\Delta_{S'}$ and $\Psi_{S'}$ are well defined as $\mathbb{R}$-divisors on $S'$ since ${\rm coeff}_{S'}(\Delta')={\rm coeff}_{S'}(\Psi')=0$ by (\ref{lem--compl-extend-main-(TDelta)}) and (\ref{lem--compl-extend-main-(Psi)}). 
We also note that $G_{S'}$ is a Weil divisor on $S'$ since $L'$ and $\Psi'$ are both Weil divisors on $X'$. 
\end{step1}

\begin{step1}\label{step5--lem-compl-extend-main}
In this step we prove $G_{S'} \geq 0$ and lift this to an effective Weil divisor $G'$ on $X'$. 

The fact $G_{S'} \geq 0$ follows from
$$G_{S'}= \lceil G_{S'}\rceil=\lceil (nR_{S'}-\Delta_{S'} + \{(n+1)\Delta_{S'}\}+\Psi_{S'})\rceil\geq \lceil -\Delta_{S'}\rceil \geq 0.$$
Here, we used the facts that $\Psi_{S'}\geq 0$ (see Step \ref{step2--lem-compl-extend-main}) and $\Delta_{S'}=\Delta^{<1}_{S'}$ (see (\ref{lem--compl-extend-main-(TDelta)})).  
To lift $G_{S'}$, we apply the surjectivity of the morphism
$$\pi'_{*}\mathcal{O}_{X'}(L'+\Psi')\longrightarrow \pi'_{*}\mathcal{O}_{S'}((L'+\Psi')|_{S'}),$$
which was proved in Step \ref{step3--lem-compl-extend-main}. 
Hence, we can find an effective Weil divisor $G'$ on $X'$ such that $S' \not\subset {\rm Supp}\,G'$ and 
$$G'|_{S'}=G_{S'}$$
as Weil divisors on $S'$. 
\end{step1}

\begin{step1}\label{step6--lem-compl-extend-main}
In this step we define an effective $\mathbb{Q}$-divisor $B^{+}$ on $X$. 
From this step, for any $\mathbb{R}$-divisor $D'$ on $X'$, we denote $\phi_{*}D'$ by $D$. 
For example, $L$ denotes $\phi_{*}L'$. 

Since $\Psi'$ is $\phi$-exceptional (see Step \ref{step2--lem-compl-extend-main}), by using (\ref{lem--compl-extend-main-(LH)}) we see that
\begin{equation*}
\begin{split}
-nK_{X}-n\lfloor B \rfloor-\lfloor (n+1)\{B\} \rfloor -nM
=&-nK_{X}-nT-\lfloor (n+1)\Delta \rfloor -nM \\
=&L=L+\Psi \sim G \geq 0.
\end{split}
\end{equation*}
We define 
$$B^{+}:=\lfloor B \rfloor+\frac{1}{n}\lfloor (n+1)\{B\}\rfloor+\frac{1}{n}G,$$
which is an effective $\mathbb{Q}$-divisor on $X$. 
\end{step1}

\begin{step1}\label{step7--lem-compl-extend-main}
With this step we complete the proof. 

We have $B^{+} \geq \lfloor B \rfloor+\frac{1}{n}\lfloor (n+1)\{B\}\rfloor$ by construction. 
By (\ref{lem--compl-extend-main-(LH)}), we see that
$$nK_{X'}+nT'+\lfloor (n+1)\Delta'\rfloor+G'+nM'-\Psi'=-L'+G'-\Psi'\sim0.$$
Hence, we have
$$0 \sim n(K_{X}+B^{+}+M)=nK_{X}+n\lfloor B \rfloor+\lfloor (n+1)\{B\}\rfloor +G+nM,$$
and the negativity lemma implies
$$K_{X'}+T'+ \frac{1}{n}\lfloor (n+1)\Delta'\rfloor+\frac{1}{n}G'+M'-\frac{1}{n}\Psi'=\phi^{*}(K_{X}+B^{+}+M).$$
By restricting the equality to $S'$ and using (\ref{lem--compl-extend-main-(TDelta)}) and (\ref{lem--compl-extend-main-(G)}), we obtain
\begin{equation*}
\begin{split}
0\sim &\phi^{*}_{S}(K_{X}+B^{+}+M)|_{S}\\
=&K_{S'}+(T'-S')|_{S'}+\frac{1}{n}\lfloor (n+1)\Delta_{S'}\rfloor+\frac{1}{n}G_{S'}-\frac{1}{n}\Psi_{S'}+M_{S'}\\
=&K_{S'}+(T'-S')|_{S'}+\Delta_{S'}+R_{S'}+M_{S'}\\
=&K_{S'}+(B'-S')|_{S'}+R_{S'}+M_{S'}.
\end{split}
\end{equation*}
By the definition of generalized adjunction, we have 
$$K_{S'}+(B'-S')|_{S'}+M_{S'}=\phi^{*}_{S}(K_{S}+B_{S}+M_{S}).$$
This relation and $R_{S'}=\phi^{*}_{S}R_{S}=\phi^{*}_{S}(B^{+}_{S}-B_{S})$ imply
\begin{equation*}
\begin{split}
(K_{X}+B^{+}+M)|_{S}=&\phi_{S*}(K_{S'}+(B'-S')|_{S'}+R_{S'}+M_{S'})\\
=&K_{S}+B_{S}+M_{S}+R_{S}=K_{S}+B^{+}_{S}+M_{S}. 
\end{split}
\end{equation*}

By the above arguments, we have proved 
\begin{itemize}
\item
$B^{+} \geq \lfloor B \rfloor+\frac{1}{n}\lfloor (n+1)\{B\}\rfloor$, 
\item
$n(K_{X}+B^{+}+M) \sim0$, and
\item
$(K_{X}+B^{+}+M)|_{S}=K_{S}+B^{+}_{S}+M_{S}$. 
\end{itemize}
By the generalized inversion of adjunction (Lemma \ref{lem--gen-inv-adj}), we see that $(X,B^{+}+M)\to Z$ is generalized lc on a neighborhood of $S \cap \pi^{-1}(z)$. 
Finally, by using the connectedness principle (Lemma \ref{lem--conn-princi}), we see that $(X,B^{+}+M)$ is generalized lc on a neighborhood of $\pi^{-1}(z)$. 
For details, see \cite[Step 9 in the proof of Proposition 8.1]{birkar-compl}. 
Therefore, after shrinking $Z$ around $z$, we get an $n$-complement $(X,B^{+}+M) \to Z$ of $(X,B+M) \to Z$ over $z \in Z$. 
\end{step1}
We finish the proof. 
\end{proof}

\begin{rem}\label{rem--monotonic-compl}
With notations as in Lemma \ref{lem--compl-extend-main}, we further assume that $B \in \Gamma(n, \Phi(\mathfrak{R}))$ for some finite set $\mathfrak{R} \subset [0,1]\cap \frac{1}{n}\mathbb{Z}$. 
Then we have $B^{+}\geq B$. 
This fact is an immediate consequence of \cite[Lemma 2.3]{chx-compl} (see also \cite[Lemma 2.20]{chhx-compl}). 
\end{rem}

\begin{cor}[cf.~{\cite[Proposition 3.5, Remark 3.6]{chhx-compl}}]\label{cor--compl-extend-1}
Let $n$ be a positive integer. 
Let $\pi \colon (X,B+M) \to Z$ be a generalized plt pair with the nef part $\boldsymbol{\rm M}$ and let $z \in Z$ be a point such that
\begin{itemize}
\item
$\pi$ is a contraction and $X$ is $\mathbb{Q}$-factorial over $z$, 
\item
$n\boldsymbol{\rm M}$ is b-Cartier, 
\item
$-(K_{X}+B+M)$ is $\pi$-nef and $\pi$-big, 
\item
$T:=\lfloor B \rfloor$ is a prime divisor on $X$ such that $\pi(T) \ni z$, 
\item
$\pi|_{T}\colon(T,B_{T}+M_{T}) \to Z$ has an $n$-complement $(T,B^{+}_{T}+M_{T}) \to Z$ over $z \in Z$ such that $B^{+}_{T} \geq B_{T}$, where $(T,B_{T}+M_{T}) \to Z$ is the generalized pair defined by using the generalized adjunction $K_{T}+B_{T}+M_{T}\sim_{\mathbb{R},\,Z}(K_{X}+B+M)|_{T}$. 
\end{itemize}
Then $(X,B+M) \to Z$ has an $n$-complement $(X,B^{+}+M) \to Z$ over $z \in Z$ such that $(K_{X}+B^{+}+M)|_{T}=K_{T}+B^{+}_{T}+M_{T}$. 
Furthermore, if $B \in \Gamma(n, \Phi(\mathfrak{R}))$ for some finite set $\mathfrak{R} \subset [0,1]\cap \frac{1}{n}\mathbb{Z}$, then we have $B^{+}\geq B$. 
\end{cor}

\begin{proof}
This is a special case of Lemma \ref{lem--compl-extend-main} where the generalized pair $(X,B+M) \to Z$ equals $(X,\Gamma+N) \to Z$. 
Note that the relation $\Gamma^{<1}\geq B^{<1}-\{(n+1)B^{<1}\}$ is obvious when $\Gamma=B$. 
\end{proof}

\begin{cor}[cf.~{\cite[Proposition 4.1]{filipazzi-moraga}}]\label{cor--compl-extend-2}
Assume Theorem \ref{thm--generalized-compl-main-1} for ${\rm dim}\, X=d-1$. 
Then Theorem \ref{thm--generalized-compl-main-1} holds for ${\rm dim}\, X=d$ such that
\begin{itemize}
\item
$\Lambda$ is a finite set and $X$ is $\mathbb{Q}$-factorial over $z$, 
\item
$(X,\Gamma+\beta M) \to Z$ is generalized plt for some $\Gamma$ and $\beta \in (0,1)$,  
\item
$-(K_{X}+\Gamma+\beta M)$ is ample over $Z$, 
\item
$T:= \lfloor \Gamma \rfloor$ is irreducible and it is a component of $\lfloor B \rfloor$, and 
\item
$T$ intersects $\pi^{-1}(z)$. 
\end{itemize}  
\end{cor}

\begin{proof}
We only show how to use Lemma \ref{lem--compl-extend-main} for this corollary. 

Since $\Lambda$ is a finite set of rational numbers, replacing $p$ if necessary we may assume $p\Lambda \subset \mathbb{Z}$. 
It is obvious that the generalized pair $\pi \colon(X,B+M) \to Z$ satisfies the three conditions of Lemma \ref{lem--compl-extend-main}. 
We also see that the generalized plt pair $(X,\Gamma+\beta M) \to Z$ satisfies the four conditions of Lemma \ref{lem--compl-extend-main}. 
Indeed, $(X,\Gamma+\beta M) \to Z$ clearly satisfies those conditions except the final condition of Lemma \ref{lem--compl-extend-main}. 
Since $B \in \Lambda$ and $p\Lambda \subset \mathbb{Z}$, it follows that $pB$ is a Weil divisor on $X$. 
Hence, the equality $\{(p+1)B^{<1}\}=B^{<1}$ holds. 
Therefore $\Gamma^{<1} \geq 0 =B^{<1}-\{(p+1)B^{<1}\}$, which is nothing but the final condition of Lemma \ref{lem--compl-extend-main}. 
Let $\pi|_{T} \colon (T,B_{T}+M_{T}) \to Z_{T}$ be the generalized pair defined by the generalized adjunction $K_{T}+B_{T}+M_{T}\sim_{\mathbb{R},\,Z}(K_{X}+B+M)|_{T}$, where $Z_{T}:=\pi(T)$. 
By the same argument as in \cite[Step 1 in the proof of Proposition 8.1]{birkar-compl}, this generalized pair satisfies the conditions of Theorem \ref{thm--generalized-compl-main-1}. 
By Lemma \ref{lem--adj-coeff} and the induction hypothesis of Theorem \ref{thm--generalized-compl-main-1}, there exists a positive integer $n'$, depending only on $d$, $p$, and $\Lambda$, such that $(T,B_{T}+M_{T}) \to Z_{T}$ has an $n'$-complement $(T,B^{+}_{T}+M_{T}) \to Z_{T}$ over $z \in Z_{T}$ such that $B^{+}_{T} \geq B_{T}$. 

Now we apply Remark \ref{rem--monotonic-compl} and Lemma \ref{lem--compl-extend-main} for $n:=pn'$, and we complete the proof. 
Note that the conditions of Lemma \ref{lem--compl-extend-main} are preserved in this situation after we replace $p$ by $pn'$. 
\end{proof}

\subsection{Proofs of statements}\label{subsec--proof-comple}

In this subsection we prove Theorem \ref{thm--generalized-compl-main-1}, Theorem \ref{thm--generalized-compl-main-2}, and corollaries. 
We only outline the proof. 
Throughout this subsection, we will freely use Lemma \ref{lem--Fano-type-bimero-cont} and Theorem \ref{thm--mmp-Fano-type} without mentioning them explicitly. 

Theorem \ref{thm--generalized-compl-main-1} follows from Corollary \ref{cor--compl-extend-2}.  

\begin{proof}[Proof of Theorem \ref{thm--generalized-compl-main-1}]
The argument in \cite[Section 4]{filipazzi-moraga} works in our situation with no changes. 
The new ingredients we need are the complex analytic analogs of the ACC for generalized lc thresholds and the relative MMP and the relative base point free theorem for fibrations of Fano type. 
These results are known by \cite{hacon-xie-acc-lct} and \cite{fujino-analytic-bchm}, respectively. 
\end{proof}

\begin{proof}[Proof of Corollary \ref{cor--compl-log-pair}]
This corollary immediately follows from Theorem \ref{thm--generalized-compl-main-1}. 
\end{proof}

We prepare two results for the proof of Theorem \ref{thm--generalized-compl-main-2}.

\begin{thm}[cf.~{\cite[Theorem 1.5]{filipazzi-moraga}, \cite[Proposition 3.4]{chhx-compl}}]\label{thm--can-bundle-formula-coeff-inherited}
Let $d$ and $p$ be positive integers and let $\Lambda  \subset [0,1] \cap \mathbb{Q}$ be a DCC set whose accumulation points are rational numbers. 
Then there exist a positive integer $q$ and a DCC set $\Omega \subset \mathbb{Q}$ whose accumulation points are rational numbers, depending only on $d$, $p$, and $\Lambda$, satisfying the following. 
Let $\pi \colon X\to S$ be a projective morphism and let $W \subset S$ be a compact subset such that $\pi$ and $W$ satisfy (P). 
Let $f\colon (X,B+M) \to Z$ be a generalized lc-trivial fibration over $S$ such that
\begin{itemize}
\item
$\pi \colon (X,B+M) \to S$ is a $d$-dimensional generalized lc pair with the nef part $\boldsymbol{\rm M}$,
\item
$B \in \Lambda$ and $p\boldsymbol{\rm M}$ is b-Cartier, and 
\item
$f$ is of Fano type.
\end{itemize}
Then, by the generalized canonical bundle formula and shrinking $S$ around $W$ suitably, we obtain a generalized lc pair $(Z, G +N) \to S$ with the nef part $\boldsymbol{\rm N}$ such that
\begin{itemize}
\item
$G \in \Omega$ and $q\boldsymbol{\rm N}$ is b-Cartier, and 
\item
$q(K_{X}+B+M) \sim qf^{*}(K_{Z}+G+N)$.
\end{itemize}
Furthermore, for any open subset $\tilde{S} \subset S$ and the restriction $\tilde{f} \colon (\tilde{X},\tilde{B}+\tilde{M}) \to \tilde{Z}$ of $f\colon (X,B+M) \to Z$ over $\tilde{S}$, the generalized canonical bundle formula yields a generalized lc pair $(\tilde{Z}, \tilde{G} +\tilde{N}) \to \tilde{S}$ with the nef part $\tilde{\boldsymbol{\rm N}}$ that satisfies the above two properties. 
If $\Lambda= \Phi(\mathfrak{R})$ for some finite set $\mathfrak{R} \subset [0,1] \cap \mathbb{Q}$, then we can take $\Omega$ as $\Phi(\mathfrak{S})$ for a finite set $\mathfrak{S} \subset [0,1] \cap \mathbb{Q}$ that depends only on $d$, $p$, and $\mathfrak{R}$. 
\end{thm}

\begin{proof}
We may assume ${\rm dim}\,Z>0$ because the case of ${\rm dim}\,Z=0$ immediately follows from \cite[Theorem 1.2]{filipazzi-moraga} (see also \cite[Theorem 1.10]{birkar-compl}). 
In this situation, we may apply \cite[Proof of Proposition 6.3]{birkar-compl} or \cite[Section 5]{filipazzi-moraga}.  
Here, we only outline how to reduce the theorem to the case where $Z$ is smooth and $\boldsymbol{\rm N}$ descends to $Z$. 
We may assume $1 \in \Lambda$. 

By the equi-dimensional reduction (\cite{eh-semistablereduction}), we obtain the diagram
 $$
\xymatrix{
X \ar@{->}[d]_{f}& \ar@{->}[l]_{\phi}X' \ar@{->}[d]^{f'}\\
Z &Z', \ar[l]^{\varphi}
}
$$
where $X'$ is a normal analytic variety that is $\mathbb{Q}$-factorial over $W$, such that
\begin{itemize}
\item
$\boldsymbol{\rm M}$ descends to $X'$,
\item
$Z'$ is smooth and $\boldsymbol{\rm N}$ descends to $Z'$, 
\item
 the pair $(X',0)$ is klt and there exists a reduced divisor $\Sigma_{X'}$ on $X'$ such that  $(X',\Sigma_{X'})$ is an lc pair and $\Sigma_{X'}\supset {\rm Supp}\,\phi^{-1}_{*}B\cup {\rm Ex}(\phi)$, and
\item
all fibers of $f'$ have the same dimension. 
\end{itemize}
We may write
$$K_{X'}+B'+\boldsymbol{\rm M}_{X'}=\phi^{*}(K_{X}+B+M)+E',$$
where $B'$ is the sum of $\phi^{-1}_{*}B$ and the reduced $\phi$-exceptional divisor on $X'$. 
Note that $B' \in \Lambda$ since $B \in \Lambda$. 
By the argument in Step \ref{step1--prop--gen-can-bundle-formula-klt-Q-div} in the proof of Proposition \ref{prop--gen-can-bundle-formula-klt-Q-div}, after shrinking $S$ around $W$, we obtain a bimeromorphic contraction 
$$\psi \colon X' \dashrightarrow X''$$
over $Z'$ and the induced morphism $f'' \colon X'' \to Z'$ such that if we put $B'':=\psi_{*}B'$ and $M'':=\psi_{*}\boldsymbol{\rm M}_{X'}$, then 
$\psi_{*}E'=f''^{*}L'$
for some $\mathbb{Q}$-divisor $L'$ on $Z'$. 
Picking a $\mathbb{Q}$-Cartier divisor $D$ on $Z$ such that $K_{X}+B+M \sim_{\mathbb{Q}} f^{*}D$, we have 
\begin{equation*}
\begin{split}
K_{X''}+B''+M''=&\psi_{*}\phi^{*}(K_{X}+B+M)+\psi_{*}E'\sim_{\mathbb{Q}}f''^{*}(\varphi^{*}D+L').
\end{split}
\end{equation*}
By construction, the moduli $\mathbb{Q}$-b-divisor associated to $f'' \colon (X'',B''+M'') \to Z'$ coincides with $\boldsymbol{\rm N}$. 

We show that $X''$ is of Fano type over $Z'$ after shrinking $S$ around $W$. 
We pick a klt pair $(X,\Delta)$ such that $-(K_{X}+\Delta)$ is ample over $Z$. 
By the same argument as in the proof of Lemma \ref{lem--Fano-type-bimero-cont}, after shrinking $S$ around $W$, we obtain a sub-klt pair $(X'',\Delta'')$ such that $A'':=-(K_{X''}+\Delta'')$ is ample over $Z$ and for any prime divisor $P''$ on $X''$ that is not exceptional over $X$, we have ${\rm coeff}_{P''}(\Delta'') \geq 0$.  
By construction of $B''$, for any prime divisor $Q''$ on $X''$ that is exceptional over $X$, we have ${\rm coeff}_{Q''}(B'')=1$. 
Thus, we can find $t \in (0,1]$ such that $\bigl(X'',(t\Delta''+(1-t)B'')+(1-t)M''\bigr) \to S$ is a generalized klt pair with the nef part $(1-t)\boldsymbol{\rm M}$. 
By the perturbation of coefficients using the ampleness of $\frac{t}{2}A''$ over $Z$, we obtain a klt pair $(X'',\Gamma'')$ such that
$$K_{X''}+\Gamma'' \sim_{\mathbb{R},\,Z}K_{X''}+(t\Delta''+(1-t)B'')+(1-t)M''+\frac{t}{2}A''.$$
Then
\begin{equation*}
\begin{split}
K_{X''}+\Gamma'' \sim_{\mathbb{R},\,Z}&K_{X''}+(t\Delta''+(1-t)B'')+(1-t)M''+\frac{t}{2}A''\\
=&t(K_{X''}+\Delta''+\frac{1}{2}A'')+(1-t)(K_{X''}+B''+M'')\\
\sim_{\mathbb{R},\,Z}&-\frac{t}{2}A''+(1-t)f''^{*}(\varphi^{*}D+L')\sim_{\mathbb{R},\,Z'}-\frac{t}{2}A''.
\end{split}
\end{equation*}
Therefore, $-(K_{X''}+\Gamma'')$ is ample over $Z'$. 
Hence, $X''$ is of Fano type over $Z'$. 

Replacing $f \colon (X,B+M) \to Z$ by $f'' \colon (X'',B''+M'') \to Z'$, we may assume that $Z$ is smooth and $\boldsymbol{\rm N}$ descends to $Z$. 
Then we can apply \cite[Proof of Lemma 5.4]{filipazzi-moraga}, and we see that Theorem \ref{thm--can-bundle-formula-coeff-inherited} holds. 
\end{proof}

\begin{prop}[cf.~{\cite[Proposition 3.7]{chhx-compl}}]\label{prop--compl-extend-fib}
Let $p$ and $n$ be positive integers such that $p$ divides $n$. 
Let $\pi \colon (X,B+M) \to S$ be a generalized lc pair with the nef part $\boldsymbol{\rm M}$ such that $\pi \colon X \to S$ is a contraction, and let $f\colon (X,B+M) \to Z$ be a generalized lc-trivial fibration over $S$. 
Let $(Z, G +N) \to S$ be a generalized lc pair with the nef part $\boldsymbol{\rm N}$ that is defined by using the generalized canonical bundle formula. 
We fix a point $s \in S$. 
Suppose that 
\begin{itemize}
\item
$B$ is a $\mathbb{Q}$-divisor on $X$ and both $p\boldsymbol{\rm M}$ and $p\boldsymbol{\rm N}$ are b-Cartier, 
\item
$p(K_{X}+B+M) \sim pf^{*}(K_{Z}+G+N)$, and 
\item
there exists a diagram 
$$
\xymatrix{
X' \ar@{->}[d]_{f'} \ar[r]^{\tau_{X}}&X\ar@{->}[d]^{f}\\
Z'\ar[r]_{\tau}&Z,
}
$$
where $\tau$ and $\tau_{X}$ are bimeromorphic morphisms, such that 
\begin{itemize}
\item
if we define $G'$ by $K_{Z'}+G'+\boldsymbol{\rm N}_{Z'}=\tau^{*}(K_{Z}+G+N)$ then $G'$ is effective, and  
\item
$f'(\tau^{-1}_{X*}P)$ is a prime divisor on $Z'$ for any component $P$ of $B$ that is vertical over $Z$. 
\end{itemize}
\end{itemize}
If $(Z',G'+\boldsymbol{\rm N}_{Z'})$ has an $n$-complement over $s \in S$, then so does $(X,B+M) \to S$. 
\end{prop}

\begin{proof}
We may assume ${\rm dim}\,Z>0$ because otherwise the statement is obvious. 
Then the argument in \cite[Proof of Proposition 3.7]{chhx-compl} works with no changes. 
\end{proof}

Now we can prove Theorem \ref{thm--generalized-compl-main-2}. 

\begin{proof}[Proof of Theorem \ref{thm--generalized-compl-main-2}]
We can apply \cite[Proof of Theorem 8.1]{chhx-compl} to our situation. 
Note that we may use \cite[Theorem 5.1, Theorem 6.1, Proposition 6.6]{chhx-compl} because these results only deal with the projective case, i.e., the case where ${\rm dim}\,Z=0$. 
We may also use the relative MMP and the relative base point free theorem for fibrations of Fano type (\cite{fujino-analytic-bchm}) throughout the proof. 
By using Corollary \ref{cor--compl-extend-1}, we can prove the complex analytic analog of \cite[Theorem 7.1]{chhx-compl}. 
For the proof of \cite[Theorem 8.1]{chhx-compl} in the complex analytic setting, we may use Theorem \ref{thm--can-bundle-formula-coeff-inherited}, Proposition \ref{prop--compl-extend-fib}, and the complex analytic analogs of \cite[Lemma 2.16, Lemma 2.18, Lemma 3.3]{chhx-compl}. 
In this way, we see that Theorem \ref{thm--generalized-compl-main-2} follows. 
\end{proof}

Finally, we prove Corollary \ref{cor--generalized-compl-main-1-intro} and Corollary \ref{cor--generalized-compl-main-2-intro} and introduce a related theorem.  

\begin{proof}[Proof of Corollary \ref{cor--generalized-compl-main-1-intro}]
This corollary is a special case of Theorem \ref{thm--generalized-compl-main-1} where $\boldsymbol{\rm M}=0$ and $X=Z$. 
We note that if $x\in X$ is a non-klt center of $(X,B)$, any $n$-complement $(X,B^{+})$ of $(X,B)$ with $B^{+}\geq B$ satisfies $B^{+}=B$ on a neighborhood of $x$.  
\end{proof}

\begin{proof}[Proof of Corollary \ref{cor--generalized-compl-main-2-intro}]
This corollary is a special case of Theorem \ref{thm--generalized-compl-main-2} where $\boldsymbol{\rm M}=0$ and $X=Z$. 
\end{proof}

\begin{cor}[cf.~{\cite[Conjecture 0.2]{fujino-indices}}]\label{cor--compl-fujinoconj}
For each positive integer $d$, there exists $n \in \mathbb{Z}_{>0}$, depending only on $d$, satisfying the following. 
Let $(X,B)$ be a $d$-dimensional lc pair such that the coefficients of $B$ belong to the hyperstandard set $\Phi(\{0,1\})$ and $(X,\Delta)$ is klt for some $\mathbb{R}$-divisor $\Delta$ on $X$. 
Then, for any non-klt center $x \in X$ of $(X,B)$, the Cartier index of $K_{X}+B$ around $x$ divides $n$. 
\end{cor}

\begin{proof}
This corollary is a special case of Corollary \ref{cor--generalized-compl-main-1-intro} where $\Lambda=\Phi(\{0,1\})$. 
\end{proof}



\begin{thebibliography}{BCHM10}



\bibitem[Am99]{ambro-phdthesis} F.~Ambro, The adjunction conjecture and its applications, PhD thesis, The Johns Hopkins University, 1999.



\bibitem[Am04]{ambro1}F.~Ambro, Shokurov's boundary property, J. Differential Geom. {\textbf{67}} (2004), no.~2, 229--255. 

\bibitem[Am05]{ambro2}F.~Ambro, The moduli $b$-divisor of an lc trivial fibration, Compos. Math. {\textbf{141}} (2005), no.~2, 385--403. 




\bibitem[ABB+23]{abb+23} K.~Ascher, D.~Bejleri, H.~Blum, K.~DeVleming, G.~Inchiostro, Y.~Liu, X.~Wang, Moduli of boundary polarized Calabi-Yau pairs, preprint (2023), arXiv:2307.06522v1.






\bibitem[Bir19]{birkar-compl}
C.~Birkar, 
Anti-pluricanonical systems on Fano varieties, 
Ann. of Math. {\textbf{190}} (2019), no.~2, 345--463. 

\bibitem[Bir21]{birkar-bab}
C.~Birkar, Singularities of linear systems and boundedness of Fano varieties, Ann. of Math. {\textbf{193}} (2021), no.~2, 347--405. 

\bibitem[Bir24]{birkar-connect}
C.~Birkar, 
On connectedness of non-klt loci of singularities of pairs, 
J. Differential Geom. {\textbf{126}} (2024), no.~2, 431--474.






\bibitem[BFMT25]{bfmt} B.~Bakker, S.~Filipazzi, M.~Mauri, J.~Tsimerman, Baily--Borel compactifications of period images and the b-semiampleness conjecture, preprint (2025), arXiv:2508.19215v2. 




\bibitem[BiMi97]{resolution-1} E.~Bierstone, P.~D.~Milman, Canonical desingularization in characteristic zero by blowing up the maximum strata of a local invariant, Invent. Math. {\textbf{128}} (1997), no.~2, 207--302. 



\bibitem[BiZh16]{bz} C.~Birkar, D.~Q.~Zhang, Effectivity of Iitaka fibrations and pluricanonical systems of polarized pairs, Publ. Math. Inst. Hautes \'Etudes Sci. {\textbf{123}} (2016), 283--331. 






\bibitem[GHHX25]{chhx-compl} G.~Chen, J.~Han, Y.~He, L.~Xie, Boundedness of complements for generalized pairs, Proc. Lond. Math. Soc. (3) {\textbf{130}} (2025), no.~5, Paper No.~e70049, 48 pp. 


\bibitem[CHLX23]{chlx} G.~Chen, J.~Han, J.~Liu, L.~Xie,  Minimal model program for algebraically integrable foliations and generalized pairs, preprint (2023), arXiv:2309.15823v3. 


\bibitem[CHX24]{chx-compl} G.~Chen, J.~Han, Q.~Xue, Boundedness of complements for log Calabi-Yau threefolds, Peking Math. J. {\textbf{7}} (2024), no.~1, 1--33.








\bibitem[EH24]{eh-analytic-mmp} M.~Enokizono, K.~Hashizume, Minimal model program for log canonical pairs on complex analytic spaces, preprint (2024), arXiv:2404.05126v2. 

\bibitem[EH25]{eh-semistablereduction} M.~Enokizono, K.~Hashizume, Semistable reduction for complex analytic spaces, Trans. Amer. Math. Soc. {\textbf{378}} (2025), no.~11, 7667--7688.

\bibitem[EH26]{eh-analytic-mmp-2} M.~Enokizono, K.~Hashizume, On termination of minimal model program for log canonical pairs on complex analytic spaces, J. Lond. Math. Soc. (2) {\textbf{113}} (2026), no.~1, Paper No. e70409, 25 pp. 

\bibitem[Fi20]{filipazzi-gen-can-bundle-formula} S.~Filipazzi, On a generalized canonical bundle 
formula and generalized adjunction, Ann. Sc. Norm. Super. Pisa Cl. Sci. (5) \textbf{21} (2020), 1187--1221. 

\bibitem[FiMo20]{filipazzi-moraga} S.~Filipazzi, J.~Moraga, Strong $(\delta,n)$-complements for semi-stable morphisms, Doc. Math. {\textbf{25}} (2020), 1953--1996.

\bibitem[FiSv23]{filipazzi-svaldi} S.~Filipazzi, R.~Svaldi, On the connectedness principle and dual complexes for generalized pairs, Forum Math. Sigma {\textbf{11}} (2023), e33.  




\bibitem[Fu01]{fujino-indices} O.~Fujino, The indices of log canonical singularities, Amer. J. Math. 123 (2001), no.~2, 229--253.
























\bibitem[Fu22]{fujino-analytic-bchm}
O.~Fujino, Minimal model program for projective morphisms between complex analytic spaces, preprint (2022), arXiv:2201.11315v1. 


\bibitem[Fu24a]{fujino-analytic-conethm}
O.~Fujino, Cone and contraction theorem for projective morphisms between complex analytic spaces, MSJ Mem. \textbf{42}, Mathematical Society of Japan, Tokyo, 2024. 

\bibitem[Fu24b]{fujino-analytic-inv-adjunction}
O. Fujino, Log canonical inversion of adjunction, Proc. Japan Acad. Ser. A Math. Sci. {\textbf{100}} (2024), no.~2, 7--11. 

\bibitem[Fu24c]{fujino-analytic-lcabundance}
O.~Fujino, On finiteness of relative log pluricanonical representations, to appear in Algebraic Geom.








\bibitem[FG14]{fg-lctrivial}O.~Fujino, Y.~Gongyo, On the moduli b-divisors of lc-trivial fibrations, Ann. Inst. Fourier {\textbf{64}} (2014), no.~4, 1721--1735. 





\bibitem[FuHa23]{fujino-hashizume-adjunction}O.~Fujino, K.~Hashizume, Adjunction and inversion of adjunction, Nagoya Math. J. {\textbf{249}} (2023), 119--147. 









\bibitem[Go11]{gongyo1}Y.~Gongyo, On the minimal model theory for dlt pairs of numerical log Kodaira dimension zero, Math. Res. Lett. {\textbf{18}} (2011), no.~5, 991--1000.













\bibitem[HaPa24]{haconpaun}C.~D.~Hacon, M.~P\u{a}un, On the canonical bundle formula and adjunction for generalized K\"{a}hler pairs, preprint (2024), arXiv:2404.12007v1. 

\bibitem[HaXi24]{hacon-xie-acc-lct} C.~D.~Hacon, L.~Xie, ACC for generalized log canonical thresholds for complex analytic spaces, Sci. China Math. {\textbf{69}} (2026), no.~7, 1759--1766.








\bibitem[HLL25]{hanliuluo-acc-mld-3fold} J.~Han, J.~Liu, Y.~Luo, ACC for minimal log discrepancies of terminal threefolds, Adv. Math. {\textbf{480}} (2025), part A, Paper No.~110457, 66 pp. 

\bibitem[HLS24]{hanliushokurov-acc-mld} J.~Han, J.~Liu, V.~V.~Shokurov, ACC for Minimal Log Discrepancies of Exceptional Singularities, Peking Math. J. {\textbf{9}} (2026), no.~2, 225--257.






\bibitem[Has20]{has-class} K.~Hashizume, A class of singularity of arbitrary pairs and log canonicalizations, Asian J. Math. {\textbf{24}} (2020), no. 2, 207--238. 




\bibitem[Has22]{has-iitakafibration}
K.~Hashizume, 
Iitaka fibrations for dlt pairs polarized by a nef and log big divisor, Forum Math. Sigma {\textbf{10}} (2022), e85.  

\bibitem[Has24]{has-finite} K.~Hashizume, Finiteness of log Abundant log Canonical Pairs in log Minimal Model Program with Scaling, Michigan Math. J. {\textbf{74}} (2024), no.~5, 1053--1108. 

\bibitem[Hir75]{hironaka-flattening} H.~Hironaka, Flattening theorem in complex-analytic geometry, Amer. J. Math. {\textbf{97}} (1975), 503--547.













\bibitem[Kar99]{karu-phdthesis} K.~Karu, Semistable reduction in characteristic zero, PhD thesis, Boston University, 1999. 


\bibitem[Kaw98]{kawamata-subadjunction-II}Y.~Kawamata, Subadjunction of log canonical divisors II, Amer. J. Math. {\textbf{120}} (1998), no.~5, 893--899.




















\bibitem[KoMo98]{kollar-mori} J.~Koll\'ar, S.~Mori, {\em{Birational geometry of algebraic varieties}}. With the collaboration of C.~H.~Clemens and A.~Corti. Translated from the 1998 Japanese original. Cambridge Tracts in Mathematics {\textbf{134}}. Cambridge University Press, Cambridge, 1998.





\bibitem[LXZ22]{liuxuzhuang-Kstability} Y.~Liu, C.~Xu, Z.~Zhuang, Finite generation for valuations computing stability thresholds and applications to K-stability, Ann. of Math. (2) {\textbf{196}} (2022), no.~2, 507--566. 













\bibitem[N04]{nakayama}N.~Nakayama, {\em Zariski-decomposition and abundance}, MSJ Mem. {\textbf{14}}, Mathematical Society of Japan, Tokyo, 2004. 







\bibitem[S93]{shokurov-flip}V.~V.~Shokurov, 3-fold log flips, Appendix by Yujiro Kawamata: The minimal discrepancy coefficients of terminal singularities in dimension three, Russ. Acad. Sci., Izv., Math. 40 (1993), no.~1, 95--202.

\bibitem[S20]{shokurov-compl}V.~V.~Shokurov, Existence and boundedness of $n$-complements, preprint (2020), arXiv:2012.06495v1.





\end{thebibliography}
\end{document}